\documentclass[11pt]{article}
\usepackage[utf8]{inputenc}

\usepackage[title]{appendix}
\usepackage{epsfig,epsf,fancybox}
\usepackage{amsmath}
\usepackage{mathrsfs}
\usepackage{amsrefs}
\usepackage{amssymb}
\usepackage{graphicx}
\usepackage{color}
\usepackage{multirow}
\usepackage{paralist}
\usepackage{verbatim}
\usepackage{galois}
\usepackage{algorithm}
\usepackage{algorithmic}
\usepackage{boxedminipage}
\usepackage{booktabs}
\usepackage{accents}
\usepackage{stmaryrd}
\usepackage{subfig}
\usepackage{epstopdf}
\usepackage{amsthm}
\usepackage{cases}  
\usepackage[top=1in,bottom=1.2in,left=1in,right=1in,xetex]{geometry}
\usepackage[pdfborder={0 0 0},colorlinks=true,linkcolor=blue,CJKbookmarks=true]{hyperref}
\usepackage{enumerate}

\newcommand\blfootnote[1]{%
	\begingroup
	\renewcommand\thefootnote{}\footnote{#1}%
	\addtocounter{footnote}{-1}%
	\endgroup
}

\numberwithin{equation}{section}

\allowdisplaybreaks[4]

\newtheorem{theorem}{Theorem}[section]
\newtheorem{lemma}[theorem]{Lemma}
\newtheorem{assumption}[theorem]{Assumption}
\newtheorem{proposition}[theorem]{Proposition}
\newtheorem{definition}[theorem]{Definition}
\newtheorem{remark}[theorem]{Remark}
\DeclareMathOperator{\esssup}{ess~sup}

\title{Multidimensional Backward Stochastic Differential Equations with Rough Drifts\footnotemark[1]}
\author{Jiahao Liang\footnotemark[2] \ and Shanjian Tang\footnotemark[3]} 

\begin{document}

\maketitle

\pagenumbering{arabic}

\begin{abstract}
In this paper, we study a multidimensional backward stochastic differential equation (BSDE) with an additional rough drift (rough BSDE), 
and give the existence and uniqueness of the adapted solution, either when the terminal value and the geometric rough path are small, 
or when each component of the rough drift only depends on the corresponding component of the first unknown variable 
(but dropped is the one-dimensional assumption of Diehl and Friz [Ann. Probab. 40 (2012), 1715-1758]). 
We also introduce a new notion of the $p$-rough stochastic integral for $p \in \left[2, 3\right)$, 
and  then succeed in giving---through a fixed-point argument---a general existence and uniqueness result on a multidimensional rough BSDE with a general square-integrable terminal value, 
allowing the rough drift to be random and time-varying but having to be linear; 
furthermore, we connect it to a system of rough partial differential equations. 
\end{abstract}

\renewcommand{\thefootnote}{\fnsymbol{footnote}}

\blfootnote{\textit{MSC2020 subject classifications}: 60L20, 60H10, 60L50. }
\footnotetext[1]{This work was partially supported by National Natural Science Foundation of China (Grants No. 11631004 and No. 12031009),  Key Laboratory of Mathematics for Nonlinear Sciences (Ministry of Education), and Shanghai Key Laboratory for Contemporary Applied Mathematics, Fudan University, Shanghai 200433, China.}
\footnotetext[2]{School of Mathematical Sciences, Fudan University, Shanghai 200433, China. E-mail: jhliang20@fudan.edu.cn.}
\footnotetext[3]{Institute of Mathematical Finance and Department of Finance and Control Sciences, School of Mathematical Sciences, Fudan University, Shanghai 200433, P. R. China. Email: sjtang@fudan.edu.cn.}

\section{Introduction} \label{S1}
The rough integral $\int_{0}^{\cdot} Y d \mathbf{X}$ and rough differential equation (RDE) 
\begin{equation*}
	Y_{t}=Y_{0}+\int_{0}^{t} F\left(Y\right) d \mathbf{X}, \quad t \in \left[0, T\right],
\end{equation*}
for a two-step $\alpha$-Hölder rough path $\mathbf{X}=\left(X, \mathbb{X}\right)$ 
with $\alpha \in \left(\frac{1}{3}, \frac{1}{2}\right]$ and a path $Y$ 
controlled by $X$ (see \cites{RPT1, RPT2, SL1}), extend the work of Young \cite{YI} 
and have received numerous attentions. 
Rough integrals of controlled rough paths have been generalized 
to arbitrary $\alpha \in \left(0, 1\right]$ (see, e.g., \cites{RI1, RI2, RI3}). 
The well-posedness of RDEs has been discussed in various settings 
(see, e.g., \cites{RDE1, RPB2, RDE2, RDE3, RPJ}). 
Moreover, the rough path theory provides a fresh perspective on Itô's 
stochastic calculus theory. 
Rough integrals against an enhanced Brownian motion 
(and also continuous semimartingales enhanced to rough paths) 
coincide with stochastic integrals whenever both are well-defined 
(see, e.g., \cites{RPB1, RPB2}). 
It has also been extended to càdlàg semimartingales (see \cites{SM1, SM2}). 
Hence, SDEs can be solved pathwisely as RDEs driven by an enhanced Brownian motion. 
A pathwise interpretation of SPDEs also gives rise to numerous studies of 
rough partial differential equations (rough PDEs) 
(see, e.g., \cites{RPDE1, RPDE2, RPDE3, RPDE4, RPDE5}).\\
\indent
Differential equations driven simultaneously by a Brownian motion 
$W$ and by a rough path $\mathbf{X}$ have also been studied. 
Diehl and Friz \cite{RBSDE} first study 
backward stochastic differential equations with rough drifts (rough BSDEs) 
\begin{equation} \label{inrbsde1}
	Y_{t}=\xi+\int_{t}^{T} f_{r}\left(Y_{r}, Z_{r}\right) d r+\int_{t}^{T} g\left(S_{r}, Y_{r}\right) d \mathbf{X}_{r}-\int_{t}^{T} Z_{r} d W_{r}, \quad t \in \left[0, T\right],
\end{equation}
for given an Itô process $S$, a deterministic vector field $g$ 
and a geometric rough path $\mathbf{X}$, 
but they restrict themselves within the case  that the first unknown process $Y$ is scalar valued. 
With a flow transformation to separate stochastic integrals from rough integrals, they  
first solve RDEs and the reduced (from the flow transformation) quadratic BSDEs, respectively; then they give the well-posedness of rough BSDE \eqref{inrbsde1} via the continuity of both solution maps of RDEs and quadratic BSDEs. 
Motivated by robustness problems in stochastic filtering, 
Crisan et al. \cite{RSDE1} first study stochastic differential equations with rough drifts (rough SDEs) with deterministic and time-invariant coefficients
\begin{equation} \label{rsde1}
	S_{t}=S_{0}+\int_{0}^{t} b\left(S_{r}\right) d r+\int_{0}^{t} F\left(S_{r}\right) d \mathbf{X}_{r}+\int_{0}^{t} \sigma\left(S_{r}\right) d W_{r}, \quad t \in \left[0, T\right], 
\end{equation}
and give the well-posedness of rough SDE \eqref{rsde1} via the technique of the flow transformation. 
Diehl et al. \cite{RSDE2} also obtain the well-posedness of rough SDE \eqref{rsde1} 
by constructing a joint rough lift of $\left(W\left(\omega\right), \mathbf{X}\right)$ 
and solving rough SDE \eqref{rsde1} pathwisely as an RDE. 
Friz et al. \cite{RSDE3} introduce stochastic controlled rough paths and establish rough 
stochastic integrals of them against two-step $\alpha$-Hölder rough paths 
using an extension of the stochastic sewing lemma of L\^{e} \cite{SSL1}*{Theorem 2.1}. 
Moreover, they introduce stochastic controlled vector fields and give by a fixed-point 
argument the well-posedness of rough SDEs with random and time-varying coefficients 
\begin{equation*} 
	S_{t}=S_{0}+\int_{0}^{t} b_{r}\left(S_{r}\right) d r+\int_{0}^{t} F_{r}\left(S_{r}\right) d \mathbf{X}_{r}+\int_{0}^{t} \sigma_{r}\left(S_{r}\right) d W_{r}, \quad t \in \left[0, T\right].
\end{equation*}
\indent
In this paper, we consider multidimensional rough BSDEs 
\begin{equation} \label{inrbsde2}
	Y_{t}=\xi+\int_{t}^{T} f_{r}\left(Y_{r}, Z_{r}\right) d r+\int_{t}^{T} g\left(Y_{r}\right) d \mathbf{X}_{r}-\int_{t}^{T} Z_{r} d W_{r}, \quad t \in \left[0, T\right],
\end{equation}
and thus drop the one-dimensional restriction of Diehl and Friz \cite{RBSDE} on the first unknown process $Y$. 
Here, $g$ is a given deterministic vector field 
and $\mathbf{X}$ is a given geometric rough path. 
A solution to rough BSDE \eqref{inrbsde2} is defined as the limit of smooth approximations 
(see Definition \ref{solud2}). 
With the help of  the well-posedness results on multidimensional quadratic BSDEs of both Tevzadze \cite{QBSDE} 
and Fan et al. \cite{DQBSDE1}, we give via the flow transformation method of \cite{RBSDE} 
the existence and uniqueness of the solution to the rough BSDE~\eqref{inrbsde2}, 
either when $\xi$ and $\mathbf{X}$ are sufficiently small (see Theorem \ref{solueu2}) 
or when each component of $g$ only depends on the corresponding component of the first unknown variable $y$ (see Theorem \ref{solueu3}).\\
\indent
In the particular case of a linear (but possibly both {\it random and time-varying}) rough drift, we also solve with the  fixed-point principle the following rough BSDE 
\begin{equation} \label{inrbsde3}
	Y_{t}=\xi+\int_{t}^{T} f_{r}\left(Y_{r}, Z_{r}\right) d r+\int_{t}^{T} \left(G_{r}Y_{r}+H_{r}\right) d \mathbf{X}_{r}-\int_{t}^{T} Z_{r} d W_{r}, \quad t \in \left[0, T\right],
\end{equation}
when the terminal value $\xi$ is just square-integrable and the rough integrator $\mathbf{X}$ is a conventional rough path. 
A crucial  difficulty we are now facing with  is the general lack of a Hölder continuity  for a solution $\left(Y, Z\right)$ to a BSDE, of  the following map:
\begin{equation*}
 t \mapsto \int_{0}^{t} Z_{r} d W_{r} \quad \in L^{2}\left(\Omega\right), \quad t\in [0,T],
\end{equation*}
which leads to the failure of using as in Friz et al. \cite{RSDE3} the fixed-point argument 
on the  space of stochastic controlled rough paths of Hölder continuity. 
To overcome this difficulty, we introduce a novel notion of the {\it rough stochastic integral 
in $p$-variation scale}, or simply {\it $p$-rough stochastic integral}, 
for $p \in \left[2, 3\right)$. 
Since the integral $\int_{0}^{\cdot} Z_{r} d W_{r}$ is a square-integrable martingale, 
it naturally has a finite $p$-variation as a path from 
$\left[0, T\right]$ to $L^{2}\left(\Omega\right)$ (see Lemma \ref{mtg}). 
Another crucial difficulty is the deduction with the general stochastic sewing lemma 
of L\^{e} \cite{SSL2}*{Theorem 3.1} the boundedness of the following rough stochastic integral map 
\begin{equation} \label{inrsim}
	\left(Y, Y^{\prime}\right) \mapsto \left(\int_{0}^{\cdot} Y d \mathbf{X}, Y\right)
\end{equation}
in the definition of the $p$-rough stochastic integral of a square-integrable 
stochastic controlled rough path $\left(Y, Y^{\prime}\right)$ against $\mathbf{X}$. 
To overcome this difficulty, we appeal to the decomposition $Y=Y^{M}+Y^{J}$ where $Y^{M}$ is a square-integrable martingale 
with $Y_{0}^{M}=0$ and $\left(Y^{J}, Y^{\prime}\right)$ is a controlled 
rough path taking values in $L^{2}\left(\Omega\right)$ (see Proposition \ref{deco}). 
Then, with the help of  the stochastic sewing lemma and (deterministic) sewing lemma,  
we obtain $\int_{0}^{\cdot} Y^{M} d \mathbf{X}$ and 
$\int_{0}^{\cdot} Y^{J} d \mathbf{X}$, respectively. 
Due to the norm equivalence \eqref{enorme} induced by this decomposition, we 
get the boundedness of the rough stochastic integral map \eqref{inrsim}. 
Using a fixed-point argument, we obtain the following well-posedness result. 
\begin{theorem}
	Let $\xi \in L^{2}\left(\Omega, \mathcal{F}_{T}\right)$, 
	$\mathbf{X}$ be a two-step $p$-rough path for $p \in \left[2, 3\right)$, 
	$\left(G, G^{\prime}\right)$ and $\left(H, H^{\prime}\right)$ 
	be an essentially bounded and a square-integrable stochastic controlled 
	rough path of finite $\left(p, p\right)$-variation, respectively (see Definition \ref{scrp}). 
	Assume that $f$ is progressively measurable and uniformly Lipschitz continuous in $y$ and $z$, 
	and $f\left(0, 0\right)$ is a square-integrable process. 
	Then the rough BSDE \eqref{inrbsde3} has a unique solution $\left(Y, Z\right)$ 
	(see Theorem \ref{solueu}). Moreover, the solution map 
	\begin{equation*}
		\left(\xi, f, G, G^{\prime}, H, H^{\prime}, \mathbf{X}\right) \mapsto \left(Y, Z\right)
	\end{equation*}
	is continuous (see Theorem \ref{solumc}). 
\end{theorem}
The theorem is new even in the one-dimensional case, 
for the terminal value $\xi$ is only required to be square-integrable 
(not necessarily essentially bounded as required in Diehl and Friz \cite{RBSDE}) 
and the rough integrator $\mathbf{X}$ can be non-geometric. \\
\indent
In a Markovian setting, we give an application to systems of rough PDEs 
with coefficients depending on the rough path variable $\mathbf{X}$ 
\begin{equation} \label{inrpde}
	\left\{
	\begin{aligned}
		&du\left(t, x, \mathbf{X}\right)+\frac{1}{2}D^{2}u\left(t, x, \mathbf{X}\right)\sigma_{t}^{2}\left(x, \mathbf{X}\right)dt+Du\left(t, x, \mathbf{X}\right)b_{t}\left(x, \mathbf{X}\right)dt\\
		&\quad+f_{t}\left(x, u\left(t, x, \mathbf{X}\right), Du\left(t, x, \mathbf{X}\right)\sigma_{t}\left(x, \mathbf{X}\right), \mathbf{X}\right)dt\\
		&\quad+\left[G_{t}\left(\mathbf{X}\right)u\left(t, x, \mathbf{X}\right)+h_{t}\left(x, \mathbf{X}\right)+H_{t}\left(\mathbf{X}\right)\right]d\mathbf{X}_{t}=0,\\
		&\quad \left(t, x, \mathbf{X}\right) \in \left[0, T\right) \times \mathbb{R}^{d} \times \mathscr{C}_{g}^{0, p\text{-}var};\\
		&u\left(T, x, \mathbf{X}\right)=l\left(x, \mathbf{X}\right), \quad \left(x, \mathbf{X}\right) \in \mathbb{R}^{d} \times \mathscr{C}_{g}^{0, p\text{-}var}.
	\end{aligned}
	\right.
\end{equation}
Here, $\mathscr{C}_{g}^{0, p\text{-}var}$ is the space of all two-step geometric $p$-rough paths. 
We define a solution $u$ to \eqref{inrpde} as a continuous map from $\mathscr{C}_{g}^{0, p\text{-}var}$ to $BUC$, 
such that $u\left(\cdot, \cdot, \mathbf{X}\right)$ is a viscosity solution to 
the corresponding system of PDEs for the canonical rough lift $\mathbf{X}$ 
of any smooth path $X$ (see Definition \ref{rpdesolu}), 
where $BUC$ is the space of bounded uniformly continuous functions with the uniform topology. 
Under certain conditions, we obtain the existence and uniqueness of the solution to equation \eqref{inrpde}, 
and give a stochastic representation of the solution (see Theorem \ref{rpdesolueu}). 
To our best knowledge, this is a first study to multidimensional BSDEs 
with conventional rough drivers.\\
\indent
The paper is organized as follows. 
Section \ref{S2} contains preliminary notations and results. 
In Section \ref{S3}, we study the solvability of BSDEs with a nonlinear rough drift \eqref{inrbsde2}. 
Section \ref{S4} contains properties of integration and composition of stochastic controlled rough paths. 
In Section \ref{S5}, we show the well-posedness of BSDEs with a linear rough drift \eqref{inrbsde3}. 
An application is given to systems of rough PDEs via a Feynman-Kac formula in Section \ref{S6}. 
Appendix \ref{SA} contains some preliminary well-posedness results on BSDEs. 

\section{Preliminaries} \label{S2}
Fix a finite time horizon $T > 0$. 
Denote by $\Delta$ the simplex $\left\{\left(s, t\right): 0 \leq s \leq t \leq T\right\}$, 
and by $\Delta_{2}$ the simplex $\left\{\left(s, u, t\right): 0 \leq s \leq u \leq t \leq T\right\}$. 
For $\left(s, t\right) \in \Delta_{2}$, denote by $\mathcal{P}\left[s, t\right]$ 
the set of all partitions of the interval $\left[s, t\right]$, and by $\left|\pi\right|$ 
the mesh size of a partition $\pi \in \mathcal{P}\left[s, t\right]$. 
Let $V$ and $\bar{V}$ be two Euclidean spaces, and $\left|\cdot\right|$ be their norms. 
Denote by $\mathcal{B}\left(V\right)$ the Borel $\sigma$-field on $V$. 
The Banach space of linear maps from $V$ to $\bar{V}$, endowed with the induced norm 
$\left|F\right|:=\sup_{x \in V, \left|x\right| \leq 1}\left|F\left(x\right)\right|$, 
is denoted by $\mathcal{L}\left(V, \bar{V}\right)$. 
Denote by $V \otimes \bar{V}$ the tensor product of $V$ and $\bar{V}$. 
Then any $F \in \mathcal{L}\left(V, \mathcal{L}\left(V, \bar{V}\right)\right)$ 
can be regarded as an element in $\mathcal{L}\left(V \otimes V, \bar{V}\right)$ 
by defining $F\left(\sum x_{i} \otimes y_{i}\right):=\sum \left[F\left(x_{i}\right)\right]\left(y_{i}\right)$ 
for every $x_{i}, y_{i} \in V$. 
For a function $F: V \rightarrow \bar{V}$, define 
$\left|F\right|_{\infty}:=\sup_{x \in V}\left|F\left(x\right)\right|$. 
For $\gamma > 0$, denote by $\lfloor\gamma\rfloor$ 
the largest integer strictly smaller than $\gamma$ so that 
$\gamma=\lfloor\gamma\rfloor+\left\{\gamma\right\}$ with 
$\lfloor\gamma\rfloor \in \mathbb{N}$ and $\left\{\gamma\right\} \in \left(0, 1\right]$. 
Denote by $Lip^{\gamma}\left(V, \bar{V}\right)$ the space of $\lfloor\gamma\rfloor$ 
times continuously differentiable functions $F: V \rightarrow \bar{V}$ such that 
\begin{equation*}
	\left|F\right|_{\gamma}:=\sum _{i=0}^{\lfloor\gamma\rfloor}\left|D^{i}F\right|_{\infty}+\sup_{x_{1} \neq x_{2}} \frac{\left|D^{\lfloor\gamma\rfloor}F\left(x_{1}\right)-D^{\lfloor\gamma\rfloor}F\left(x_{2}\right)\right|}{|x_{1}-x_{2}|^{\left\{\gamma\right\}}}<\infty.
\end{equation*}
\indent
Throughout the paper, we fix $p \in \left[2, 3\right)$. 
By $a \lesssim b$, we mean that $a \leq Cb$ for some positive constant $C$ 
only depending on some universal parameters, which will be indicated in the subscripts of the relation like $a \lesssim_{M} b$ whenever the dependence is necessarily stressed. 
\subsection{$q$-variation paths} \label{S1.1}
For a path $X: \left[0, T\right] \rightarrow V$, define the increment of $X$ by 
\begin{equation*}
	\delta X: \Delta \rightarrow V, \quad \delta X_{s, t}:=X_{t}-X_{s}, \quad \forall \left(s, t\right) \in \Delta.
\end{equation*}
Similarly, for a function $A: \Delta \rightarrow V$, 
define the increment of $A$ by 
\begin{equation*}
	\delta A: \Delta_{2} \rightarrow V, \quad \delta A_{s, u, t}:=A_{s, t}-A_{s, u}-A_{u, t}, \quad \forall \left(s, u, t\right) \in \Delta_{2}.
\end{equation*}
For $q \in \left[1, \infty\right)$ and $\left(s, t\right) \in \Delta$, 
we define the $q$-variation of $A$ on the interval $\left[s, t\right]$ by 
\begin{equation*}
	\left|A\right|_{q\text{-}var;\left[s, t\right]}:=\sup_{\pi \in \mathcal{P}\left[s, t\right]}\left(\sum_{\left[u, v\right] \in \pi}\left|A_{u, v}\right|^{q}\right)^{\frac{1}{q}}.
\end{equation*}
Write $\left|A\right|_{q\text{-}var}:=\left|A\right|_{q\text{-}var;\left[0, T\right]}$. 
A path $X$ is said to be of finite $q$-variation if $\left|\delta X\right|_{q\text{-}var}$ is finite. 
The space of continuous paths of finite $q$-variation with values in $V$ 
is denoted by $C^{q\text{-}var}\left([0, T], V\right)=:C^{q\text{-}var}$. 
Any $\frac{1}{q}$-Hölder continuous path is of finite $q$-variation. 
On the other hand, there exist continuous paths which are of finite $q$-variation but not $\frac{1}{q}$-Hölder continuous (e.g. $t \mapsto t^{\beta}$ with $\beta < \frac{1}{q}$). 
See \cite{RPB2}*{Section 5} for properties of continuous paths of finite $q$-variation. 
\begin{definition} \label{rp}
	We call $\mathbf{X}=\left(X, \mathbb{X}\right)$ a two-step $p$-rough path over $\mathbb{R}^{e}$, 
	denoted by $\mathbf{X} \in \mathscr{C}^{p\text{-}var}\left([0, T], \mathbb{R}^{e}\right)$, 
	if the following are satisfied
	\begin{enumerate}[(i)]
		\item $X \in C^{p\text{-}var}\left([0, T], \mathbb{R}^{e}\right)$;
		\item $\mathbb{X}: \Delta \rightarrow \mathbb{R}^{e} \otimes \mathbb{R}^{e}$ 
		is continuous and $\left|\mathbb{X}\right|_{\frac{p}{2}\text{-}var}$ is finite;
		\item $\mathbf{X}$ satisfies Chen's relation 
		\begin{equation} \label{Chen}
			\delta \mathbb{X}_{s, u, t}=\delta X_{s, u} \otimes \delta X_{u, t}, \quad \forall \left(s, u, t\right) \in \Delta_{2}.
		\end{equation}
	\end{enumerate}
\end{definition}
For 
$\mathbf{X}=\left(X, \mathbb{X}\right), \bar{\mathbf{X}}=\left(\bar{X}, \bar{\mathbb{X}}\right) \in \mathscr{C}^{p\text{-}var}$, 
we define 
\begin{equation*}
	\left|\mathbf{X}\right|_{p\text{-}var}:=\left|\delta X\right|_{p\text{-}var}+\left|\mathbb{X}\right|_{\frac{p}{2}\text{-}var}, \quad \rho_{p\text{-}var}\left(\mathbf{X}, \bar{\mathbf{X}}\right):=\left|\delta X-\delta \bar{X}\right|_{p\text{-}var}+\left|\mathbb{X}-\bar{\mathbb{X}}\right|_{\frac{p}{2}\text{-}var};
\end{equation*}
\begin{equation*}
	\quad \left|\mathbf{X}\right|_{\mathscr{C}^{p\text{-}var}}:=\left|X_{0}\right|+\left|\mathbf{X}\right|_{p\text{-}var}, \quad \rho_{\mathscr{C}^{p\text{-}var}}\left(\mathbf{X}, \bar{\mathbf{X}}\right):=\left|X_{0}-\bar{X}_{0}\right|+\rho_{p\text{-}var}\left(\mathbf{X}, \bar{\mathbf{X}}\right). 
\end{equation*}
The metric space 
$\left(\mathscr{C}^{p\text{-}var}, \rho_{\mathscr{C}^{p\text{-}var}}\right)$ 
is complete. 
We will see later that the initial value $X_{0}$ of $X$ 
has no effect on $p$-rough stochastic integrals against $\mathbf{X}$. 
Hence, we always assume without loss of generality that $X_{0}=0$. 
Then $\rho_{p\text{-}var}$ is a true metric on the subspace 
$\left\{\mathbf{X} \in \mathscr{C}^{p\text{-}var}: X_{0}=0\right\}$. 
For a smooth path $X \in C^{\infty}\left([0, T], \mathbb{R}^{e}\right)$, 
we always define 
\begin{equation*}
	\mathbb{X}: \Delta \rightarrow \mathbb{R}^{e} \otimes \mathbb{R}^{e}, \quad \mathbb{X}_{s, t}:=\int_{s}^{t} \delta X_{s, r} \otimes \dot{X}_{r} d r, \quad \forall \left(s, t\right) \in \Delta,
\end{equation*}
and $\mathbf{X}:=\left(X, \mathbb{X}\right)$. 
Then $\mathbf{X} \in \mathscr{C}^{p\text{-}var}\left([0, T], \mathbb{R}^{e}\right)$ and we call 
$\mathbf{X}$ the canonical rough lift of the smooth path $X$. 
\begin{definition} \label{grp}
	We call $\mathbf{X}$ a two-step geometric 
	$p$-rough path over $\mathbb{R}^{e}$, denoted by 
	$\mathbf{X} \in \mathscr{C}_{g}^{0, p\text{-}var}\left([0, T], \mathbb{R}^{e}\right)$, 
	if $\mathbf{X} \in \mathscr{C}^{p\text{-}var}\left([0, T], \mathbb{R}^{e}\right)$ 
	and there exists a sequence of smooth paths 
	$X^{n} \in C^{\infty}\left([0, T], \mathbb{R}^{e}\right), n=1, 2, \cdots$, such that 
	\begin{equation*}
		\lim_{n \rightarrow \infty} \rho_{p\text{-}var}\left(\mathbf{X}^{n}, \mathbf{X}\right)=0.
	\end{equation*}
\end{definition}
\subsection{Stochastic objects} \label{S1.2}
Let $\left(\Omega, \mathcal{F}, \mathbb{P}\right)$ be a complete probability space 
carrying a $d$-dimensional Brownian motion $W$, 
and $\left\{\mathcal{F}_{t}\right\}$ be the augmentation under $\mathbb{P}$ of 
the filtration $\left\{\mathcal{F}_{t}^{W}\right\}$ generated by $W$. 
Denote by $\mathbb{E}_{t}$ the expectation operator conditioned at $\mathcal{F}_{t}$. 
For $m \in \left[1, \infty\right]$ and a subfield $\mathcal{G}$ of $\mathcal{F}$, 
the space of $\mathcal{G}$-measurable $L^{m}$-integrable random variables $\xi$ 
with values in $V$ is denoted by $L^{m}\left(\Omega, \mathcal{G}; V\right)$, 
and its norm is denoted by $\left\|\xi\right\|_{m}$. 
Write $L^{m}\left(\Omega; V\right):=L^{m}\left(\Omega, \mathcal{F}; V\right)$. 
Denote by $\mathscr{M}_{m}\left(V\right)=:\mathscr{M}_{m}$ the space of $L^{m}$-integrable martingales with values in $V$. 
The space of measurable adapted processes $F: \Omega \times \left[0, T\right] \rightarrow V$ such that 
\begin{equation*}
	\left\|F\right\|_{\mathbb{L}^{2}}:=\left(\mathbb{E}\int_{0}^{T}\left|F_{r}\right|^{2} d r\right)^{\frac{1}{2}} < \infty,
\end{equation*}
is denoted by $\mathbb{L}^{2}\left(\left[0, T\right], \Omega; V\right)=:\mathbb{L}^{2}$. 
Denote by $\mathbb{S}^{\infty}\left(\left[0, T\right], \Omega; V\right)=:\mathbb{S}^{\infty}$ the space of essentially bounded continuous adapted processes 
$F: \Omega \times \left[0, T\right] \rightarrow V$, equipped with the norm 
$\left\|F\right\|_{\mathbb{S}^{\infty}}:=\esssup_{\left(\omega, t\right)} \left|F_{r}\right|$. 
Denote by $BMO\left(\left[0, T\right], \Omega; V\right)=:BMO$ the space of processes 
$F \in \mathbb{L}^{2}\left(\left[0, T\right], \Omega; V\right)$ such that 
\begin{equation*}
	\left\|F\right\|_{BMO}:=\mathop{\esssup}\limits_{\left(\omega, t\right)}\left(\mathbb{E}_{t}\int_{t}^{T} \left|F_{r}\right|^{2} d r\right)^{\frac{1}{2}} < \infty. 
\end{equation*}
A two-parameter process $A: \Omega \times \Delta \rightarrow V$ is said to be adapted (to $\left\{\mathcal{F}_{t}\right\}$) 
if $A_{s, t}$ is $\mathcal{F}_{t}$-measurable for every $\left(s, t\right) \in \Delta$. 
Denote by $\mathbb{E}_{\cdot}A$ the two-parameter process:
$\left(s, t\right) \mapsto \mathbb{E}_{s}A_{s, t}, (s,t)\in \Delta$. 
\begin{definition}
	For $m \in \left[1, \infty\right]$ and $q \in \left[1, \infty\right)$, denote by $$C_{2}^{q\text{-}var}L_{m}\left(\left[0, T\right], \Omega; V\right)=:C_{2}^{q\text{-}var}L_{m}$$
	 the space of 
	$\mathcal{F} \times \mathcal{B}\left(\Delta\right)/\mathcal{B}\left(V\right)$-measurable 
	(or simply measurable) adapted two-parameter processes 
	$A: \Omega \times \Delta \rightarrow V$ such that 
	$A \in C\left(\Delta, L^{m}\left(\Omega; V\right)\right)$, i.e. the map 
	$\left(s, t\right) \mapsto A_{s, t}$ is continuous 
	from $\Delta$ to $L^{m}\left(\Omega; V\right)$, and 
	$\left\|A\right\|_{m,q\text{-}var}:=\left\|A\right\|_{m,q\text{-}var;\left[0, T\right]}$ 
	is finite, where 
	\begin{equation*}
		\left\|A\right\|_{m,q\text{-}var;\left[s, t\right]}:=\sup_{\pi \in \mathcal{P}\left[s, t\right]}\left(\sum_{\left[u, v\right] \in \pi}\left\|A_{u, v}\right\|_{m}^{q}\right)^{\frac{1}{q}}, \quad \forall \left(s, t\right) \in \Delta.
	\end{equation*}
\end{definition}
The normed vector space $\left(C_{2}^{q\text{-}var}L_{m}, \left\|\cdot\right\|_{m,q\text{-}var}\right)$ 
is a Banach space and $A \in C_{2}^{q\text{-}var}L_{m}$ 
implies that $A_{t,t}=0$ for every $t \in \left[0, T\right]$. 
Similarly, we denote by $C^{q\text{-}var}L_{m}\left(\left[0, T\right], \Omega; V\right)=: C^{q\text{-}var}L_{m}$ 
the space of measurable adapted processes 
$Y: \Omega \times \left[0, T\right] \rightarrow V$ such that 
$Y \in C\left(\left[0, T\right], L^{m}\left(\Omega ; V\right)\right)$ and 
$\left\|\delta Y\right\|_{m,q\text{-}var}$ is finite. 
Equipped with the norm 
\begin{equation*}
	\left\|Y\right\|_{C^{q\text{-}var}L_{m}}:=\left\|Y_{T}\right\|_{m}+\left\|\delta Y\right\|_{m,q\text{-}var},
\end{equation*}
the vector space $C^{q\text{-}var}L_{m}$ is also a Banach space and we have 
\begin{equation*}
	\sup_{t \in\left[0, T\right]}\left\|Y_{t}\right\|_{m} \leq \left\|Y\right\|_{C^{q\text{-}var}L_{m}}, 
	\quad \forall\, Y\in  C^{q\text{-}var}L_{m}.
\end{equation*}
We call a function $w : \Delta \rightarrow \left[0, \infty\right)$ a control if it is continuous and satisfies the following superadditivity
\begin{equation*}
	w\left(s, u\right)+w\left(u, t\right) \leq w\left(s, t\right), \quad \forall \left(s, u, t\right) \in \Delta_{2}.
\end{equation*}
Analogous to \cite{RPB2}*{Proposition 5.8}, for $A \in C_{2}^{q\text{-}var}L_{m}$,  the map 
\begin{equation*}
	\left(s, t\right) \mapsto \left\|A\right\|_{m,q\text{-}var;\left[s, t\right]}^{q}
\end{equation*}
defines a control.\\ 
\indent
As in many cases, we do not distinguish between a process and its modification. 
We further agree that a process $Y$ is  continuous, i.e. its sample paths  are a.s. 
continuous, if it has a continuous modification. 
Since $\left\{\mathcal{F}_{t}\right\}$ is the augmentation of 
$\left\{\mathcal{F}_{t}^{W}\right\}$, any local martingale is continuous 
(see, e.g., \cite{BMaSC}*{Theorems 1.3.13 and 3.4.15}). 

\section{BSDEs with a nonlinear rough drift} \label{S3}

Consider the following rough BSDE 
\begin{equation} \label{rbsde2}
	Y_{t}=\xi+\int_{t}^{T} f_{r}\left(Y_{r}, Z_{r}\right) d r+\int_{t}^{T} g\left(Y_{r}\right) d \mathbf{X}_{r}-\int_{t}^{T} Z_{r} d W_{r}, \quad t \in \left[0, T\right].
\end{equation}
Here, $\xi$ is an $\mathbb{R}^{k}$-valued random variable, 
$\mathbf{X}=\left(X, \mathbb{X}\right) \in \mathscr{C}_{g}^{0, p\text{-}var}\left([0, T], \mathbb{R}^{e}\right)$ 
and $\left(Y, Z\right)$ is the pair of unknown processes. The random vector field
$f: \Omega \times \left[0, T\right] \times \mathbb{R}^{k} \times \mathbb{R}^{k \times d} \rightarrow \mathbb{R}^{k}$ 
is progressively measurable, and the function 
$g: \mathbb{R}^{k} \rightarrow \mathbb{R}^{k \times e}$ 
is deterministic and time-invariant. 
For $i=1, \cdots, k$, we denote by $y^{i}$ and $z^{i}$ the $i$th component of vector $y \in \mathbb{R}^{k}$ 
and the $i$th row of matrix $z \in \mathbb{R}^{k \times d}$,  respectively, and similarly for other vectors and matrices. 
Inspired by \cite{RBSDE}*{Theorem 3}, we give the following definition. 
\begin{definition} \label{solud2} 
	We call $\left(Y, Z\right) \in \mathbb{S}^{\infty}\left(\left[0, T\right], \Omega; \mathbb{R}^{k}\right) \times BMO\left(\left[0, T\right], \Omega; \mathbb{R}^{k \times d}\right)$ 
	a solution to \eqref{rbsde2} if there exist 
	$X^{n} \in C^{\infty}\left(\left[0, T\right], \mathbb{R}^{e}\right)$ and 
	$\left(Y^{n}, Z^{n}\right) \in \mathbb{S}^{\infty} \times BMO$ for $n=1, 2, \cdots$, 
	such that $\left(Y^{n}, Z^{n}\right)$ is a solution to BSDE 
	\begin{equation} \label{bsdesequ}
		Y_{t}^{n}=\xi+\int_{t}^{T} \left(f_{r}\left(Y_{r}^{n}, Z_{r}^{n}\right)+g\left(Y_{r}^{n}\right)\dot{X}_{r}^{n}\right) d r-\int_{t}^{T} Z_{r}^{n} d W_{r}, \quad t \in \left[0, T\right]
	\end{equation}
	and 
	\begin{equation*}
		\lim_{n \rightarrow \infty}\left[\, \rho_{p\text{-}var}\left(\mathbf{X}^{n}, \mathbf{X}\right)+\left\|Y^{n}-Y\right\|_{\mathbb{S}^{\infty}}+\left\|Z^{n}-Z\right\|_{BMO}\right]=0. 
	\end{equation*}
\end{definition}
In the next two subsections, we give the solvability of rough BSDE \eqref{rbsde2} 
under the two cases: when the terminal value and the geometric rough path are small and when each component of the rough drift  $g(y)$ only depends on the corresponding  component of the first unknown variable $y$. 
\subsection{Solvability under the small terminal value condition} \label{S3.1}
We first introduce the following assumption. 
\begin{assumption} \label{assu2}
	~
	\begin{enumerate}[(i)]
		\item $\xi \in L^{\infty}\left(\Omega, \mathcal{F}_{T}; \mathbb{R}^{k}\right)$;
		\item there exist a constant $L \geq 0$ and measurable adapted processes 
		$\lambda, \mu: \Omega \times \left[0, T\right] \rightarrow \left[0, \infty\right)$ such that 
		$\int_{0}^{T} \left(\lambda_{r}+\mu_{r}^{2}\right) d r \in L^{\infty}\left(\Omega; \mathbb{R}\right)$, 
		and for every $\left(t, y, z\right) \in \left[0, T\right] \times \mathbb{R}^{k} \times \mathbb{R}^{k \times d}$, 
		\begin{equation*}
			\left|f_{t}\left(y, z\right)\right|+\left|\partial_{y}f_{t}\left(y, z\right)\right| \leq \lambda_{t}+L\left(\left|y\right|^{2}+\left|z\right|^{2}\right), \quad \left|\partial_{z}f_{t}\left(y, z\right)\right| \leq \mu_{t}+L\left(\left|y\right|+\left|z\right|\right);
		\end{equation*}
		\item $g \in Lip^{\gamma}\left(\mathbb{R}^{k}, \mathbb{R}^{k \times e}\right)$ 
		for some real number $\gamma > p+2$. 
	\end{enumerate}
\end{assumption}
We now give the following solvability of rough BSDE \eqref{rbsde2} under the small terminal value condition. 
\begin{theorem} \label{solueu2}
	Let Assumption \ref{assu2} hold. Then there exist $\varepsilon > 0$ 
	and a bounded subset $\mathcal{B}$ of $\mathbb{S}^{\infty} \times BMO$, 
	only depending on $T, L$ and $\left|g\right|_{\gamma}$, such that for 
	$\left\|\xi\right\|_{\infty} \vee \left\|\int_{0}^{T} \left(\lambda_{r}+\mu_{r}^{2}\right) d r\right\|_{\infty} \vee \left|\mathbf{X}\right|_{p\text{-}var} \leq \varepsilon$, 
	the rough BSDE \eqref{rbsde2} has a unique solution $\left(Y, Z\right) \in \mathcal{B}$. 
\end{theorem}
To prove Theorem \ref{solueu2}, we need some preliminaries. 
Let $X^{n} \in C^{\infty}\left(\left[0, T\right], \mathbb{R}^{e}\right)$, $n=1, 2, \cdots$ 
be a sequence of smooth paths such that 
\begin{equation} \label{flowcon1}
	\left|\mathbf{X}^{n}\right|_{p\text{-}var} \leq 2\left|\mathbf{X}\right|_{p\text{-}var}, \quad \forall n=1, 2, \cdots
\end{equation}
and 
\begin{equation} \label{flowcon2}
	\lim_{n \rightarrow \infty} \rho_{p\text{-}var}\left(\mathbf{X}^{n}, \mathbf{X}\right)=0.
\end{equation}
Denote by $\phi^{0}: \left[0, T\right] \times \mathbb{R}^{k} \rightarrow \mathbb{R}^{k}$ the solution flow to RDE 
\begin{equation} \label{rdeflow}
	\phi_{t}^{0}\left(y\right)=y+\int_{t}^{T} g\left(\phi_{r}^{0}\left(y\right)\right) d \mathbf{X}_{r}, \quad t \in \left[0, T\right],
\end{equation}
and by $\phi^{n}: \left[0, T\right] \times \mathbb{R}^{k} \rightarrow \mathbb{R}^{k}$ the solution flow to ODE 
\begin{equation} \label{odeflow}
	\phi_{t}^{n}\left(y\right)=y+\int_{t}^{T} g\left(\phi_{r}^{n}\left(y\right)\right)\dot{X}_{r}^{n} d r, \quad t \in \left[0, T\right],
\end{equation}
for $n=1, 2, \cdots$. Define 
$\tilde{f}^{n}: \Omega \times \left[0, T\right] \times \mathbb{R}^{k} \times \mathbb{R}^{k \times d} \rightarrow \mathbb{R}^{k}$ 
for $n=0, 1, \cdots$ by 
\begin{equation} \label{tranf}
	\tilde{f}_{t}^{n}\left(\tilde{y}, \tilde{z}\right):=\left(D\phi_{t}^{n}\left(\tilde{y}\right)\right)^{-1}\left(f_{t}\left(\phi_{t}^{n}\left(\tilde{y}\right), D\phi_{t}^{n}\left(\tilde{y}\right)\tilde{z}\right)+\frac{1}{2}D^{2}\phi_{t}^{n}\left(\tilde{y}\right)\tilde{z}^{2}\right),
\end{equation}
where 
\begin{equation*}
	D^{2}\phi_{t}^{n}\left(\tilde{y}\right)\tilde{z}^{2}:=\left(\begin{array}{c}
		\operatorname{Tr}\left[D^{2}\phi_{t}^{n, 1}\left(\tilde{y}\right)\tilde{z}\tilde{z}^{\top}\right] \\
		\vdots \\
		\operatorname{Tr}\left[D^{2}\phi_{t}^{n, k}\left(\tilde{y}\right)\tilde{z}\tilde{z}^{\top}\right]
		\end{array}\right), \quad \forall \left(t, \tilde{y}, \tilde{z}\right) \in \left[0, T\right] \times \mathbb{R}^{k} \times \mathbb{R}^{k \times d}.
\end{equation*}
Then we have the following two lemmas. 
\begin{lemma} \label{flowtran}
	Let Assumption \ref{assu2} hold. Assume that $\left|\mathbf{X}\right|_{p\text{-}var} \leq 1$. 
	Then for every $\left(t, \tilde{y}, \tilde{z}\right) \in \left[0, T\right] \times \mathbb{R}^{k} \times \mathbb{R}^{k \times d}$ 
	and $n=0, 1, \cdots$, we have 
	\begin{equation*}
		\left|\tilde{f}^{n}_{t}\left(\tilde{y}, \tilde{z}\right)\right|+\left|\partial_{\tilde{y}}\tilde{f}^{n}_{t}\left(\tilde{y}, \tilde{z}\right)\right| \lesssim_{L, \left|g\right|_{\gamma}} \lambda_{t}+\mu_{t}^{2}+\left|\mathbf{X}\right|_{p\text{-}var}+\left|\tilde{y}\right|^{2}+\left|\tilde{z}\right|^{2}
	\end{equation*}
	and
	\begin{equation*}
		\left|\partial_{\tilde{z}}\tilde{f}^{n}_{t}\left(\tilde{y}, \tilde{z}\right)\right| \lesssim_{L, \left|g\right|_{\gamma}} \mu_{t}+\left|\mathbf{X}\right|_{p\text{-}var}+\left|\tilde{y}\right|+\left|\tilde{z}\right|.
	\end{equation*}
\end{lemma}
\begin{proof}
	For every fixed $\left(t, \tilde{y}, \tilde{z}\right) \in \left[0, T\right] \times \mathbb{R}^{k} \times \mathbb{R}^{k \times d}$ 
	and $n$, we have 
	\begin{align*}
		&\left|\tilde{f}_{t}^{n}\left(\tilde{y}, \tilde{z}\right)\right|\\
		&\quad \leq \left|\left(D\phi_{t}^{n}\left(\tilde{y}\right)\right)^{-1}\right|\left(\lambda_{t}+L\left|\phi_{t}^{n}\left(\tilde{y}\right)\right|^{2}+L\left|D\phi_{t}^{n}\left(\tilde{y}\right)\right|^{2}\left|\tilde{z}\right|^{2}+\frac{1}{2}\left|D^{2}\phi_{t}^{n}\left(\tilde{y}\right)\right|\left|\tilde{z}\right|^{2}\right)\\
		&\quad \lesssim_{L} \left|\left(D\phi^{n}\right)^{-1}\right|_{\infty}\left(\lambda_{t}+\left|\phi^{n}-\tilde{y}\right|_{\infty}^{2}+\left|\tilde{y}\right|^{2}+\left(\left|D\phi^{n}\right|_{\infty}^{2}+\left|D^{2}\phi^{n}\right|_{\infty}\right)\left|\tilde{z}\right|^{2}\right).
	\end{align*}
	Since 
	\begin{align*}
		\partial_{\tilde{y}}\tilde{f}_{t}^{n}\left(\tilde{y}, \tilde{z}\right)&=\left(D\phi_{t}^{n}\left(\tilde{y}\right)\right)^{-1}\left[-D^{2}\phi_{t}^{n}\left(\tilde{y}\right)\tilde{f}_{t}^{n}\left(\tilde{y}, \tilde{z}\right)+\partial_{y}f_{t}\left(\phi_{t}^{n}\left(\tilde{y}\right), D\phi_{t}^{n}\left(\tilde{y}\right)\tilde{z}\right)D\phi_{t}^{n}\left(\tilde{y}\right)\right.\\
		&\quad+\left.\partial_{z}f_{t}\left(\phi_{t}^{n}\left(\tilde{y}\right), D\phi_{t}^{n}\left(\tilde{y}\right)\tilde{z}\right)D^{2}\phi_{t}^{n}\left(\tilde{y}\right)\tilde{z}+\frac{1}{2}D^{3}\phi_{t}^{n}\left(\tilde{y}\right)\tilde{z}^{2}\right],
	\end{align*}
	we have 
	\begin{align*}
		&\left|\partial_{\tilde{y}}\tilde{f}_{t}^{n}\left(\tilde{y}, \tilde{z}\right)\right|\\
		&\quad \leq \left|\left(D\phi_{t}^{n}\left(\tilde{y}\right)\right)^{-1}\right|\left[\left|\tilde{f}_{t}^{n}\left(\tilde{y}, \tilde{z}\right)\right|+\left(\lambda_{t}+L\left|\phi_{t}^{n}\left(\tilde{y}\right)\right|^{2}+L\left|D\phi_{t}^{n}\left(\tilde{y}\right)\right|^{2}\left|\tilde{z}\right|^{2}\right)\left|D\phi_{t}^{n}\left(\tilde{y}\right)\right|\right.\\
		&\quad \quad+\left.\left(\mu_{t}+L\left|\phi_{t}^{n}\left(\tilde{y}\right)\right|+L\left|D\phi_{t}^{n}\left(\tilde{y}\right)\right|\left|\tilde{z}\right|\right)\left|D^{2}\phi_{t}^{n}\left(\tilde{y}\right)\right|\left|\tilde{z}\right|+\frac{1}{2}\left|D^{3}\phi_{t}^{n}\left(\tilde{y}\right)\right|\left|\tilde{z}\right|^{2}\right]\\
		&\quad \lesssim_{L} \left|\left(D\phi^{n}\right)^{-1}\right|_{\infty}\left(\left(1+\left|\left(D\phi^{n}\right)^{-1}\right|_{\infty}\right)\left|D^{2}\phi^{n}\right|_{\infty}+\left|D\phi^{n}\right|_{\infty}\right)\\
    	&\quad \quad \quad \times \left(\lambda_{t}+\mu_{t}^{2}+\left|\phi^{n}-\tilde{y}\right|_{\infty}^{2}+\left|\tilde{y}\right|^{2}\right)\\
		&\quad \quad+\left|\left(D\phi^{n}\right)^{-1}\right|_{\infty}^{2}\left|D^{2}\phi^{n}\right|_{\infty}\left(\left|D\phi^{n}\right|_{\infty}^{2}+\left|D^{2}\phi^{n}\right|_{\infty}\right)\left|\tilde{z}\right|^{2}\\
		&\quad \quad+\left|\left(D\phi^{n}\right)^{-1}\right|_{\infty}\left(\left|D\phi^{n}\right|_{\infty}^{3}+\left(1+\left|D\phi^{n}\right|_{\infty}\right)\left|D^{2}\phi^{n}\right|_{\infty}+\left|D^{3}\phi^{n}\right|_{\infty}\right)\left|\tilde{z}\right|^{2}.
	\end{align*}
	In a similar way, we have 
	\begin{align*}
		\left|\partial_{\tilde{z}}\tilde{f}_{t}^{n}\left(\tilde{y}, \tilde{z}\right)\right| &\leq \left|\left(D\phi_{t}^{n}\left(\tilde{y}\right)\right)^{-1}\right|\left(\left|\partial_{z}f_{t}\left(\phi_{t}^{n}\left(\tilde{y}\right), D\phi_{t}^{n}\left(\tilde{y}\right)\tilde{z}\right)\right|\left|D\phi_{t}^{n}\left(\tilde{y}\right)\right|+\left|D^{2}\phi_{t}^{n}\left(\tilde{y}\right)\right|\left|\tilde{z}\right|\right)\\
		&\lesssim_{L} \left|\left(D\phi^{n}\right)^{-1}\right|_{\infty}\left|D\phi^{n}\right|_{\infty}\left(\mu_{t}+\left|\phi^{n}-\tilde{y}\right|_{\infty}+\left|\tilde{y}\right|\right)\\
    	&\quad+\left|\left(D\phi^{n}\right)^{-1}\right|_{\infty}\left(\left|D\phi^{n}\right|_{\infty}^{2}+\left|D^{2}\phi^{n}\right|_{\infty}\right)\left|\tilde{z}\right|.
	\end{align*}
	By \cite{RPB2}*{Theorem 10.14 and Proposition 11.11}, we have 
	\begin{equation} \label{flowe1}
		\left|\phi^{n}-\tilde{y}\right|_{\infty} \lesssim_{\left|g\right|_{\gamma}} \left|\mathbf{X}\right|_{p\text{-}var}, \quad \forall n=0, 1, \cdots
	\end{equation}
  	and 
	\begin{equation} \label{flowe2}
		\left|D\phi^{n}\right|_{\infty} \vee \left|\left(D\phi^{n}\right)^{-1}\right|_{\infty} \vee \left|D^{2}\phi^{n}\right|_{\infty} \vee \left|D^{3}\phi^{n}\right|_{\infty} \leq C_{\left|g\right|_{\gamma}}, \quad \forall n=0, 1, \cdots.
	\end{equation}
	Combining the above inequalities, we get the desired result. 
\end{proof}
\begin{lemma} \label{flowtran2}
	Under the same assumptions as Lemma \ref{flowtran}, there exist sufficiently small positive real numbers $\varepsilon$ and $\tilde{R}$, 
	only depending on $T, L$ and $\left|g\right|_{\gamma}$, such that for 
	$\left\|\xi\right\|_{\infty}\vee \left\|\int_{0}^{T} \left(\lambda_{r}+\mu_{r}^{2}\right) d r\right\|_{\infty}\vee \left|\mathbf{X}\right|_{p\text{-}var} \leq \varepsilon$, BSDE 
	\begin{equation} \label{tranbsde}
		\tilde{Y}_{t}^{n}=\xi+\int_{t}^{T} \tilde{f}_{r}^{n}\left(\tilde{Y}_{r}^{n}, \tilde{Z}_{r}^{n}\right) d r-\int_{t}^{T} \tilde{Z}_{r}^{n} d W_{r}, \quad t \in \left[0, T\right] 
	\end{equation}
	has a unique solution $\left(\tilde{Y}^{n}, \tilde{Z}^{n}\right) \in \mathbb{S}^{\infty} \times BMO$ satisfying 
	\begin{equation} \label{tranbsdesolue1}
		\left\|\tilde{Y}^{n}\right\|_{\mathbb{S}^{\infty}}+\left\|\tilde{Z}^{n}\right\|_{BMO} \leq \tilde{R},
	\end{equation}
	for $n=0, 1, \cdots$. Moreover, we have 
	\begin{equation} \label{transoluc}
		\lim_{n \rightarrow \infty}\left(\left\|\tilde{Y}^{n}-\tilde{Y}^{0}\right\|_{\mathbb{S}^{\infty}}+\left\|\tilde{Z}^{n}-\tilde{Z}^{0}\right\|_{BMO}\right)=0.
	\end{equation}
\end{lemma}
\begin{proof}
	By Lemma \ref{flowtran} and Theorem \ref{bsdesolueu}, we get the existence of 
	$\varepsilon$ and $\tilde{R}$ and the existence and uniqueness of 
	$\left(\tilde{Y}^{n}, \tilde{Z}^{n}\right), n=0, 1, \cdots$ 
	for $\left\|\xi\right\|_{\infty}\vee \left\|\int_{0}^{T} \left(\lambda_{r}+\mu_{r}^{2}\right) d r\right\|_{\infty} \leq \varepsilon$ 
	and $\left|\mathbf{X}\right|_{p\text{-}var} \leq \varepsilon$. 
	Put $\Delta_{n}\phi:=\phi^{n}-\phi^{0}$ and similarly for other functions. 
	For every fixed $\left(t, \tilde{y}, \tilde{z}\right) \in \left[0, T\right] \times \mathbb{R}^{k} \times \mathbb{R}^{k \times d}$ 
	and $n$, we have 
	\begin{align*}
		&\left|\Delta_{n}\tilde{f}_{t}\left(\tilde{y}, \tilde{z}\right)\right|\\
		&\quad \leq \left|\Delta_{n}\left(D\phi_{t}\left(\tilde{y}\right)\right)^{-1}\right|\left(\left|f_{t}\left(\phi_{t}^{n}\left(\tilde{y}\right), D\phi_{t}^{n}\left(\tilde{y}\right)\tilde{z}\right)\right|+\frac{1}{2}\left|D^{2}\phi_{t}^{n}\left(\tilde{y}\right)\right|\left|\tilde{z}\right|^{2}\right)\\
		&\quad \quad+\left|\left(D\phi_{t}^{0}\left(\tilde{y}\right)\right)^{-1}\right|\left(\left|\Delta_{n}f_{t}\left(\phi_{t}\left(\tilde{y}\right), D\phi_{t}\left(\tilde{y}\right)\tilde{z}\right)\right|+\frac{1}{2}\left|\Delta_{n}D^{2}\phi_{t}\left(\tilde{y}\right)\right|\left|\tilde{z}\right|^{2}\right)\\
		&\quad \lesssim_{L} \left|\Delta_{n}\left(D\phi_{t}\left(\tilde{y}\right)\right)^{-1}\right|\left(\lambda_{t}+\left|\phi_{t}^{n}\left(\tilde{y}\right)\right|^{2}+\left|D\phi_{t}^{n}\left(\tilde{y}\right)\right|^{2}\left|\tilde{z}\right|^{2}+\frac{1}{2}\left|D^{2}\phi_{t}^{n}\left(\tilde{y}\right)\right|\left|\tilde{z}\right|^{2}\right)\\
		&\quad \quad+\left|\left(D\phi_{t}^{0}\left(\tilde{y}\right)\right)^{-1}\right|\left|\Delta_{n}\phi_{t}\left(\tilde{y}\right)\right|\\
    	&\quad \quad \quad \times \left(\lambda_{t}+\left|\phi_{t}^{n}\left(\tilde{y}\right)\right|^{2}+\left|\phi_{t}^{0}\left(\tilde{y}\right)\right|^{2}+\left|D\phi_{t}^{n}\left(\tilde{y}\right)\right|^{2}\left|\tilde{z}\right|^{2}+\left|D\phi_{t}^{0}\left(\tilde{y}\right)\right|^{2}\left|\tilde{z}\right|^{2}\right)\\
		&\quad \quad+\left|\left(D\phi_{t}^{0}\left(\tilde{y}\right)\right)^{-1}\right|\left|\Delta_{n}D\phi_{t}\left(\tilde{y}\right)\right|\left|\tilde{z}\right|\\
    	&\quad \quad \quad \times \left(\mu_{t}+\left|\phi_{t}^{n}\left(\tilde{y}\right)\right|+\left|\phi_{t}^{0}\left(\tilde{y}\right)\right|+\left|D\phi_{t}^{n}\left(\tilde{y}\right)\right|\left|\tilde{z}\right|+\left|D\phi_{t}^{0}\left(\tilde{y}\right)\right|\left|\tilde{z}\right|\right)\\
		&\quad \quad+\frac{1}{2}\left|\left(D\phi_{t}^{0}\left(\tilde{y}\right)\right)^{-1}\right|\left|\Delta_{n}D^{2}\phi_{t}\left(\tilde{y}\right)\right|\left|\tilde{z}\right|^{2}\\
		&\quad \lesssim_{L} \left(\left|\Delta_{n}\left(D\phi\right)^{-1}\right|_{\infty}+\left|\left(D\phi^{0}\right)^{-1}\right|_{\infty}\left(\left|\Delta_{n}\phi\right|_{\infty}+\left|\Delta_{n}D\phi\right|_{\infty}\right)\right)\\
		&\quad \quad \quad \times \left(\lambda_{t}+\mu_{t}^{2}+\left|\phi^{n}-\tilde{y}\right|_{\infty}^{2}+\left|\phi^{0}-\tilde{y}\right|_{\infty}^{2}+\left|\tilde{y}\right|^{2}\right)\\
		&\quad \quad+\left(\left|\Delta_{n}\left(D\phi\right)^{-1}\right|_{\infty}+\left|\left(D\phi_{t}^{0}\left(\tilde{y}\right)\right)^{-1}\right|\left(\left|\Delta_{n}\phi\right|_{\infty}+\left|\Delta_{n}D\phi\right|_{\infty}+\left|\Delta_{n}D^{2}\phi\right|_{\infty}\right)\right)\\
		&\quad \quad \quad \times \left(1+\left|D\phi^{n}\right|_{\infty}^{2}+\left|D\phi^{0}\right|_{\infty}^{2}+\left|D^{2}\phi^{n}\right|_{\infty}\right)\left|\tilde{z}\right|^{2}.
	\end{align*}
	Then 
	\begin{align} 
		&\mathbb{E}_{t}\int_{t}^{T} \left|\left(\Delta_{n}\tilde{f}\right)_{r}\left(\tilde{Y}_{r}^{0}, \tilde{Z}_{r}^{0}\right)\right| d r \notag\\
		&\quad \lesssim_{L} \left(\left|\Delta_{n}\left(D\phi\right)^{-1}\right|_{\infty}+\left|\left(D\phi^{0}\right)^{-1}\right|_{\infty}\left(\left|\Delta_{n}\phi\right|_{\infty}+\left|\Delta_{n}D\phi\right|_{\infty}\right)\right) \notag\\
		&\quad \quad \quad \times \left(\left\|\int_{0}^{T} \left(\lambda_{r}+\mu_{r}^{2}\right) d r\right\|_{\infty}+T\left|\phi^{n}-\tilde{y}\right|_{\infty}^{2}+T\left|\phi^{0}-\tilde{y}\right|_{\infty}^{2}+T\left\|\tilde{Y}^{0}\right\|_{\mathbb{S}^{\infty}}^{2}\right) \label{convere1}\\ 
		&\quad \quad+\left(\left|\Delta_{n}\left(D\phi\right)^{-1}\right|_{\infty}+\left|\left(D\phi_{t}^{0}\left(\tilde{y}\right)\right)^{-1}\right|\left(\left|\Delta_{n}\phi\right|_{\infty}+\left|\Delta_{n}D\phi\right|_{\infty}+\left|\Delta_{n}D^{2}\phi\right|_{\infty}\right)\right) \notag\\
		&\quad \quad \quad \times \left(1+\left|D\phi^{n}\right|_{\infty}^{2}+\left|D\phi^{0}\right|_{\infty}^{2}+\left|D^{2}\phi^{n}\right|_{\infty}\right)\left\|\tilde{Z}^{0}\right\|_{BMO}^{2}. \notag
	\end{align}
	By the classical RDE theory (e.g. \cite{RPB2}*{Theorem 11.12}), we have 
	\begin{equation} \label{convere2}
		\lim_{n \rightarrow \infty}\left(\left|\Delta_{n}\phi\right|_{\infty}+\left|\Delta_{n}D\phi\right|_{\infty}+\left|\Delta_{n}\left(D\phi\right)^{-1}\right|_{\infty}+\left|\Delta_{n}D^{2}\phi\right|_{\infty}\right)=0.
	\end{equation}
	Combining \eqref{flowe1}, \eqref{flowe2}, \eqref{convere1} and \eqref{convere2}, we have 
	\begin{equation*} 
		\lim_{n \rightarrow \infty}\mathop{\esssup}\limits_{\left(\omega, t\right)}\mathbb{E}_{t}\int_{t}^{T} \left|\left(\Delta_{n}\tilde{f}\right)_{r}\left(\tilde{Y}_{r}^{0}, \tilde{Z}_{r}^{0}\right)\right| d r=0.
	\end{equation*}
	Then by Theorem \ref{bsdesolumc}, we get \eqref{transoluc}. 
\end{proof}
We are now ready to prove Theorem \ref{solueu2}. 
\begin{proof}[Proof of Theorem \ref{solueu2}]
	We first show the existence when 
	$\left\|\xi\right\|_{\infty}, \left\|\int_{0}^{T} \left(\lambda_{r}+\mu_{r}^{2}\right) d r\right\|_{\infty}$ and 
	$\left|\mathbf{X}\right|_{p\text{-}var}$ are sufficiently small. 
	Take $X^{n} \in C^{\infty}\left(\left[0, T\right], \mathbb{R}^{e}\right), n=1, 2, \cdots$ 
	such that \eqref{flowcon1} and \eqref{flowcon2} hold. By Lemma \ref{flowtran2}, 
	there exist sufficiently small positive real numbers $\varepsilon$ and $\tilde{R}$, 
	only depending on $T, L$ and $\left|g\right|_{\gamma}$, such that for 
	$\left\|\xi\right\|_{\infty} \leq \varepsilon$ and $\left\|\int_{0}^{T} \left(\lambda_{r}+\mu_{r}^{2}\right) d r\right\|_{\infty}\vee \left|\mathbf{X}\right|_{p\text{-}var} \leq \varepsilon$, 
	BSDE \eqref{tranbsde} has a unique solution 
	$\left(\tilde{Y}^{n}, \tilde{Z}^{n}\right) \in \mathbb{S}^{\infty} \times BMO$ 
	satisfying \eqref{tranbsdesolue1} for $n=0, 1, \cdots$, 
	and \eqref{transoluc} holds. Define 
	\begin{equation} \label{Yd}
		Y^{n}_{t}:=\phi_{t}^{n}\left(\tilde{Y}^{n}_{t}\right), \quad Z^{n}_{t}:=D\phi_{t}^{n}\left(\tilde{Y}^{n}_{t}\right)\tilde{Z}^{n}_{t}, \quad \forall t \in \left[0, T\right], \quad \forall n=0, 1, \cdots.
	\end{equation}
	Since 
	\begin{equation*}
		\left\|Y^{n}\right\|_{\mathbb{S}^{\infty}} \leq \left|\phi^{n}-\tilde{y}\right|_{\infty}+\left\|\tilde{Y}^{n}\right\|_{\mathbb{S}^{\infty}}, \quad \left\|Z^{n}\right\|_{BMO} \leq \left|D\phi^{n}\right|_{\infty}\left\|\tilde{Z}^{n}\right\|_{BMO},
	\end{equation*}
	we have $\left(Y^{n}, Z^{n}\right) \in \mathbb{S}^{\infty} \times BMO$. 
	For $n=1, 2, \cdots$, applying Itô's formula, we have 
	\begin{align*}
		Y_{t}^{n}&=\xi-\int_{t}^{T}\left(\partial_{t}\phi_{r}^{n}\left(\tilde{Y}_{r}^{n}\right)-D\phi_{r}^{n}\left(\tilde{Y}_{r}^{n}\right)\tilde{f}_{r}^{n}\left(\tilde{Y}_{r}^{n}, \tilde{Z}_{r}^{n}\right)+\frac{1}{2}D^{2}\phi_{r}^{n}\left(\tilde{Y}_{r}^{n}\right)\left(\tilde{Z}_{r}^{n}\right)^{2}\right) d r\\
		&\quad-\int_{t}^{T} D\phi_{r}^{n}\left(\tilde{Y}_{r}^{n}\right)\tilde{Z}_{r}^{n} d W_{r}\\
		&=\xi+\int_{t}^{T}\left(f_{t}\left(\phi_{t}^{n}\left(\tilde{Y}_{r}^{n}\right), D\phi_{t}^{n}\left(\tilde{Y}_{r}^{n}\right)\tilde{Z}_{r}^{n}\right)+g\left(\phi_{r}^{n}\left(\tilde{Y}_{r}^{n}\right)\right)\dot{X}_{r}^{n}\right) d r\\
    	&\quad-\int_{t}^{T} D\phi_{r}^{n}\left(\tilde{Y}_{r}^{n}\right)\tilde{Z}_{r}^{n} d W_{r}\\
		&=\xi+\int_{t}^{T} \left(f_{r}\left(Y_{r}^{n}, Z_{r}^{n}\right)+g\left(Y_{r}^{n}\right)\dot{X}_{r}^{n}\right) d r-\int_{t}^{T} Z_{r}^{n} d W_{r}, \quad \forall t \in \left[0, T\right].
	\end{align*}
	Hence, $\left(Y^{n}, Z^{n}\right)$ is a solution to BSDE \eqref{bsdesequ} 
	for $n=1, 2, \cdots$. Put $\Delta_{n}Y:=Y^{n}-Y^{0}$ 
	and similarly for other functions. Note that 
	\begin{align*}
		\left\|\Delta_{n}Y\right\|_{\mathbb{S}^{\infty}} &\leq \left\|\left(\Delta_{n}\phi\right)\left(\tilde{Y}^{n}\right)\right\|_{\mathbb{S}^{\infty}}+\left\|\phi^{0}\left(\tilde{Y}^{n}\right)-\phi^{0}\left(\tilde{Y}^{0}\right)\right\|_{\mathbb{S}^{\infty}}\\
		&\leq \left|\Delta_{n}\phi\right|_{\infty}+\left|D\phi^{0}\right|_{\infty}\left\|\Delta_{n}\tilde{Y}\right\|_{\mathbb{S}^{\infty}}
	\end{align*}
  	and 
	\begin{align*}
		\left\|\Delta_{n}Z\right\|_{BMO} &\leq \left\|D\phi^{n}\left(\tilde{Y}^{n}\right)\right\|_{\mathbb{S}^{\infty}}\left\|\Delta_{n}\tilde{Z}\right\|_{BMO}+\left\|\Delta_{n}\left[D\phi\left(\tilde{Y}\right)\right]\right\|_{\mathbb{S}^{\infty}}\left\|\tilde{Z}^{0}\right\|_{BMO}\\
		&\leq \left|D\phi^{n}\right|_{\infty}\left\|\Delta_{n}\tilde{Z}\right\|_{BMO}\\
		&\quad+\left(\left|\Delta_{n}D\phi\right|_{\infty}+\left|D^{2}\phi^{0}\right|_{\infty}\left\|\Delta_{n}\tilde{Y}\right\|_{\mathbb{S}^{\infty}}\right)\left\|\tilde{Z}^{0}\right\|_{BMO},
	\end{align*}
	for $n=0, 1, \cdots$. 
	Combining \eqref{flowe2}, \eqref{transoluc}, \eqref{convere2} 
	and the last two inequalities, we have 
	\begin{equation*}
		\lim_{n \rightarrow \infty}\left(\left\|\Delta_{n}Y\right\|_{\mathbb{S}^{\infty}}+\left\|\Delta_{n}Z\right\|_{BMO}\right)=0.
	\end{equation*}
	Therefore, $\left(Y^{0}, Z^{0}\right)$ is a solution to rough BSDE \eqref{rbsde2}. Define 
	\begin{equation*}
		\mathcal{B}:=\left\{\left(Y, Z\right) \in \mathbb{S}^{\infty} \times BMO : \left\|\psi^{0}\left(Y\right)\right\|_{\mathbb{S}^{\infty}}+\left\|D\psi^{0}\left(Y\right)Z\right\|_{BMO} \leq \tilde{R}\right\}, 
	\end{equation*}
	where $\psi^{0}$ is the $y$-inverse of $\phi^{0}$. In view of \eqref{Yd}, we have 
	\begin{equation*}
		\tilde{Y}_{t}^{0}=\psi_{t}^{0}\left(Y_{t}^{0}\right), \quad \tilde{Z}_{t}^{0}:=D\psi_{t}^{0}\left(Y_{t}^{0}\right)Z_{t}^{0}, \quad \forall t \in \left[0, T\right], 
	\end{equation*}
	and thus $\left(Y^{0}, Z^{0}\right) \in \mathcal{B}$. 
	\\
	\indent
	We now show the uniqueness of solution in $\mathcal{B}$ when 
	$\left\|\xi\right\|_{\infty}, \left\|\int_{0}^{T} \left(\lambda_{r}+\mu_{r}^{2}\right) d r\right\|_{\infty}$ and 
	$\left|\mathbf{X}\right|_{p\text{-}var}$ are sufficiently small. 
	Assume $\left(\bar{Y}^{0}, \bar{Z}^{0}\right) \in \mathcal{B}$ is another solution. 
	Then there exist $\bar{X}^{n} \in C^{\infty}\left(\left[0, T\right], \mathbb{R}^{e}\right)$ 
	and $\left(\bar{Y}^{n}, \bar{Z}^{n}\right) \in \mathbb{S}^{\infty} \times BMO$ for 
	$n=1, 2, \cdots$, such that $\left(\bar{Y}^{n}, \bar{Z}^{n}\right)$ is a solution to BSDE 
	\begin{equation*}
		\bar{Y}_{t}^{n}=\xi+\int_{t}^{T} \left(f_{r}\left(\bar{Y}_{r}^{n}, \bar{Z}_{r}^{n}\right)+g\left(\bar{Y}_{r}^{n}\right)\dot{\bar{X}}_{r}^{n}\right) d r-\int_{t}^{T} \bar{Z}_{r}^{n} d W_{r}, \quad t \in \left[0, T\right]
	\end{equation*}
	and 
	\begin{equation*}
		\lim_{n \rightarrow \infty}\left(\rho_{p\text{-}var}\left(\bar{\mathbf{X}}^{n}, \mathbf{X}\right)+\left\|\bar{Y}^{n}-\bar{Y}\right\|_{\mathbb{S}^{\infty}}+\left\|\bar{Z}^{n}-\bar{Z}\right\|_{BMO}\right)=0. 
	\end{equation*}
	Assume without loss of generality that 
	\begin{equation*}
		\left|\bar{\mathbf{X}}^{n}\right|_{p\text{-}var} \leq 2\left|\mathbf{X}\right|_{p\text{-}var}, \quad \forall n=1, 2, \cdots.
	\end{equation*}
  	For $n=1, 2, \cdots$, define $\bar{\phi}^{n}$ and $\tilde{\bar{f}}^{n}$ like defining $\phi^{n}$ and $\tilde{f}^{n}$ 
	by \eqref{odeflow} and \eqref{tranf}, and let $\bar{\psi}^{n}$ be the $y$-inverse of $\bar{\phi}^{n}$. 
	Write $\bar{\phi}^{0}:=\phi^{0}$, $\bar{\psi}^{0}:=\psi^{0}$ and $\tilde{\bar{f}}^{0}:=\tilde{f}^{0}$ for the sake of unified notation. 
	Define 
	\begin{equation*}
		\tilde{\bar{Y}}^{n}_{t}:=\bar{\psi}_{t}^{n}\left(\bar{Y}^{n}_{t}\right), \quad \tilde{\bar{Z}}^{n}_{t}:=D\bar{\psi}_{t}^{n}\left(\bar{Y}^{n}_{t}\right)\bar{Z}^{n}_{t}, \quad \forall t \in \left[0, T\right], \quad \forall n=0, 1, \cdots.
	\end{equation*}
	Then we have 
	\begin{equation} \label{tranbsdesolue2}
		\left\|\tilde{\bar{Y}}^{0}\right\|_{\mathbb{S}^{\infty}}+\left\|\tilde{\bar{Z}}^{0}\right\|_{BMO} \leq \tilde{R}. 
	\end{equation}
	Analogous to the proof of the existence, 
	$\left(\tilde{\bar{Y}}^{n}, \tilde{\bar{Z}}^{n}\right)$ is a solution to BSDE 
	\begin{equation*}
		\tilde{\bar{Y}}_{t}^{n}=\xi+\int_{t}^{T} \tilde{\bar{f}}_{r}^{n}\left(\tilde{\bar{Y}}_{r}^{n}, \tilde{\bar{Z}}_{r}^{n}\right) d r-\int_{t}^{T} \tilde{\bar{Z}}_{r}^{n} d W_{r}, \quad t \in \left[0, T\right],
	\end{equation*}
	for $n=1, 2, \cdots$, and we have 
	\begin{equation*}
		\lim_{n \rightarrow \infty}\left(\left\|\Delta_{n}\tilde{\bar{Y}}\right\|_{\mathbb{S}^{\infty}}+\left\|\Delta_{n}\tilde{\bar{Z}}\right\|_{BMO}\right)=0.
	\end{equation*}
	On the other hand, analogous to the proof of Lemma \ref{flowtran2}, we have 
	\begin{equation*} 
		\lim_{n \rightarrow \infty}\mathop{\esssup}\limits_{\left(\omega, t\right)}\mathbb{E}_{t}\int_{t}^{T} \left|\left(\Delta_{n}\tilde{\bar{f}}\right)_{r}\left(\tilde{\bar{Y}}_{r}^{0}, \tilde{\bar{Z}}_{r}^{0}\right)\right| d r=0.
	\end{equation*}
	Then $\left(\tilde{\bar{Y}}^{0}, \tilde{\bar{Z}}^{0}\right)$ is a solution to BSDE 
	\begin{equation*}
		\tilde{\bar{Y}}_{t}^{0}=\xi+\int_{t}^{T} \tilde{f}_{r}^{0}\left(\tilde{\bar{Y}}_{r}^{0}, \tilde{\bar{Z}}_{r}^{0}\right) d r-\int_{t}^{T} \tilde{\bar{Z}}_{r}^{0} d W_{r}, \quad t \in \left[0, T\right],
	\end{equation*}
	satisfying the estimate \eqref{tranbsdesolue2}. 
	By Theorem \ref{bsdesolueu}, we have $\left(\tilde{\bar{Y}}^{0}, \tilde{\bar{Z}}^{0}\right)=\left(\tilde{Y}^{0}, \tilde{Z}^{0}\right)$ 
	and thus $\left(\bar{Y}^{0}, \bar{Z}^{0}\right)=\left(Y^{0}, Z^{0}\right)$. 
\end{proof}
\subsection{Solvability when $g^{i}$ only depends on $y^{i}$} \label{S3.2}
To show the local existence and uniqueness of the solution to rough BSDE \eqref{rbsde2}, 
we need the following assumption. 
\begin{assumption} \label{assu3}
	~
	\begin{enumerate}[(i)]
		\item $\xi \in L^{\infty}\left(\Omega, \mathcal{F}_{T}; \mathbb{R}^{k}\right)$;
		\item there exist constants $\kappa, \theta \in \left[0, 1\right)$ such that 
		for every $\left(t, y, z\right) \in \left[0, T\right] \times \mathbb{R}^{k} \times \mathbb{R}^{k \times d}$ 
		and $i=1, \cdots, k$, 
		\begin{equation*}
			\left|f_{t}^{i}\left(y, z\right)\right| \lesssim 1+\left|y\right|^{2}+\left|z^{i}\right|^{2}+\sum_{j \neq i}\left|z^{j}\right|^{1+\theta}, \quad \left|\partial_{y}f_{t}^{i}\left(y, z\right)\right| \lesssim 1+\left|y\right|^{2}+\kappa\left|z\right|^{2},
		\end{equation*}
		\begin{equation*}
			\left|\partial_{z^{i}}f_{t}^{i}\left(y, z\right)\right| \lesssim 1+\left|y\right|+\left|z\right|, \quad \left|\partial_{z^{j}}f_{t}^{i}\left(y, z\right)\right| \lesssim 1+\left|y\right|+\left|z\right|^{\theta}, \quad \forall j \neq i;
		\end{equation*}
		\item $g \in Lip^{\gamma}\left(\mathbb{R}^{k}, \mathbb{R}^{k \times e}\right)$ 
		for some real number $\gamma > p+2$, and $g^{i}$ only depends on $y^{i}$ for $i=1, \cdots, k$.
	\end{enumerate}
\end{assumption}
The following assumption is further required to ensure the global existence and uniqueness of the solution. 
\begin{assumption} \label{assu4}
	For every $\left(t, y, z\right) \in \left[0, T\right] \times \mathbb{R}^{k} \times \mathbb{R}^{k \times d}$ 
	and $i=1, \cdots, k$, 
	\begin{equation*}
		\left|f_{t}^{i}\left(y, z\right)\right| \lesssim 1+\left|y\right|+\left|z^{i}\right|^{2}.
	\end{equation*}
\end{assumption}
We now give the following solvability result for rough BSDE \eqref{rbsde2}. 
\begin{theorem} \label{solueu3}
	Let Assumption \ref{assu3} hold for $\kappa$ sufficiently small 
	(where the upper bound only depends on $\left\|\xi\right\|_{\infty}$ and $\left|g\right|_{\gamma}$). 
	Then there exist $\varepsilon > 0$ and a bounded subset $\mathcal{B}_{\varepsilon}$ of 
	$\mathbb{S}^{\infty}\left(\left[T-\varepsilon, T\right], \Omega; \mathbb{R}^{k}\right) \times BMO\left(\left[T-\varepsilon, T\right], \Omega; \mathbb{R}^{k \times d}\right)$, 
	only depending on $\left\|\xi\right\|_{\infty}, \left|g\right|_{\gamma}$ and $\left|\mathbf{X}\right|_{p\text{-}var}$, 
	such that the rough BSDE \eqref{rbsde2} has a unique local solution $\left(Y, Z\right)$ 
	on the time interval $\left[T-\varepsilon, T\right]$ with $\left(Y, Z\right) \in \mathcal{B}_{\varepsilon}$. 
	If Assumption \ref{assu4} holds additionally, then the rough BSDE \eqref{rbsde2} has a 
	unique global solution on $\left[0, T\right]$. 
\end{theorem}
We prove the local existence and uniqueness result, 
using  Lemmas \ref{flowtran2.1} and \ref{flowtran2.2} below instead of Lemmas \ref{flowtran} and \ref{flowtran2}, in an analogous way to that of  Theorem \ref{solueu2}.  
We prove the further global existence and uniqueness result under additional Assumption \ref{assu4}, 
using a similar argument  to that of Diehl and Friz \cite{RBSDE}*{Theorem 3} 
together with the a prior estimate of Fan et al. \cite{DQBSDE2}*{Lemma 4.1}.  We omit here the detailed proof of Theorem \ref{solueu3}.  Concerning functions $X^{n}, \phi^{n}$ and $\tilde{f}^{n}$ defined in Section \ref{S3.1}, we have 
\begin{lemma} \label{flowtran2.1}
	Let Assumption \ref{assu3} hold. Assume that $\left|\mathbf{X}\right|_{p\text{-}var} \leq 1$. 
	Then for every $\tilde{\kappa} > 0$, there exists $\varepsilon_{0} > 0$ 
	only depending on $\left|g\right|_{\gamma}$ 
	such that for every $\left(t, \tilde{y}, \tilde{z}\right) \in \left[T-\varepsilon_{0}, T\right] \times \mathbb{R}^{k} \times \mathbb{R}^{k \times d}$, 
	$i=1, \cdots, k$ and $n=0, 1, \cdots$, we have 
	\begin{equation*}
		\left|\tilde{f}_{t}^{n, i}\left(\tilde{y}, \tilde{z}\right)\right| \lesssim_{\left|g\right|_{\gamma}} 1+\left|\tilde{y}\right|^{2}+\left|\tilde{z}^{i}\right|^{2}+\sum_{j \neq i}\left|\tilde{z}^{j}\right|^{1+\theta},
	\end{equation*}
	\begin{equation*}
		\left|\partial_{\tilde{y}}\tilde{f}_{t}^{n, i}\left(\tilde{y}, \tilde{z}\right)\right| \lesssim_{\left|g\right|_{\gamma}} 1+\left|\tilde{y}\right|^{2}+\left(\kappa+\tilde{\kappa}\right)\left|\tilde{z}\right|^{2},
	\end{equation*}
	\begin{equation*}
		\left|\partial_{\tilde{z}^{i}}\tilde{f}_{t}^{n, i}\left(\tilde{y}, \tilde{z}\right)\right| \lesssim_{\left|g\right|_{\gamma}} 1+\left|\tilde{y}\right|+\left|\tilde{z}\right|, \quad \left|\partial_{\tilde{z}^{j}}\tilde{f}_{t}^{n, i}\left(\tilde{y}, \tilde{z}\right)\right| \lesssim_{\left|g\right|_{\gamma}} 1+\left|\tilde{y}\right|+\left|\tilde{z}\right|^{\theta}, \quad \forall j \neq i;
	\end{equation*}
	Moreover, if Assumption \ref{assu4} holds additionally, then for every $\left(t, \tilde{y}, \tilde{z}\right) \in \left[0, T\right] \times \mathbb{R}^{k} \times \mathbb{R}^{k \times d}$, 
	$i=1, \cdots, k$ and $n=0, 1, \cdots$, we have 
	\begin{equation*}
		\left|\tilde{f}_{t}^{n, i}\left(\tilde{y}, \tilde{z}\right)\right| \lesssim_{\left|g\right|_{\gamma}} 1+\left|\tilde{y}\right|+\left|\tilde{z}^{i}\right|^{2}.
	\end{equation*}
\end{lemma}
\begin{proof}
	Fix $i, n$ and $\left(t, \tilde{y}, \tilde{z}\right) \in \left[0, T\right] \times \mathbb{R}^{k} \times \mathbb{R}^{k \times d}$. 
	Since $g^{i}$ only depends on $y^{i}$, we can solve the RDE \eqref{rdeflow} and ODE \eqref{odeflow} component by component. 
	Then we have $\phi_{t}^{n, i}\left(y\right)=\phi_{t}^{n, i}\left(y^{i}\right)$ and 
	\begin{align*}
		&f_{t}^{i}\left(\phi_{t}^{n}\left(\tilde{y}\right), D\phi_{t}^{n}\left(\tilde{y}\right)\tilde{z}\right)\\
		&\quad=f_{t}^{i}\left(\phi_{t}^{n, 1}\left(\tilde{y}^{1}\right), \cdots, \phi_{t}^{n, k}\left(y^{k}\right), D\phi_{t}^{n, 1}\left(\tilde{y}^{1}\right)\tilde{z}^{1}, \cdots, D\phi_{t}^{n, k}\left(\tilde{y}^{k}\right)\tilde{z}^{k}\right).
	\end{align*}
	In view of \eqref{tranf}, we have 
	\begin{equation*}
		\tilde{f}_{t}^{n, i}\left(\tilde{y}, \tilde{z}\right)=\frac{1}{D\phi_{t}^{n, i}\left(\tilde{y}^{i}\right)}\left(f_{t}^{i}\left(\phi_{t}^{n}\left(\tilde{y}\right), D\phi_{t}^{n}\left(\tilde{y}\right)\tilde{z}\right)+\frac{1}{2}D^{2}\phi_{t}^{n, i}\left(\tilde{y}^{i}\right)\left|\tilde{z}^{i}\right|^{2}\right).
	\end{equation*}
	Then we have 
	\begin{align*}
		&\partial_{\tilde{y}^{i}}\tilde{f}_{t}^{n, i}\left(\tilde{y}, \tilde{z}\right)\\
		&\quad =\frac{1}{D\phi_{t}^{n, i}\left(\tilde{y}^{i}\right)}\left[-D^{2}\phi_{t}^{n, i}\left(\tilde{y}^{i}\right)\tilde{f}_{t}^{n, i}\left(\tilde{y}, \tilde{z}\right)+\partial_{y^{i}}f_{t}^{i}\left(\phi_{t}^{n}\left(\tilde{y}\right), D\phi_{t}^{n}\left(\tilde{y}\right)\tilde{z}\right)D\phi_{t}^{n, i}\left(\tilde{y}^{i}\right)\right.\\
		&\quad \quad+\left.\partial_{z^{i}}f_{t}^{i}\left(\phi_{t}^{n}\left(\tilde{y}\right), D\phi_{t}^{n}\left(\tilde{y}\right)\tilde{z}\right)D^{2}\phi_{t}^{n, i}\left(\tilde{y}^{i}\right)\tilde{z}^{i}+\frac{1}{2}D^{3}\phi_{t}^{n, i}\left(\tilde{y}^{i}\right)\left|\tilde{z}^{i}\right|^{2}\right],
	\end{align*}
	\begin{align*}
		\partial_{\tilde{y}^{j}}\tilde{f}_{t}^{n, i}\left(\tilde{y}, \tilde{z}\right)&=\frac{1}{D\phi_{t}^{n, i}\left(\tilde{y}^{i}\right)}\left[\partial_{y^{j}}f_{t}^{i}\left(\phi_{t}^{n}\left(\tilde{y}\right), D\phi_{t}^{n}\left(\tilde{y}\right)\tilde{z}\right)D\phi_{t}^{n, j}\left(\tilde{y}^{j}\right)\right.\\
		&\quad+\left.\partial_{z^{j}}f_{t}^{i}\left(\phi_{t}^{n}\left(\tilde{y}\right), D\phi_{t}^{n}\left(\tilde{y}\right)\tilde{z}\right)D^{2}\phi_{t}^{n, j}\left(\tilde{y}^{j}\right)\tilde{z}^{j}\right], \quad j \neq i,
	\end{align*}
	\begin{equation*}
		\partial_{\tilde{z}^{i}}\tilde{f}_{t}^{n, i}\left(\tilde{y}, \tilde{z}\right)=\frac{1}{D\phi_{t}^{n, i}\left(\tilde{y}^{i}\right)}\left(\partial_{\tilde{z}^{i}}f_{t}^{i}\left(\phi_{t}^{n}\left(\tilde{y}\right), D\phi_{t}^{n}\left(\tilde{y}\right)\tilde{z}\right)D\phi_{t}^{n, i}\left(\tilde{y}^{i}\right)+D^{2}\phi_{t}^{n, i}\left(\tilde{y}^{i}\right)\tilde{z}^{i}\right),
	\end{equation*}
	\begin{equation*}
		\partial_{\tilde{z}^{j}}\tilde{f}_{t}^{n, i}\left(\tilde{y}, \tilde{z}\right)=\frac{1}{D\phi_{t}^{n, i}\left(\tilde{y}^{i}\right)}\partial_{\tilde{z}^{j}}f_{t}^{i}\left(\phi_{t}^{n}\left(\tilde{y}\right), D\phi_{t}^{n}\left(\tilde{y}\right)\tilde{z}\right)D\phi_{t}^{n, j}\left(\tilde{y}^{j}\right), \quad j \neq i.
	\end{equation*}
	By \cite{RBSDE}*{Lemma B.1}, for every $\tilde{\kappa} > 0$, there exists $\varepsilon_{0} > 0$ 
	only depending on $\left|g\right|_{\gamma}$ such that for every $\left(t, \tilde{y}\right) \in \left[T-\varepsilon_{0}, T\right] \times \mathbb{R}^{k}$ 
	and $n=0, 1, \cdots$, we have 
	\begin{equation} \label{flowe2.1}
		\left|\phi_{t}^{n}\left(\tilde{y}\right)-\tilde{y}\right| \vee \left|D^{2}\phi_{t}^{n}\left(\tilde{y}\right)\right| \vee \left|D^{3}\phi_{t}^{n}\left(\tilde{y}\right)\right| \leq \tilde{\kappa}.
	\end{equation}
	Analogous to the proof of Lemma \ref{flowtran}, we have the desired result. 
\end{proof}
\begin{lemma} \label{flowtran2.2}
	Let Assumption \ref{assu3} hold for $\kappa$ sufficiently small 
	(where the upper bound only depends on $\left\|\xi\right\|_{\infty}$ and $\left|g\right|_{\gamma}$). 
	Then there exist $\varepsilon > 0$ and a bounded subset $\tilde{\mathcal{B}}_{\varepsilon}$ of 
	$\mathbb{S}^{\infty}\left(\left[T-\varepsilon, T\right], \Omega; \mathbb{R}^{k}\right) \times BMO\left(\left[T-\varepsilon, T\right], \Omega; \mathbb{R}^{k \times d}\right)$, 
	only depending on $\left\|\xi\right\|_{\infty}, \left|g\right|_{\gamma}$ and $\left|\mathbf{X}\right|_{p\text{-}var}$, 
	such that for $n=0, 1, \cdots$, the BSDE \eqref{tranbsde} has a unique local solution 
	$\left(\tilde{Y}^{n}, \tilde{Z}^{n}\right)$ on the time interval $\left[T-\varepsilon, T\right]$ 
	with $\left(\tilde{Y}^{n}, \tilde{Z}^{n}\right) \in \tilde{\mathcal{B}}_{\varepsilon}$. 
	Moreover, \eqref{transoluc} also holds. 
\end{lemma}
\begin{proof}
	Assume without loss of generality that $\left|\mathbf{X}\right|_{p\text{-}var} \leq 1$. 
	By Lemma \ref{flowtran2.1}, for every $\tilde{\kappa} > 0$, there exists $\varepsilon_{0} > 0$ 
	such that when restricted on the time interval $\left[T-\varepsilon_{0}, T\right]$, 
	$\tilde{f}^{n}$ satisfies Assumption \ref{assu3} but replacing $\kappa$ by $\kappa+\tilde{\kappa}$, 
	for $n=0, 1, \cdots$. 
	Then we can choose $\varepsilon_{0}$ to make $\kappa+\tilde{\kappa}$ sufficiently small, 
	and thus by Theorem \ref{bsdesolueu2}, we get the existence and uniqueness of the local solution 
	$\left(\tilde{Y}^{n}, \tilde{Z}^{n}\right) \in \tilde{\mathcal{B}}_{\varepsilon}$ 
	on the time interval $\left[T-\varepsilon, T\right]$, for some $\varepsilon \in \left(0, \varepsilon_{0}\right]$ 
	and bounded subset $\tilde{\mathcal{B}}_{\varepsilon}$. 
	Analogous to the proof of Lemma \ref{flowtran2}, we can show that 
	\begin{equation*} 
		\lim_{n \rightarrow \infty}\mathop{\esssup}\limits_{\left(\omega, t\right)}\mathbb{E}_{t}\int_{t}^{T} \left|\tilde{f}_{r}^{n}\left(\tilde{Y}_{r}^{0}, \tilde{Z}_{r}^{0}\right)-\tilde{f}_{r}^{0}\left(\tilde{Y}_{r}^{0}, \tilde{Z}_{r}^{0}\right)\right| d r=0.
	\end{equation*}
	Then by Theorem \ref{bsdesolumc2}, we get \eqref{transoluc}. 
\end{proof}
\begin{remark}
	As in Diehl and Friz \cite{RBSDE}, we may also consider the following more general rough BSDE 
	\begin{equation} \label{grbsde}
		Y_{t}=\xi+\int_{t}^{T} f_{r}\left(Y_{r}, Z_{r}\right) d r+\int_{t}^{T} g\left(S_{r}, Y_{r}\right) d \mathbf{X}_{r}-\int_{t}^{T} Z_{r} d W_{r}, \quad t \in \left[0, T\right],
	\end{equation}
	for given an Itô process $S$, a deterministic vector field $g$ 
	and a geometric $p$-rough path $\mathbf{X}$ with $p \in \left[1, \infty\right)$. 
	In an analogous way, we may get similar solvability results 
	of the rough BSDE \eqref{grbsde} as Theorems \ref{solueu2} and \ref{solueu3}. 
	We do not pursue the details of the extension here. 
\end{remark}

\begin{remark}
	One anonymous referee suggested us to use the result of Hu and Tang~\cite{DQBSDE2} on multidimensional diagonally quadratic BSDEs, to discuss rough BSDEs.  The last subsection is added to respond to the suggestion. 
	Other well-posedness results on multidimensional quadratic BSDEs (e.g. \cites{QBSDE2, QBSDE3}) are also worthy to be used  to solve the rough BSDE \eqref{rbsde2} with the flow transformation method, but their conditions associated to the suitable flow transformation seem rather awkward and remain to be studied. 
\end{remark}

\section{Integration and composition} \label{S4}
In this section, we build a framework for studying BSDEs with a linear rough drift \eqref{inrbsde3}. 
To this end, we will introduce stochastic controlled rough paths, 
establish rough stochastic integrals in $p$-variation scale 
(or simply $p$-rough stochastic integrals), 
and investigate composition of stochastic controlled rough paths. 
\subsection{Stochastic controlled rough paths} \label{S4.1}
We first introduce stochastic controlled rough paths. 
This notion is first introduced by Friz et al. \cite{RSDE3}*{Definitions 3.2 and 3.4} 
in Hölder scale. 
\begin{definition} \label{scrp}
	Given $X \in C^{p\text{-}var}\left(\left[0, T\right], \mathbb{R}^{e}\right)$, 
	for $m \in \left[1, \infty\right]$ and $q, q^{\prime} \in \left[1, \infty\right)$, 
	we call $\left(Y, Y^{\prime}\right)$ an $L^{m}$-integrable stochastic 
	controlled rough path of finite 
	$\left(q, q^{\prime}\right)$-variation with values in $V$, 
	denoted by $\left(Y, Y^{\prime}\right) \in \mathbf{D}_{X}^{\left(q, q^{\prime}\right)\text{-}var}L_{m}\left(\left[0, T\right], \Omega; V\right)=:\mathbf{D}_{X}^{\left(q, q^{\prime}\right)\text{-}var}L_{m}$, 
	if the following are satisfied
  	\begin{enumerate}[(i)]
		\item $Y \in C^{q\text{-}var}L_{m}\left(\left[0, T\right], \Omega; V\right)$;
		\item $Y^{\prime} \in C^{q^{\prime}\text{-}var}L_{m}\left(\left[0, T\right], \Omega; \mathcal{L}\left(\mathbb{R}^{e}, V\right)\right)$;
		\item writing 
    	\begin{equation*}
      	R^{Y}: \Omega \times \Delta \rightarrow V, \quad R_{s, t}^{Y}:= \delta Y_{s, t}-Y_{s}^{\prime} \delta X_{s, t}, \quad \forall \left(s, t\right) \in \Delta,
    	\end{equation*}
    	we have $\mathbb{E}_{\cdot}R^{Y} \in C_{2}^{\frac{qq^{\prime}}{q+q^{\prime}}\text{-}var}L_{m}\left(\left[0, T\right], \Omega; V\right)$. 
  	\end{enumerate}
\end{definition}
\begin{definition} \label{dcrp}
	We call $\left(Y, Y^{\prime}\right)$ a controlled rough path 
	of finite $\left(q, q^{\prime}\right)$-variation with values in 
	$L^{m}\left(\Omega; V\right)$, denoted by 
	$\left(Y, Y^{\prime}\right) \in \mathscr{D}_{X}^{\left(q, q^{\prime}\right)\text{-}var}L_{m}\left(\left[0, T\right], \Omega; V\right)=:\mathscr{D}_{X}^{\left(q, q^{\prime}\right)\text{-}var}L_{m}$, 
	if $\left(Y, Y^{\prime}\right) \in \mathbf{D}_{X}^{\left(q, q^{\prime}\right)\text{-}var}L_{m}$ 
	and $R^{Y} \in C_{2}^{\frac{qq^{\prime}}{q+q^{\prime}}\text{-}var}L_{m}$. 
\end{definition}
Let $K \geq 1$ be a constant. For $\left(Y, Y^{\prime}\right) \in \mathbf{D}_{X}^{\left(q, q^{\prime}\right)\text{-}var}L_{m}$, 
define 
\begin{equation} \label{norm}
	\left\|Y, Y^{\prime}\right\|_{\mathbf{D}_{X}^{\left(q, q^{\prime}\right)\text{-}var}L_{m}}^{\left(K\right)}:=K\left\|Y\right\|_{C^{q\text{-}var}L_{m}}+\left\|Y^{\prime}\right\|_{C^{q^{\prime}\text{-}var}L_{m}}+K\left\|\mathbb{E}_{\cdot}R^{Y}\right\|_{m,\frac{qq^{\prime}}{q+q^{\prime}}\text{-}var}.
\end{equation}
Then $\left(\mathbf{D}_{X}^{\left(q, q^{\prime}\right)\text{-}var}L_{m}, \left\|\cdot\right\|_{\mathbf{D}_{X}^{\left(q, q^{\prime}\right)\text{-}var}L_{m}}^{\left(K\right)}\right)$ 
is a Banach space and $\mathscr{D}_{X}^{\left(q, q^{\prime}\right)\text{-}var}L_{m}$ 
is its Banach subspace.  
For $K, K^{\prime} \geq 1$, 
$\left\|\cdot\right\|_{\mathbf{D}_{X}^{\left(q, q^{\prime}\right)\text{-}var}L_{m}}^{\left(K\right)}$ 
and $\left\|\cdot\right\|_{\mathbf{D}_{X}^{\left(q, q^{\prime}\right)\text{-}var}L_{m}}^{\left(K^{\prime}\right)}$ 
are equivalent. Moreover, 
for $X, \bar{X} \in C^{p\text{-}var}$, 
$\left(Y, Y^{\prime}\right) \in \mathbf{D}_{X}^{\left(q, q^{\prime}\right)\text{-}var}L_{m}$   
and $\left(\bar{Y}, \bar{Y}^{\prime}\right) \in \mathbf{D}_{\bar{X}}^{\left(q, q^{\prime}\right)\text{-}var}L_{m}$, 
we introduce the ``distance'' 
\begin{align}
	\left\|Y, Y^{\prime};\bar{Y}, \bar{Y}^{\prime}\right\|_{\mathbf{D}_{X, \bar{X}}^{\left(q, q^{\prime}\right)\text{-}var}L_{m}}^{\left(K\right)}&:=K\left\|Y-\bar{Y}\right\|_{C^{q\text{-}var}L_{m}}+\left\|Y^{\prime}-\bar{Y}^{\prime}\right\|_{C^{q^{\prime}\text{-}var}L_{m}} \notag\\
	&\quad+K\left\|\mathbb{E}_{\cdot}R^{Y}-\mathbb{E}_{\cdot}R^{\bar{Y}}\right\|_{m,\frac{qq^{\prime}}{q+q^{\prime}}\text{-}var}. \label{dist}
\end{align}
Although this ``distance" is not a metric for $X \neq \bar{X}$, since 
$\left(Y, Y^{\prime}\right)$ and $\left(\bar{Y}, \bar{Y}^{\prime}\right)$ 
lie in different Banach spaces, 
it will be usefull to describe the stability of integration and composition 
and continuity of the solution map. 
The following two lemmas give simple examples for stochastic controlled rough paths 
and will be used to study BSDEs with a linear rough drift in Section \ref{S5}. 
\begin{lemma} \label{L2scrp}
	For any $F \in \mathbb{L}^{2}$, we have $\int_{0}^{\cdot} F_{r} d r \in C^{1\text{-}var}L_{2}$ and 
	\begin{equation} \label{L2scrpe}
		\left\|\int_{0}^{\cdot} F_{r} d r\right\|_{2,1\text{-}var;\left[s, t\right]} \leq \left\|F\right\|_{\mathbb{L}^{2}}\left|t-s\right|^{\frac{1}{2}}, \quad \forall \left(s, t\right) \in \Delta.
	\end{equation}
	As a consequence, for any $q, q^{\prime} \in \left[2, \infty\right)$, $K \geq 1$ and $X \in C^{p\text{-}var}$, 
	we have $\left(\int_{0}^{\cdot} F_{r} d r, 0\right) \in \mathscr{D}_{X}^{\left(q, q^{\prime}\right)\text{-}var}L_{2}$ and 
	\begin{equation*}
		\left\|\int_{0}^{\cdot} F_{r} d r, 0\right\|_{\mathbf{D}_{X}^{\left(q, q^{\prime}\right)\text{-}var}L_{2}}^{\left(K\right)} \lesssim KT^{\frac{1}{2}}\left\|F\right\|_{\mathbb{L}^{2}}.
	\end{equation*}
\end{lemma}
\begin{proof}
	For every $\left(s, t\right) \in \Delta$ and 
	$\pi=\left\{s=t_{0}<t_{1}<\cdots<t_{N}=t\right\} \in \mathcal{P}\left[s, t\right]$, 
	applying the Cauchy--Schwarz inequality, we have 
	\begin{align*}
		\sum_{i=0}^{N-1}\left\|\int_{t_{i}}^{t_{i+1}} F_{r} d r\right\|_{2} &\leq \sum_{i=0}^{N-1}\left(\mathbb{E}\int_{t_{i}}^{t_{i+1}} \left|F_{r}\right|^{2} d r\right)^{\frac{1}{2}}\left|t_{i+1}-t_{i}\right|^{\frac{1}{2}}\\
		&\leq \left(\sum_{i=0}^{N-1}\mathbb{E}\int_{t_{i}}^{t_{i+1}} \left|F_{r}\right|^{2} d r\right)^{\frac{1}{2}}\left(\sum_{i=0}^{N-1}\left(t_{i+1}-t_{i}\right)\right)^{\frac{1}{2}}\\
		&\leq \left\|F\right\|_{\mathbb{L}^{2}}\left|t-s\right|^{\frac{1}{2}}. 
	\end{align*}
	Hence, $\int_{0}^{\cdot} F_{r} d r \in C^{1\text{-}var}L_{2}$ and \eqref{L2scrpe} holds. 
\end{proof}
\begin{lemma} \label{mtg}
	Let $M \in \mathscr{M}_{m}$ with $m \in \left[2, \infty\right)$. 
	Then $M \in C^{m\text{-}var}L_{m}$ and we have 
	\begin{equation} \label{m2var}
		\left\|\delta M\right\|_{m,m\text{-}var;\left[s, t\right]} \lesssim \left\|\delta M_{s, t}\right\|_{m}, \quad \forall \left(s, t\right) \in \Delta. 
	\end{equation}
	As a consequence, for any $q \in \left[m, \infty\right)$, $q^{\prime} \in \left[2, \infty\right)$, 
	$K \geq 1$ and $X \in C^{p\text{-}var}$, we have 
	$\left(M, 0\right) \in \mathbf{D}_{X}^{\left(q, q^{\prime}\right)\text{-}var}L_{m}$ and 
	\begin{equation*}
		\left\|M, 0\right\|_{\mathbf{D}_{X}^{\left(q, q^{\prime}\right)\text{-}var}L_{m}}^{\left(K\right)} \lesssim K\left\|M_{T}\right\|_{m}.
	\end{equation*}
\end{lemma}
\begin{proof}
	Since $M$ is continuous and $L^{m}$-integrable, the dominated convergence theorem gives 
	$M \in C\left(\left[0, T\right], L^{m}\left(\Omega\right)\right)$. 
	For every $\left(s, t\right) \in \Delta$ and every partition 
	$\pi=\left\{s=t_{0}<t_{1}<\cdots<t_{N}=t\right\} \in \mathcal{P}\left[s, t\right]$, 
	applying the Burkholder--Davis--Gundy inequality, we have 
	\begin{equation*}
		\sum_{i=0}^{N-1}\left\|\delta M_{t_{i}, t_{i+1}}\right\|_{m}^{m} \lesssim \sum_{i=0}^{N-1}\mathbb{E}\left|\delta \langle M\rangle_{t_{i}, t_{i+1}}\right|^{\frac{m}{2}} \lesssim \mathbb{E}\left|\sum_{i=0}^{N-1}\delta \langle M\rangle_{t_{i}, t_{i+1}}\right|^{\frac{m}{2}} \lesssim \left\|\delta M_{s, t}\right\|_{m}^{m},
	\end{equation*}
	where $\langle M\rangle$ is the quadratic variation process of $M$. 
	Then $M \in C^{m\text{-}var}L_{m}$ and \eqref{m2var} holds. 
	Since $\mathbb{E}_{s}R_{s, t}^{M}=\mathbb{E}_{s}\delta M_{s, t}=0$ for every $\left(s, t\right) \in \Delta$, 
	we have $\mathbb{E}_{\cdot}R^{M} \in C_{2}^{1\text{-}var}L_{m}$ and thus 
	$\left(M, 0\right) \in \mathbf{D}_{X}^{\left(q, q^{\prime}\right)\text{-}var}L_{m}$. 
\end{proof}
In the sequel of this section, we let $m \in \left[2, \infty\right)$ and $q, q^{\prime} \in \left[p, \infty\right)$ 
such that $\frac{1}{q}+\frac{1}{q^{\prime}} > \frac{1}{2}$. 
We have the following result on decomposition of stochastic controlled rough paths. 
A similar result in Hölder scale can be found in \cite{RSDE3}*{Theorem 3.5}. 
\begin{proposition} \label{deco}
	For any $\left(Y, Y^{\prime}\right) \in \mathbf{D}_{X}^{\left(q, q^{\prime}\right)\text{-}var}L_{m}$, 
	there exists a unique pair of processes $\left(Y^{M}, Y^{J}\right)$ such that 
	$Y^{M} \in \mathscr{M}_{m} \cap C^{q\text{-}var}L_{m}$ with $Y_{0}^{M}=0$, 
	$\left(Y^{J}, Y^{\prime}\right) \in \mathscr{D}_{X}^{\left(q, q^{\prime}\right)\text{-}var}L_{m}$, 
	and $Y=Y^{M}+Y^{J}$. Moreover, we have 
	\begin{equation} \label{RYJe}
		\left\|R^{Y^{J}}\right\|_{m,\frac{qq^{\prime}}{q+q^{\prime}}\text{-}var} \lesssim \left\|\delta Y^{\prime}\right\|_{m,q^{\prime}\text{-}var}\left|\delta X\right|_{p\text{-}var}+\left\|\mathbb{E}_{\cdot}R^{Y}\right\|_{m,\frac{qq^{\prime}}{q+q^{\prime}}\text{-}var}.
	\end{equation}
\end{proposition}
\begin{proof}
	For every $\left(s, u, t\right) \in \Delta_{2}$, we have 
	\begin{equation*}
		\left(\mathbb{E}_{s}-\mathbb{E}_{u}\right)\delta Y_{u, t}=\left(\mathbb{E}_{s}-\mathbb{E}_{u}\right)\left(Y_{u}^{\prime}\delta X_{u, t}+R_{u, t}^{Y}\right)=\left(\mathbb{E}_{s}-\mathbb{E}_{u}\right)\left(\delta Y_{s, u}^{\prime}\delta X_{u, t}+R_{u, t}^{Y}\right).
	\end{equation*}
	Then for every $\left(s, u, t\right) \in \Delta_{2}$, we have 
	\begin{align*}
		\left\|\left(\mathbb{E}_{s}-\mathbb{E}_{u}\right)\delta Y_{u, t}\right\|_{m} &\lesssim \left\|\delta Y_{s, u}^{\prime}\right\|_{m}\left|\delta X_{u, t}\right|+\left\|\mathbb{E}_{u}R_{u, t}^{Y}\right\|_{m}\\
		&\lesssim \left\|\delta Y^{\prime}\right\|_{m, q^{\prime}\text{-}var; \left[s, t\right]}\left|\delta X\right|_{p\text{-}var; \left[s, t\right]}+\left\|\mathbb{E}_{\cdot}R^{Y}\right\|_{m,\frac{qq^{\prime}}{q+q^{\prime}}\text{-}var; \left[s, t\right]}\\
		&\lesssim \left(\left\|\delta Y^{\prime}\right\|_{m, q^{\prime}\text{-}var; \left[s, t\right]}^{\frac{1}{z}}\left|\delta X\right|_{p\text{-}var; \left[s, t\right]}^{\frac{1}{z}}+\left\|\mathbb{E}_{\cdot}R^{Y}\right\|_{m,\frac{qq^{\prime}}{q+q^{\prime}}\text{-}var; \left[s, t\right]}^{\frac{1}{z}}\right)^{z}, 
	\end{align*}
	where $z:=\frac{1}{q}+\frac{1}{q^{\prime}} > \frac{1}{2}$. 
	By \cite{RPB2}*{Exercise 1.9 and Proposition 5.8}, the map 
	\begin{equation*}
		\left(s, t\right) \mapsto \left\|\delta Y^{\prime}\right\|_{m, q^{\prime}\text{-}var; \left[s, t\right]}^{\frac{1}{z}}\left|\delta X\right|_{p\text{-}var; \left[s, t\right]}^{\frac{1}{z}}+\left\|\mathbb{E}_{\cdot}R^{Y}\right\|_{m,\frac{qq^{\prime}}{q+q^{\prime}}\text{-}var; \left[s, t\right]}^{\frac{1}{z}}
	\end{equation*}
	defines a control. Then by \cite{SSL2}*{Theorem 3.3} with $A:=\delta Y$ (and thus $\mathcal{A}=Y-Y_{0}$), there exists a unique pair of processes $\left(\mathcal{M}, \mathcal{J}\right)$ 
	such that $\mathcal{M} \in \mathscr{M}_{m}$ with $\mathcal{M}_{0}=0$, $\mathcal{J}$ is a measurable adapted $L^{m}$-integrable process satisfying 
	\begin{equation} \label{Je}
		\left\|\delta \mathcal{J}_{s, t}-\mathbb{E}_{s} \delta Y_{s, t}\right\|_{m} \lesssim \left(\left\|\delta Y^{\prime}\right\|_{m, q^{\prime}\text{-}var; \left[s, t\right]}^{\frac{1}{z}}\left|\delta X\right|_{p\text{-}var; \left[s, t\right]}^{\frac{1}{z}}+\left\|\mathbb{E}_{\cdot}R^{Y}\right\|_{m,\frac{qq^{\prime}}{q+q^{\prime}}\text{-}var; \left[s, t\right]}^{\frac{1}{z}}\right)^{z}, 
	\end{equation}
	for every $\left(s, t\right) \in \Delta$, and $Y-Y_{0}=\mathcal{M}+\mathcal{J}$. 
	Define $Y^{M}:=\mathcal{M}$ and $Y^{J}:=Y_{0}+\mathcal{J}$ so that $Y=Y^{M}+Y^{J}$. 
	In view of \eqref{Je}, we have $Y^{J} \in C^{q\text{-}var}L_{m}$ and thus $Y^{M}=Y-Y^{J} \in C^{q\text{-}var}L_{m}$. 
	Since 
	\begin{equation*}
		R^{Y^{J}}_{s, t}=\delta Y^{J}_{s, t}-Y_{s}^{\prime}\delta X_{s, t}=\delta \mathcal{J}_{s, t}-\mathbb{E}_{s} \delta Y_{s, t}+\mathbb{E}_{s}R_{s, t}^{Y}, \quad \forall \left(s, t\right) \in \Delta,
	\end{equation*}
	we have 
	\begin{align*}
		\left\|R^{Y^{J}}_{s, t}\right\|_{m} &\leq \left\|\delta \mathcal{J}_{s, t}-\mathbb{E}_{s} \delta Y_{s, t}\right\|_{m}+\left\|\mathbb{E}_{s}R_{s, t}^{Y}\right\|_{m}\\
		&\lesssim \left(\left\|\delta Y^{\prime}\right\|_{m, q^{\prime}\text{-}var; \left[s, t\right]}^{\frac{1}{z}}\left|\delta X\right|_{p\text{-}var; \left[s, t\right]}^{\frac{1}{z}}+\left\|\mathbb{E}_{\cdot}R^{Y}\right\|_{m,\frac{qq^{\prime}}{q+q^{\prime}}\text{-}var; \left[s, t\right]}^{\frac{1}{z}}\right)^{z}, 
	\end{align*}
	for every $\left(s, t\right) \in \Delta$. 
	By \cite{RPB2}*{Proposition 5.10}, we have $R^{Y^{J}} \in C_{2}^{\frac{qq^{\prime}}{q+q^{\prime}}\text{-}var}L_{m}$ and 
	\begin{align*}
		\left\|R^{Y^{J}}\right\|_{m,\frac{qq^{\prime}}{q+q^{\prime}}\text{-}var} &\lesssim \left(\left\|\delta Y^{\prime}\right\|_{m, q^{\prime}\text{-}var}^{\frac{1}{z}}\left|\delta X\right|_{p\text{-}var}^{\frac{1}{z}}+\left\|\mathbb{E}_{\cdot}R^{Y}\right\|_{m,\frac{qq^{\prime}}{q+q^{\prime}}\text{-}var}^{\frac{1}{z}}\right)^{z}\\
		&\lesssim \left\|\delta Y^{\prime}\right\|_{m,q^{\prime}\text{-}var}\left|\delta X\right|_{p\text{-}var}+\left\|\mathbb{E}_{\cdot}R^{Y}\right\|_{m,\frac{qq^{\prime}}{q+q^{\prime}}\text{-}var}.
	\end{align*}
	Therefore, we have $\left(Y^{J}, Y^{\prime}\right) \in \mathscr{D}_{X}^{\left(q, q^{\prime}\right)\text{-}var}L_{m}$.\\
	\indent
	We now show the uniqueness of $\left(Y^{M}, Y^{J}\right)$. Assume $\left(\bar{Y^{M}}, \bar{Y^{J}}\right)$ is another pair of processes such that 
	$\bar{Y^{M}} \in \mathscr{M}_{m} \cap C^{q\text{-}var}L_{m}$ with $\bar{Y_{0}^{M}}=0$, 
	$\left(\bar{Y^{J}}, Y^{\prime}\right) \in \mathscr{D}_{X}^{\left(q, q^{\prime}\right)\text{-}var}L_{m}$, 
	and $Y=\bar{Y^{M}}+\bar{Y^{J}}$. 
	Put $\Delta Y^{M}:=Y^{M}-\bar{Y^{M}}$ and $\Delta Y^{J}:=Y^{J}-\bar{Y^{J}}$. 
	Then we have $\Delta Y^{M} \in \mathscr{M}_{m}$ with $\Delta Y_{0}^{M}=0$, 
	$\left(\Delta Y^{M}, 0\right)=\left(-\Delta Y^{J}, 0\right) \in \mathscr{D}_{X}^{\left(q, q^{\prime}\right)\text{-}var}L_{m}$ 
	and thus $\Delta Y^{M} \in C^{\frac{qq^{\prime}}{q+q^{\prime}}\text{-}var}L_{m}$. 
	For every partition $\pi=\left\{s=t_{0}<t_{1}<\cdots<t_{N}=t\right\} \in \mathcal{P}\left[0, T\right]$, we have 
	\begin{align*}
		\left\|\Delta Y_{T}^{M}\right\|_{2}^{2}&=\sum_{i=0}^{N-1}\mathbb{E}\left|\delta \Delta Y_{t_{i}, t_{i+1}}^{M}\right|^{2} \leq \left(\sum_{i=0}^{N-1}\left\|\delta \Delta Y_{t_{i}, t_{i+1}}^{M}\right\|_{2}^{\frac{qq^{\prime}}{q+q^{\prime}}}\right)\sup_{i}\left\|\delta \Delta Y_{t_{i}, t_{i+1}}^{M}\right\|_{2}^{2-\frac{qq^{\prime}}{q+q^{\prime}}}\\
		&\leq \left\|\delta \Delta Y^{M}\right\|_{m, \frac{qq^{\prime}}{q+q^{\prime}}\text{-}var}^{\frac{qq^{\prime}}{q+q^{\prime}}}\sup_{i}\left\|\delta \Delta Y_{t_{i}, t_{i+1}}^{M}\right\|_{m}^{2-\frac{qq^{\prime}}{q+q^{\prime}}} \rightarrow 0, \quad as\ \left|\pi\right| \rightarrow 0.
	\end{align*}
	Hence, $\left(Y^{M}, Y^{J}\right)=\left(\bar{Y^{M}}, \bar{Y^{J}}\right)$. 
\end{proof}
The above decomposition result induces an equivalent norm on $\mathbf{D}_{X}^{\left(q, q^{\prime}\right)\text{-}var}L_{m}$. 
\begin{proposition} \label{enorm}
	For any $\left(Y, Y^{\prime}\right) \in \mathbf{D}_{X}^{\left(q, q^{\prime}\right)\text{-}var}L_{m}$, we have 
	\begin{align}
		&\left\|Y, Y^{\prime}\right\|_{\mathbf{D}_{X}^{\left(q, q^{\prime}\right)\text{-}var}L_{m}}^{\left(K\right)}\left(1+K\left|\mathbf{X}\right|_{p\text{-}var}\right)^{-1} \notag\\
		&\quad \lesssim K\left\|Y_{T}\right\|_{m}+K\left\|\delta Y^{M}\right\|_{m,q\text{-}var}+\left\|Y^{\prime}\right\|_{C^{q^{\prime}\text{-}var}L_{m}}+K\left\|R^{Y^{J}}\right\|_{m,\frac{qq^{\prime}}{q+q^{\prime}}\text{-}var} \label{enorme}\\
		&\quad \lesssim \left\|Y, Y^{\prime}\right\|_{\mathbf{D}_{X}^{\left(q, q^{\prime}\right)\text{-}var}L_{m}}^{\left(K\right)}\left(1+K\left|\mathbf{X}\right|_{p\text{-}var}\right). \notag
	\end{align}
\end{proposition}
\begin{proof}
	Since $\left(Y^{J}, Y^{\prime}\right) \in \mathscr{D}_{X}^{\left(q, q^{\prime}\right)\text{-}var}L_{m}$, we have 
	\begin{align*}
		\left\|\delta Y_{s, t}^{J}\right\|_{m} &\leq \left\|Y_{s}^{\prime}\right\|_{m}\left|\delta X_{s, t}\right|+\left\|R_{s, t}^{Y^{J}}\right\|_{m}\\
		&\leq \sup_{r \in \left[0, T\right]}\left\|Y_{r}^{\prime}\right\|_{m}\left|\delta X\right|_{p\text{-}var; \left[s, t\right]}+\left\|R^{Y^{J}}\right\|_{m,\frac{qq^{\prime}}{q+q^{\prime}}\text{-}var; \left[s, t\right]}\\
		&\lesssim \left(\sup_{r \in \left[0, T\right]}\left\|Y_{r}^{\prime}\right\|_{m}^{q}\left|\delta X\right|_{p\text{-}var; \left[s, t\right]}^{q}+\left\|R^{Y^{J}}\right\|_{m,\frac{qq^{\prime}}{q+q^{\prime}}\text{-}var; \left[s, t\right]}^{q}\right)^{\frac{1}{q}}, \quad \forall \left(s, t\right) \in \Delta.
	\end{align*}
	By \cite{RPB2}*{Exercise 1.9 and Proposition 5.8}, the map 
	\begin{equation*}
		\left(s, t\right) \mapsto \sup_{r \in \left[0, T\right]}\left\|Y_{r}^{\prime}\right\|_{m}^{q}\left|\delta X\right|_{p\text{-}var; \left[s, t\right]}^{q}+\left\|R^{Y^{J}}\right\|_{m,\frac{qq^{\prime}}{q+q^{\prime}}\text{-}var; \left[s, t\right]}^{q}
	\end{equation*}
	defines a control. In view of \eqref{RYJe}, by \cite{RPB2}*{Proposition 5.10} we have 
	\begin{align}
		\left\|\delta Y^{J}\right\|_{m,q\text{-}var} &\lesssim \left(\sup_{r \in \left[0, T\right]}\left\|Y_{r}^{\prime}\right\|_{m}^{q}\left|\delta X\right|_{p\text{-}var}^{q}+\left\|R^{Y^{J}}\right\|_{m,\frac{qq^{\prime}}{q+q^{\prime}}\text{-}var}^{q}\right)^{\frac{1}{q}} \notag\\
		&\lesssim \sup_{r \in \left[0, T\right]}\left\|Y_{r}^{\prime}\right\|_{m}\left|\delta X\right|_{p\text{-}var}+\left\|R^{Y^{J}}\right\|_{m,\frac{qq^{\prime}}{q+q^{\prime}}\text{-}var} \label{YJe}\\
		&\lesssim \left\|Y, Y^{\prime}\right\|_{\mathbf{D}_{X}^{\left(q, q^{\prime}\right)\text{-}var}L_{m}}^{\left(K\right)}\left(\frac{1}{K}+\left|\mathbf{X}\right|_{p\text{-}var}\right). \notag
	\end{align}
	Combined with 
	\begin{equation*}
		\left\|\delta Y\right\|_{m,q\text{-}var} \leq \left\|\delta Y^{M}\right\|_{m,q\text{-}var}+\left\|\delta Y^{J}\right\|_{m,q\text{-}var}
	\end{equation*}
	and 
	\begin{equation*}
		\left\|\mathbb{E}_{\cdot}R^{Y}\right\|_{m,\frac{qq^{\prime}}{q+q^{\prime}}\text{-}var}=\left\|\mathbb{E}_{\cdot}R^{Y^{J}}\right\|_{m,\frac{qq^{\prime}}{q+q^{\prime}}\text{-}var} \leq \left\|R^{Y^{J}}\right\|_{m,\frac{qq^{\prime}}{q+q^{\prime}}\text{-}var}, 
	\end{equation*}
	we get the first inequality of \eqref{enorme}. 
	On the other hand, we can derive the last inequality of \eqref{enorme} in a similar way 
	with estimates \eqref{RYJe}, \eqref{YJe} and 
	\begin{equation*}
		\left\|\delta Y^{M}\right\|_{m,q\text{-}var} \leq \left\|\delta Y\right\|_{m,q\text{-}var}+\left\|\delta Y^{J}\right\|_{m,q\text{-}var}.
	\end{equation*}
\end{proof}
In a similar way, we have the following result. 
\begin{proposition} \label{edist}
	For $X, \bar{X} \in C^{p\text{-}var}$, 
	$\left(Y, Y^{\prime}\right) \in \mathbf{D}_{X}^{\left(q, q^{\prime}\right)\text{-}var}L_{m}$ 
	and $\left(\bar{Y}, \bar{Y}^{\prime}\right) \in \mathbf{D}_{\bar{X}}^{\left(q, q^{\prime}\right)\text{-}var}L_{m}$, 
	we have 
	\begin{align}
		&\left\|Y, Y^{\prime};\bar{Y}, \bar{Y}^{\prime}\right\|_{\mathbf{D}_{X, \bar{X}}^{\left(q, q^{\prime}\right)\text{-}var}L_{m}}^{\left(K\right)}\left(1+K\left|\bar{\mathbf{X}}\right|_{p\text{-}var}\right)^{-1} \notag\\
		&\quad \lesssim K\left\|Y^{\prime}\right\|_{C^{q^{\prime}\text{-}var}L_{m}}\rho_{p\text{-}var}\left(\mathbf{X}, \bar{\mathbf{X}}\right)+K\left\|Y_{T}-\bar{Y}_{T}\right\|_{m}+K\left\|\delta Y^{M}-\delta \bar{Y}^{M}\right\|_{m,q\text{-}var} \notag\\
		&\quad \quad+\left\|Y^{\prime}-\bar{Y}^{\prime}\right\|_{C^{q^{\prime}\text{-}var}L_{m}}+K\left\|\mathbb{E}_{\cdot}R^{Y}-\mathbb{E}_{\cdot}R^{\bar{Y}}\right\|_{m,\frac{qq^{\prime}}{q+q^{\prime}}\text{-}var} \label{ediste}\\
		&\quad \lesssim K\left\|Y^{\prime}\right\|_{C^{q^{\prime}\text{-}var}L_{m}}\rho_{p\text{-}var}\left(\mathbf{X}, \bar{\mathbf{X}}\right)+\left\|Y, Y^{\prime};\bar{Y}, \bar{Y}^{\prime}\right\|_{\mathbf{D}_{X, \bar{X}}^{\left(q, q^{\prime}\right)\text{-}var}L_{m}}^{\left(K\right)}\left(1+K\left|\bar{\mathbf{X}}\right|_{p\text{-}var}\right). \notag
	\end{align}
\end{proposition}
\begin{proof}
	Put $\Delta Y:=Y-\bar{Y}$ and similarly for other functions. 
	Analogous to \eqref{RYJe}, we have 
	\begin{align}
		\left\|\Delta R^{Y^{J}}\right\|_{m,\frac{qq^{\prime}}{q+q^{\prime}}\text{-}var} &\lesssim \left\|\delta Y^{\prime}\right\|_{m,q^{\prime}\text{-}var}\left|\delta \Delta X\right|_{p\text{-}var}+\left\|\delta \Delta Y^{\prime}\right\|_{m,q^{\prime}\text{-}var}\left|\delta \bar{X}\right|_{p\text{-}var} \notag\\
		&\quad+\left\|\mathbb{E}_{\cdot}\Delta R^{Y}\right\|_{m,\frac{qq^{\prime}}{q+q^{\prime}}\text{-}var}. \label{DRYJe}
	\end{align}
	Since
	\begin{equation*}
		\delta \Delta Y_{s, t}^{J}=Y_{s}^{\prime}\delta X_{s, t}-\bar{Y}_{s}^{\prime}\delta \bar{X}_{s, t}+\Delta R_{s, t}^{Y^{J}}=Y_{s}^{\prime}\delta \Delta X_{s, t}+\Delta Y_{s}^{\prime}\delta \bar{X}_{s, t}+\Delta R_{s, t}^{Y^{J}}, \quad \forall \left(s, t\right) \in \Delta,
	\end{equation*}
	we have 
	\begin{align*}
		\left\|\delta \Delta Y_{s, t}^{J}\right\|_{m} &\leq \sup_{r \in \left[0, T\right]}\left\|Y_{r}^{\prime}\right\|_{m}\left|\delta \Delta X\right|_{p\text{-}var; \left[s, t\right]}+\sup_{r \in \left[0, T\right]}\left\|\Delta Y_{r}^{\prime}\right\|_{m}\left|\delta \bar{X}\right|_{p\text{-}var; \left[s, t\right]}\\
		&\quad+\left\|\Delta R^{Y^{J}}\right\|_{m,\frac{qq^{\prime}}{q+q^{\prime}}\text{-}var; \left[s, t\right]}, \quad \forall \left(s, t\right) \in \Delta.
	\end{align*}
	In view of \eqref{DRYJe}, we have 
	\begin{align}
		&\left\|\delta \Delta Y^{J}\right\|_{m,q\text{-}var} \notag\\
		&\quad \lesssim \sup_{r \in \left[0, T\right]}\left\|Y_{r}^{\prime}\right\|_{m}\left|\delta \Delta X\right|_{p\text{-}var}+\sup_{r \in \left[0, T\right]}\left\|\Delta Y_{r}^{\prime}\right\|_{m}\left|\delta \bar{X}\right|_{p\text{-}var}+\left\|\Delta R^{Y^{J}}\right\|_{m,\frac{qq^{\prime}}{q+q^{\prime}}\text{-}var} \label{DYJe}\\
		&\quad \lesssim \left\|Y^{\prime}\right\|_{C^{q^{\prime}\text{-}var}L_{m}}\rho_{p\text{-}var}\left(\mathbf{X}, \bar{\mathbf{X}}\right)+\left\|Y, Y^{\prime};\bar{Y}, \bar{Y}^{\prime}\right\|_{\mathbf{D}_{X, \bar{X}}^{\left(q, q^{\prime}\right)\text{-}var}L_{m}}^{\left(K\right)}\left(\frac{1}{K}+\left|\bar{\mathbf{X}}\right|_{p\text{-}var}\right). \notag
	\end{align}
	Combined with 
	\begin{equation*}
		\left\|\delta \Delta Y\right\|_{m,q\text{-}var} \leq \left\|\delta \Delta Y^{M}\right\|_{m,q\text{-}var}+\left\|\delta \Delta Y^{J}\right\|_{m,q\text{-}var}
	\end{equation*}
	and 
	\begin{equation*}
		\left\|\mathbb{E}_{\cdot}\Delta R^{Y}\right\|_{m,\frac{qq^{\prime}}{q+q^{\prime}}\text{-}var}=\left\|\mathbb{E}_{\cdot}\Delta R^{Y^{J}}\right\|_{m,\frac{qq^{\prime}}{q+q^{\prime}}\text{-}var} \leq \left\|\Delta R^{Y^{J}}\right\|_{m,\frac{qq^{\prime}}{q+q^{\prime}}\text{-}var}, 
	\end{equation*}
	we get the first inequality of \eqref{ediste}. 
	On the other hand, we can derive the last inequality of \eqref{ediste} in a similar way 
	with estimates \eqref{DRYJe}, \eqref{DYJe} and 
	\begin{equation*}
		\left\|\delta \Delta Y^{M}\right\|_{m,q\text{-}var} \leq \left\|\delta \Delta Y\right\|_{m,q\text{-}var}+\left\|\delta \Delta Y^{J}\right\|_{m,q\text{-}var}.
	\end{equation*}
\end{proof}
\subsection{$p$-rough stochastic integrals} \label{S4.2}
To establish the $p$-rough stochastic integrals, we need the following two lemmas. 
Lemma \ref{sl} is a specific version of the (deterministic) sewing lemma, 
which is first introduced by Gubinelli \cite{SL1}*{Proposition 1} and highlighted in \cite{SL2}. 
Lemma \ref{ssl} is an extension of the stochastic sewing lemma of Lê \cite{SSL1}*{Theorem 2.1}. 
Both of them can be proved in an analogous manner to the proof of \cite{SSL2}*{Theorem 3.1} 
and \cite{RSDE3}*{Theorem 2.8}. 
\begin{lemma} \label{sl}
	Let $A: \Omega \times \Delta \rightarrow V$ be a measurable adapted $L^{m}$-integrable 
	two-parameter process such that for every $s \in \left[0, T\right]$, 
	the map $t \rightarrow A_{s, t}$ is a.s. continuous on $\left[s, T\right]$. 
	Assume that there exists a constant $z > 1$ and a control $w$ such that 
	\begin{equation*}
		\left\|\sup_{v \in \left[u, t\right]}\left|\delta A_{s, u, v}\right|\right\|_{m} \leq w\left(s, t\right)^{z}, \quad \forall \left(s, u, t\right) \in \Delta_{2}.
	\end{equation*}
	Then there exists a unique continuous adapted $L^{m}$-integrable process 
	$\mathcal{A}: \Omega \times \left[0, T\right] \rightarrow V$ 
	with $\mathcal{A}_{0}=0$ such that 
	\begin{equation*}
		\left\|\delta\mathcal{A}_{s, t}-A_{s, t}\right\|_{m} \lesssim w\left(s, t\right)^{z}, \quad \forall \left(s, t\right) \in \Delta.
	\end{equation*}
	Moreover, we have 
	\begin{equation} \label{rsc}
		\lim_{\pi \in \mathcal{P}\left[0, T\right], |\mathcal{\pi}| \rightarrow 0}\left\|\sup_{t \in \left[0 ,T\right]}\left|\mathcal{A}_{t}-\sum_{\left[u, v\right] \in \mathcal{\pi}, u \leq t}A_{u, v \wedge t}\right|\right\|_{m}=0.
	\end{equation}
\end{lemma}
\begin{lemma} \label{ssl}
	Let $A: \Omega \times \Delta \rightarrow V$ be a measurable adapted $L^{m}$-integrable 
	two-parameter process such that for every $s \in \left[0, T\right]$, 
	the map $t \rightarrow A_{s, t}$ is a.s. continuous on $\left[s, T\right]$. 
	Assume that there exists a constant $z > \frac{1}{2}$ and a control $w$ such that 
	\begin{equation*}
		\left\|\sup_{v \in \left[u, t\right]}\left|\delta A_{s, u, v}\right|\right\|_{m} \leq w\left(s, t\right)^{z}, \quad \mathbb{E}_{s}\delta A_{s, u, t}=0, \quad \forall \left(s, u, t\right) \in \Delta_{2}.
	\end{equation*}
	Then there exists a unique continuous adapted $L^{m}$-integrable process 
	$\mathcal{A}: \Omega \times \left[0, T\right] \rightarrow V$ 
	with $\mathcal{A}_{0}=0$ such that 
	\begin{equation*}
		\left\|\delta\mathcal{A}_{s, t}-A_{s, t}\right\|_{m} \lesssim w\left(s, t\right)^{z}, \quad \mathbb{E}_{s}\left[\delta\mathcal{A}_{s, t}-A_{s, t}\right]=0, \quad \forall \left(s, t\right) \in \Delta. 
	\end{equation*}
	Moreover, \eqref{rsc} also holds. 
\end{lemma}
We now give the following result on $p$-rough stochastic integrals. 
A similar result in Hölder scale can be found in \cite{RSDE3}*{Theorem 3.7}. 
\begin{theorem} \label{rsi}
	Let $\mathbf{X}=\left(X, \mathbb{X}\right) \in \mathscr{C}^{p\text{-}var}\left([0, T], \mathbb{R}^{e}\right)$ 
	and $\left(Y, Y^{\prime}\right) \in \mathbf{D}_{X}^{\left(q, q^{\prime}\right)\text{-}var}L_{m}\left(\left[0, T\right], \Omega; \mathcal{L}\left(\mathbb{R}^{e}, V\right)\right)$ 
	with $\frac{1}{p}+\frac{1}{q}+\frac{1}{q^{\prime}}>1$. 
	Then there exists a continuous adapted $L^{m}$-integrable process 
	$\int_{0}^{\cdot} Y d \mathbf{X}: \Omega \times \left[0, T\right] \rightarrow V$ 
	with the vanishing initial value such that 
	\begin{equation} \label{rsid}
		\lim_{\pi \in \mathcal{P}\left[0, T\right], |\mathcal{\pi}| \rightarrow 0}\left\|\sup_{t \in \left[0, T\right]}\left|\int_{0}^{t} Y d \mathbf{X}-\sum_{\left[u, v\right] \in \mathcal{\pi}, u \leq t}\left(Y_{u}\delta X_{u, v \wedge t}+Y_{u}^{\prime}\mathbb{X}_{u, v \wedge t}\right)\right|\right\|_{m}=0,
	\end{equation}
	and we have 
	\begin{align}
		&\left\|\int_{s}^{t} Y d \mathbf{X}-Y_{s}\delta X_{s, t}-Y_{s}^{\prime}\mathbb{X}_{s, t}\right\|_{m} \notag\\
    	&\quad \lesssim \left(\left\|\delta Y^{M}\right\|_{m,q\text{-}var;\left[s, t\right]}+\left\|R^{Y^{J}}\right\|_{m,\frac{qq^{\prime}}{q+q^{\prime}}\text{-}var;\left[s, t\right]}\right)\left|\delta X\right|_{p\text{-}var;\left[s, t\right]} \label{rsie1}\\
		&\quad \quad+\left\|\delta Y^{\prime}\right\|_{m,q^{\prime}\text{-}var;\left[s, t\right]}\left|\mathbb{X}\right|_{\frac{p}{2}\text{-}var;\left[s, t\right]}, \quad \forall \left(s, t\right) \in \Delta. \notag
	\end{align}
	We call $\int_{0}^{\cdot} Y d \mathbf{X}$ 
    the $p$-rough stochastic integral of $\left(Y, Y^{\prime}\right)$ against $\mathbf{X}$. 
	Moreover, $\left(\int_{0}^{\cdot} Y d \mathbf{X}, Y\right) \in \mathscr{D}_{X}^{\left(p, q\right)\text{-}var}L_{m}\left(\left[0, T\right], \Omega; V\right)$, 
	and for any $K \geq 1$ we have 
	\begin{align*}
		&\left\|\int_{0}^{\cdot} Y d \mathbf{X}, Y\right\|_{\mathbf{D}_{X}^{\left(p, q\right)\text{-}var}L_{m}}^{\left(K\right)}\\
		&\quad \lesssim \left\|Y, Y^{\prime}\right\|_{\mathbf{D}_{X}^{\left(q, q^{\prime}\right)\text{-}var}L_{m}}^{\left(K\right)}\left(\frac{1}{K}+K\left|\mathbf{X}\right|_{p\text{-}var}\right)\left(1+K\left|\mathbf{X}\right|_{p\text{-}var}\right).
	\end{align*}
	As a consequence, the $p$-rough stochastic integral map 
	\begin{equation} \label{rsim}
		\left(Y, Y^{\prime}\right) \mapsto \left(\int_{0}^{\cdot} Y d \mathbf{X}, Y\right)
	\end{equation}
	is a bounded linear map from 
	$\mathbf{D}_{X}^{\left(q, q^{\prime}\right)\text{-}var}L_{m}$ 
	to $\mathscr{D}_{X}^{\left(p, q\right)\text{-}var}L_{m}$. 
\end{theorem}
\begin{proof}
	Define 
	\begin{equation} \label{AMAJ}
		A_{s, t}^{M}:=Y_{s}^{M}\delta X_{s, t}, \quad A_{s, t}^{J}:=Y_{s}^{J}\delta X_{s, t}+Y_{s}^{\prime}\mathbb{X}_{s, t}, \quad \forall \left(s, t\right)\in \Delta.
	\end{equation}
	Clearly, $A^{M}$ and $A^{J}$ are measurable adapted $L^{m}$-integrable 
	two-parameter processes. 
	For every $\left(s, u, t\right) \in \Delta_{2}$, we have 
	$\delta A_{s, u, t}^{M}=-\delta Y_{s, u}^{M}\delta X_{u ,t}$. 
	Then for every $\left(s, u, t\right) \in \Delta_{2}$, we have 
	$\mathbb{E}_{s}\delta A_{s, u, t}^{M}=0$ and 
	\begin{equation*}
		\left\|\sup_{v \in \left[u, t\right]}\left|\delta A_{s, u, v}^{M}\right|\right\|_{m}\leq \left\|\delta Y_{s, u}^{M}\right\|_{m}\sup_{v \in \left[u, t\right]}\left|\delta X_{u, v}\right| \leq \left\|\delta Y^{M}\right\|_{m,q\text{-}var;\left[s, t\right]}\left|\delta X\right|_{p\text{-}var;\left[s, t\right]}.
	\end{equation*}
	By \cite{RPB2}*{Exercise 1.9 and Proposition 5.8}, the map 
	\begin{equation*}
		\left(s, t\right) \mapsto \left\|\delta Y^{M}\right\|_{m,q\text{-}var;\left[s, t\right]}^{\frac{1}{z_{1}}}\left|\delta X\right|_{p\text{-}var;\left[s, t\right]}^{\frac{1}{z_{1}}}
	\end{equation*}
	defines a control, where $z_{1}:=\frac{1}{p}+\frac{1}{q} > \frac{1}{2}$. 
	Then by Lemma \ref{ssl}, 
	there exists a unique continuous adapted $L^{m}$-integrable process 
	$\int_{0}^{\cdot} Y^{M} d \mathbf{X}: \Omega \times \left[0, T\right] \rightarrow V$ 
	with the vanishing initial value such that 
	\begin{equation*}
		\left\|\int_{s}^{t} Y^{M} d \mathbf{X}-A_{s, t}^{M}\right\|_{m} \lesssim \left\|\delta Y^{M}\right\|_{m,q\text{-}var;\left[s, t\right]}\left|\delta X\right|_{p\text{-}var;\left[s, t\right]}, \quad \forall \left(s, t\right) \in \Delta,
	\end{equation*}
	and we have 
	\begin{equation*}
		\lim_{\pi \in \mathcal{P}\left[0, T\right], |\mathcal{\pi}| \rightarrow 0}\left\|\sup_{t \in \left[0, T\right]}\left|\int_{0}^{t} Y^{M} d \mathbf{X}-\sum_{\left[u, v\right] \in \mathcal{\pi}, u \leq t}A_{u, v \wedge t}^{M}\right|\right\|_{m}=0.
	\end{equation*}
	In view of Chen's relation \eqref{Chen}, we have 
	\begin{equation*}
		\delta A_{s, u, t}^{J}=-R_{s, u}^{Y^{J}} \delta X_{u, t}-\delta Y_{s, u}^{\prime} \mathbb{X}_{u, t}, \quad \forall \left(s, u, t\right) \in \Delta_{2}.
	\end{equation*}
	Then for every $\left(s, u, t\right) \in \Delta_{2}$, we have 
	\begin{align*}
		&\left\|\sup_{v \in \left[u, t\right]}\left|\delta A_{s, u, v}^{J}\right|\right\|_{m}\\
		&\quad \leq \left\|R_{s, u}^{Y^{J}}\right\|_{m}\sup_{v \in \left[u, t\right]}\left|\delta X_{u, v}\right|+\left\|\delta Y_{s, u}^{\prime}\right\|_{m}\sup_{v \in \left[u, t\right]}\left|\mathbb{X}_{u, v}\right|\\
		&\quad \leq \left\|R^{Y^{J}}\right\|_{m,\frac{qq^{\prime}}{q+q^{\prime}}\text{-}var;\left[s, t\right]}\left|\delta X\right|_{p\text{-}var;\left[s, t\right]}+\left\|\delta Y^{\prime}\right\|_{m,q^{\prime}\text{-}var;\left[s, t\right]}\left|\mathbb{X}\right|_{\frac{p}{2}\text{-}var;\left[s, t\right]}\\
		&\quad \leq \left(\left\|R^{Y^{J}}\right\|_{m,\frac{qq^{\prime}}{q+q^{\prime}}\text{-}var;\left[s, t\right]}^{\frac{1}{z_{2}}}\left|\delta X\right|_{p\text{-}var;\left[s, t\right]}^{\frac{1}{z_{2}}}+\left\|\delta Y^{\prime}\right\|_{m,q^{\prime}\text{-}var;\left[s, t\right]}^{\frac{1}{z_{2}}}\left|\mathbb{X}\right|_{\frac{p}{2}\text{-}var;\left[s, t\right]}^{\frac{1}{z_{2}}}\right)^{z_{2}},
	\end{align*}
	where $z_{2}:=\frac{1}{p}+\frac{1}{q}+\frac{1}{q^{\prime}} > 1$. By \cite{RPB2}*{Exercise 1.9 and Proposition 5.8}, the map 
	\begin{equation*}
		\left(s, t\right) \mapsto \left\|R^{Y^{J}}\right\|_{m,\frac{qq^{\prime}}{q+q^{\prime}}\text{-}var;\left[s, t\right]}^{\frac{1}{z_{2}}}\left|\delta X\right|_{p\text{-}var;\left[s, t\right]}^{\frac{1}{z_{2}}}+\left\|\delta Y^{\prime}\right\|_{m,q^{\prime}\text{-}var;\left[s, t\right]}^{\frac{1}{z_{2}}}\left|\mathbb{X}\right|_{\frac{p}{2}\text{-}var;\left[s, t\right]}^{\frac{1}{z_{2}}}
	\end{equation*}
	defines a control. Then by Lemma \ref{sl}, 
	there exists a unique continuous adapted $L^{m}$-integrable process 
	$\int_{0}^{\cdot} Y^{J} d \mathbf{X}: \Omega \times \left[0, T\right] \rightarrow V$ 
	with the vanishing initial value such that 
	\begin{align*}
		&\left\|\int_{s}^{t} Y^{J} d \mathbf{X}-A_{s, t}^{J}\right\|_{m}\\
		&\quad \lesssim \left(\left\|R^{Y^{J}}\right\|_{m,\frac{qq^{\prime}}{q+q^{\prime}}\text{-}var;\left[s, t\right]}^{\frac{1}{z_{2}}}\left|\delta X\right|_{p\text{-}var;\left[s, t\right]}^{\frac{1}{z_{2}}}+\left\|\delta Y^{\prime}\right\|_{m,q^{\prime}\text{-}var;\left[s, t\right]}^{\frac{1}{z_{2}}}\left|\mathbb{X}\right|_{\frac{p}{2}\text{-}var;\left[s, t\right]}^{\frac{1}{z_{2}}}\right)^{z_{2}}\\
    	&\quad \lesssim \left\|R^{Y^{J}}\right\|_{m,\frac{qq^{\prime}}{q+q^{\prime}}\text{-}var;\left[s, t\right]}\left|\delta X\right|_{p\text{-}var;\left[s, t\right]}+\left\|\delta Y^{\prime}\right\|_{m,q^{\prime}\text{-}var;\left[s, t\right]}\left|\mathbb{X}\right|_{\frac{p}{2}\text{-}var;\left[s, t\right]}, 
	\end{align*}
	for every $\left(s, t\right) \in \Delta$, and we have 
	\begin{equation*}
		\lim_{\pi \in \mathcal{P}\left[0, T\right], |\mathcal{\pi}| \rightarrow 0}\left\|\sup_{t \in \left[0, T\right]}\left|\int_{0}^{t} Y^{J} d \mathbf{X}-\sum_{\left[u, v\right] \in \mathcal{\pi}, u \leq t}A_{u, v \wedge t}^{J}\right|\right\|_{m}=0.
	\end{equation*}
	Therefore, 
	$\int_{0}^{\cdot} Y d \mathbf{X}:=\int_{0}^{\cdot} Y^{M} d \mathbf{X}+\int_{0}^{\cdot} Y^{J} d \mathbf{X}$ 
	is a continuous adapted $L^{m}$-integrable process with the vanishing initial value 
	satisfying \eqref{rsid}, and the estimate \eqref{rsie1} holds.\\
	\indent
	Note that 
	\begin{equation*}
		R_{s, t}^{\int_{0}^{\cdot} Y d \mathbf{X}}=\int_{s}^{t} Y d \mathbf{X}-Y_{s} \delta X_{s, t}, \quad \forall \left(s, t\right) \in \Delta. 
	\end{equation*}
	In view of \eqref{rsie1} and 
	\begin{equation*}
		\left\|Y_{s}^{\prime}\mathbb{X}_{s, t}\right\|_{m} \leq \left\|Y_{s}^{\prime}\right\|_{m}\left|\mathbb{X}_{s, t}\right| \leq \sup_{r \in \left[0, T\right]}\left\|Y_{r}^{\prime}\right\|_{m}\left|\mathbb{X}\right|_{\frac{p}{2}\text{-}var;\left[s, t\right]}, \quad \forall \left(s, t\right) \in \Delta, 
	\end{equation*}
	we have 
	\begin{align*}
		\left\|R_{s, t}^{\int_{0}^{\cdot} Y d \mathbf{X}}\right\|_{m} &\leq \left\|\int_{s}^{t} Y d \mathbf{X}-Y_{s}\delta X_{s, t}-Y_{s}^{\prime}\mathbb{X}_{s, t}\right\|_{m}+\left\|Y_{s}^{\prime}\mathbb{X}_{s, t}\right\|_{m}\\
		&\lesssim \left(\left\|\delta Y^{M}\right\|_{m,q\text{-}var;\left[s, t\right]}+\left\|R^{Y^{J}}\right\|_{m,\frac{qq^{\prime}}{q+q^{\prime}}\text{-}var;\left[s, t\right]}\right)\left|\delta X\right|_{p\text{-}var;\left[s, t\right]}\\
		&\quad+\left(\sup_{r \in \left[0, T\right]}\left\|Y_{r}^{\prime}\right\|_{m}+\left\|\delta Y^{\prime}\right\|_{m,q^{\prime}\text{-}var}\right)\left|\mathbb{X}\right|_{\frac{p}{2}\text{-}var;\left[s, t\right]}, \quad \forall \left(s, t\right) \in \Delta.
	\end{align*}
	Then by Proposition \ref{enorm}, we have 
	$R^{\int_{0}^{\cdot} Y d \mathbf{X}} \in C_{2}^{\frac{pq}{p+q}\text{-}var}L_{m}$ and 
	\begin{align*}
		\left\|R^{\int_{0}^{\cdot} Y d \mathbf{X}}\right\|_{m,\frac{pq}{p+q}\text{-}var} &\lesssim \left(\left\|\delta Y^{M}\right\|_{m,q\text{-}var}+\left\|R^{Y^{J}}\right\|_{m,\frac{qq^{\prime}}{q+q^{\prime}}\text{-}var}\right)\left|\delta X\right|_{p\text{-}var}\\
		&\quad+\left(\sup_{r \in \left[0, T\right]}\left\|Y_{r}^{\prime}\right\|_{m}+\left\|\delta Y^{\prime}\right\|_{m,q^{\prime}\text{-}var}\right)\left|\mathbb{X}\right|_{\frac{p}{2}\text{-}var}\\
		&\lesssim \left\|Y, Y^{\prime}\right\|_{\mathbf{D}_{X}^{\left(q, q^{\prime}\right)\text{-}var}L_{m}}^{\left(K\right)}\left|\mathbf{X}\right|_{p\text{-}var}\left(1+\left|\mathbf{X}\right|_{p\text{-}var}\right).
	\end{align*}
	Analogous to \eqref{YJe}, we have 
	\begin{align*}
		\left\|\delta \int_{0}^{\cdot} Y d \mathbf{X}\right\|_{m,p\text{-}var} &\lesssim \sup_{r \in \left[0, T\right]}\left\|Y_{r}\right\|_{m}\left|\delta X\right|_{p\text{-}var}+\left\|R^{\int_{0}^{\cdot} Y d \mathbf{X}}\right\|_{m,\frac{pq}{p+q}\text{-}var}\\
		&\lesssim \left\|Y\right\|_{C^{q\text{-}var}L_{m}}\left|\mathbf{X}\right|_{p\text{-}var}+\left\|R^{\int_{0}^{\cdot} Y d \mathbf{X}}\right\|_{m,\frac{pq}{p+q}\text{-}var}.
	\end{align*}
	Therefore, we have 
	$\left(\int_{0}^{\cdot} Y d \mathbf{X}, Y\right) \in \mathscr{D}_{X}^{\left(p, q\right)\text{-}var}L_{m}$ 
	and 
	\begin{align*}
		&\left\|\int_{0}^{\cdot} Y d \mathbf{X}, Y\right\|_{\mathbf{D}_{X}^{\left(p, q\right)\text{-}var}L_{m}}^{\left(K\right)}\\
		&\quad \lesssim K\left\|\delta \int_{0}^{\cdot} Y d \mathbf{X}\right\|_{m,p\text{-}var}+\left\|Y\right\|_{C^{q\text{-}var}L_{m}}+K\left\|R^{\int_{0}^{\cdot} Y d \mathbf{X}}\right\|_{m,\frac{pq}{p+q}\text{-}var}\\
		&\quad \lesssim \left\|Y, Y^{\prime}\right\|_{\mathbf{D}_{X}^{\left(q, q^{\prime}\right)\text{-}var}L_{m}}^{\left(K\right)}\left(\frac{1}{K}+K\left|\mathbf{X}\right|_{p\text{-}var}\right)\left(1+K\left|\mathbf{X}\right|_{p\text{-}var}\right).
	\end{align*}
\end{proof}
\begin{remark}
	In the above proof, we apply Lemmas \ref{ssl} and \ref{sl} to processes 
	$A^{M}$ and $A^{J}$ respectively to obtain 
	$\int_{0}^{\cdot} Y^{M} d \mathbf{X}$ and $\int_{0}^{\cdot} Y^{J} d \mathbf{X}$,  
	and then define 
	$\int_{0}^{\cdot} Y d \mathbf{X}:=\int_{0}^{\cdot} Y^{M} d \mathbf{X}+\int_{0}^{\cdot} Y^{J} d \mathbf{X}$. 
	In fact, we can directly apply \cite{SSL2}*{Theorem 3.1} to the process 
	$A:=A^{M}+A^{Y}$ to obtain $\int_{0}^{\cdot} Y d \mathbf{X}$. 
	However, $\left\|\delta A_{s, u, t}\right\|_{m}$ and 
	$\left\|\mathbb{E}_{s}\delta A_{s, u, t}\right\|_{m}$ are 
	controlled by two different controls, 
	i.e. there exist positive constants $\varepsilon_{1}, \varepsilon_{2}$ 
	and controls $w_{1} \neq w_{2}$ such that 
	\begin{equation*}
		\left\|\delta A_{s, u, t}\right\|_{m} \leq w_{1}\left(s, t\right)^{\frac{1}{2}+\varepsilon_{1}}, \quad \left\|\mathbb{E}_{s}\delta A_{s, u, t}\right\|_{m} \leq w_{2}\left(s, t\right)^{1+\varepsilon_{2}}, \quad \forall \left(s, u ,t\right) \in \Delta_{2}.
	\end{equation*}
	Then we need to replace $w$ with $w_{1}+w_{2}$ in \cite{SSL2}*{Theorem 3.1}. 
	This method is difficult to get the boundedness of $p$-rough stochastic integral map \eqref{rsim}. 
\end{remark}
Then we state the stability of $p$-rough stochastic integrals. 
\begin{theorem} \label{rsistab}
	Let $\mathbf{X}=\left(X, \mathbb{X}\right), \bar{\mathbf{X}}=\left(\bar{X}, \bar{\mathbb{X}}\right) \in \mathscr{C}^{p\text{-}var}$, 
	$\left(Y, Y^{\prime}\right) \in \mathbf{D}_{X}^{\left(q, q^{\prime}\right)\text{-}var}L_{m}$ 
	and $\left(\bar{Y}, \bar{Y}^{\prime}\right) \in \mathbf{D}_{\bar{X}}^{\left(q, q^{\prime}\right)\text{-}var}L_{m}$ 
	with $\frac{1}{p}+\frac{1}{q}+\frac{1}{q^{\prime}}>1$, 
	such that 
	\begin{equation*}
		\left|\mathbf{X}\right|_{p\text{-}var} \vee \left|\bar{\mathbf{X}}\right|_{p\text{-}var} \leq \varepsilon, \quad \left\|\left(Y, Y^{\prime}\right)\right\|_{\mathbf{D}_{X}^{\left(q, q^{\prime}\right)\text{-}var}L_{m}}^{\left(1\right)} \vee \left\|\left(\bar{Y}, \bar{Y}^{\prime}\right)\right\|_{\mathbf{D}_{\bar{X}}^{\left(q, q^{\prime}\right)\text{-}var}L_{m}}^{\left(1\right)} \leq M, 
	\end{equation*}
	for some positive real numbers $\varepsilon$ and $M$. 
	Then for any $K \geq 1$, we have 
	\begin{align*}
		&\left\|\int_{0}^{\cdot} Y d \mathbf{X}, Y; \int_{0}^{\cdot} \bar{Y} d \bar{\mathbf{X}}, \bar{Y}\right\|_{\mathbf{D}_{X, \bar{X}}^{\left(p, q\right)\text{-}var}L_{m}}^{\left(K\right)}\\
		&\quad \lesssim K\left(1+\varepsilon\right)M\rho_{p\text{-}var}\left(\mathbf{X}, \bar{\mathbf{X}}\right)+\left(\frac{1}{K}+K\varepsilon\right)\left(1+K\varepsilon\right)\left\|Y, Y^{\prime};\bar{Y}, \bar{Y}^{\prime}\right\|_{\mathbf{D}_{X, \bar{X}}^{\left(q, q^{\prime}\right)\text{-}var}L_{m}}^{\left(K\right)}.
	\end{align*}
\end{theorem}
\begin{proof}
	Like defining $A^{M}$ and $A^{J}$ by \eqref{AMAJ}, define 
	\begin{equation*}
		\bar{A}_{s, t}^{M}:=\bar{Y}_{s}^{M}\delta \bar{X}_{s, t}, \quad \bar{A}_{s, t}^{J}:=\bar{Y}_{s}^{J}\delta \bar{X}_{s, t}+\bar{Y}_{s}^{\prime}\bar{\mathbb{X}}_{s, t}, \quad \forall \left(s, t\right)\in \Delta.
	\end{equation*}
	Put $\Delta Y:=Y-\bar{Y}$ and similarly for other functions. 
	For every $\left(s, u, t\right) \in \Delta_{2}$, we have 
	\begin{equation*}
		\delta \Delta A_{s, u ,t}^{M}=-\delta Y_{s, u}^{M}\delta X_{u ,t}+\delta \bar{Y}_{s, u}^{M}\delta \bar{X}_{u ,t}=-\delta Y_{s, u}^{M}\delta \Delta X_{u ,t}-\delta \Delta Y_{s, u}^{M}\delta \bar{X}_{u ,t}. 
	\end{equation*}
	Then for every $\left(s, u, t\right) \in \Delta_{2}$, we have $\mathbb{E}_{s}\delta \Delta A_{s, u ,t}^{M}=0$ and 
	\begin{align*}
		\left\|\sup_{v \in \left[u, t\right]}\left|\delta \Delta A_{s, u, v}^{M}\right|\right\|_{m} &\leq \left\|\delta Y^{M}\right\|_{m,q\text{-}var;\left[s, t\right]}\left|\delta \Delta X\right|_{p\text{-}var;\left[s, t\right]}\\
		&\quad+\left\|\delta \Delta Y^{M}\right\|_{m,q\text{-}var;\left[s, t\right]}\left|\delta \bar{X}\right|_{p\text{-}var;\left[s, t\right]}.
	\end{align*}
	By Lemma \ref{ssl}, we have 
	\begin{align*}
		\left\|\Delta \int_{s}^{t} Y^{M} d \mathbf{X}-\Delta A_{s, t}^{M}\right\|_{m} &\lesssim \left\|\delta Y^{M}\right\|_{m,q\text{-}var;\left[s, t\right]}\left|\delta \Delta X\right|_{p\text{-}var;\left[s, t\right]}\\
		&\quad+\left\|\delta \Delta Y^{M}\right\|_{m,q\text{-}var;\left[s, t\right]}\left|\delta \bar{X}\right|_{p\text{-}var;\left[s, t\right]}, \quad \forall \left(s, t\right) \in \Delta. 
	\end{align*}
	In view of Chen's relation \eqref{Chen}, we have 
	\begin{align*}
		\delta \Delta A_{s, u, t}^{J}&=-R_{s, u}^{Y^{J}} \delta X_{u, t}-\delta Y_{s, u}^{\prime} \mathbb{X}_{u, t}+R_{s, u}^{\bar{Y}^{J}} \delta \bar{X}_{u, t}+\delta \bar{Y}_{s, u}^{\prime} \bar{\mathbb{X}}_{u, t}\\
		&=-R_{s, u}^{Y^{J}} \delta \Delta X_{u, t}-\Delta R_{s, u}^{Y^{J}} \delta \bar{X}_{u, t}-\delta Y_{s, u}^{\prime} \Delta \mathbb{X}_{u, t}-\delta \Delta Y_{s, u}^{\prime} \bar{\mathbb{X}}_{u, t}, 
	\end{align*}
	for every $\left(s, u, t\right) \in \Delta_{2}$. 
	Then for every $\left(s, u, t\right) \in \Delta_{2}$, we have 
	\begin{align*}
		&\left\|\sup_{v \in \left[u, t\right]}\left|\delta \Delta A_{s, u, v}^{J}\right|\right\|_{m}\\
    	&\quad \leq \left\|R^{Y^{J}}\right\|_{m,\frac{qq^{\prime}}{q+q^{\prime}}\text{-}var;\left[s, t\right]}\left|\delta \Delta X\right|_{p\text{-}var;\left[s, t\right]}+\left\|\Delta R^{Y^{J}}\right\|_{m,\frac{qq^{\prime}}{q+q^{\prime}}\text{-}var;\left[s, t\right]}\left|\delta \bar{X}\right|_{p\text{-}var;\left[s, t\right]}\\
		&\quad \quad+\left\|\delta Y^{\prime}\right\|_{m,q^{\prime}\text{-}var;\left[s, t\right]}\left|\Delta \mathbb{X}\right|_{\frac{p}{2}\text{-}var;\left[s, t\right]}+\left\|\delta \Delta Y^{\prime}\right\|_{m,q^{\prime}\text{-}var;\left[s, t\right]}\left|\bar{\mathbb{X}}\right|_{\frac{p}{2}\text{-}var;\left[s, t\right]}.
	\end{align*}
	By Lemma \ref{sl}, we have 
	\begin{align*}
		&\left\|\Delta \int_{s}^{t} Y^{J} d \mathbf{X}-\Delta A_{s, t}^{J}\right\|_{m}\\
		&\quad \lesssim \left\|R^{Y^{J}}\right\|_{m,\frac{qq^{\prime}}{q+q^{\prime}}\text{-}var;\left[s, t\right]}\left|\delta \Delta X\right|_{p\text{-}var;\left[s, t\right]}+\left\|\Delta R^{Y^{J}}\right\|_{m,\frac{qq^{\prime}}{q+q^{\prime}}\text{-}var;\left[s, t\right]}\left|\delta \bar{X}\right|_{p\text{-}var;\left[s, t\right]}\\
    	&\quad \quad+\left\|\delta Y^{\prime}\right\|_{m,q^{\prime}\text{-}var;\left[s, t\right]}\left|\Delta \mathbb{X}\right|_{\frac{p}{2}\text{-}var;\left[s, t\right]}+\left\|\delta \Delta Y^{\prime}\right\|_{m,q^{\prime}\text{-}var;\left[s, t\right]}\left|\bar{\mathbb{X}}\right|_{\frac{p}{2}\text{-}var;\left[s, t\right]},
	\end{align*}
	for every $\left(s, t\right) \in \Delta$. 
	Moreover, for every $\left(s, t\right) \in \Delta$ we have 
	\begin{equation*}
		\left\|Y_{s}^{\prime}\mathbb{X}_{s, t}-\bar{Y}_{s}^{\prime}\bar{\mathbb{X}}_{s, t}\right\|_{m} \leq \sup_{r \in \left[0, T\right]}\left\|Y_{r}^{\prime}\right\|_{m}\left|\Delta \mathbb{X}\right|_{\frac{p}{2}\text{-}var;\left[s, t\right]}+\sup_{r \in \left[0, T\right]}\left\|\Delta Y_{r}^{\prime}\right\|_{m}\left|\bar{\mathbb{X}}\right|_{\frac{p}{2}\text{-}var;\left[s, t\right]}.
	\end{equation*}
	Then we have 
	\begin{align*}
		&\left\|\Delta R_{s, t}^{\int_{0}^{\cdot} Y d \mathbf{X}}\right\|_{m}\\
		&\quad \leq \left\|\Delta \int_{s}^{t} Y^{M} d \mathbf{X}-\Delta A_{s, t}^{M}\right\|_{m}+\left\|\Delta \int_{s}^{t} Y^{J} d \mathbf{X}-\Delta A_{s, t}^{J}\right\|_{m}+\left\|Y_{s}^{\prime}\mathbb{X}_{s, t}-\bar{Y}_{s}^{\prime}\bar{\mathbb{X}}_{s, t}\right\|_{m}\\
		&\quad \lesssim \left(\left\|\delta Y^{M}\right\|_{m,q\text{-}var;\left[s, t\right]}+\left\|R^{Y^{J}}\right\|_{m,\frac{qq^{\prime}}{q+q^{\prime}}\text{-}var;\left[s, t\right]}\right)\left|\delta \Delta X\right|_{p\text{-}var;\left[s, t\right]}\\
		&\quad \quad+\left(\sup_{r \in \left[0, T\right]}\left\|Y_{r}^{\prime}\right\|_{m}+\left\|\delta Y^{\prime}\right\|_{m,q^{\prime}\text{-}var}\right)\left|\Delta \mathbb{X}\right|_{\frac{p}{2}\text{-}var;\left[s, t\right]}\\
		&\quad \quad+\left(\left\|\delta \Delta Y^{M}\right\|_{m,q\text{-}var;\left[s, t\right]}+\left\|\Delta R^{Y^{J}}\right\|_{m,\frac{qq^{\prime}}{q+q^{\prime}}\text{-}var;\left[s, t\right]}\right)\left|\delta \bar{X}\right|_{p\text{-}var;\left[s, t\right]}\\
		&\quad \quad+\left(\sup_{r \in \left[0, T\right]}\left\|\Delta Y_{r}^{\prime}\right\|_{m}+\left\|\delta \Delta Y^{\prime}\right\|_{m,q^{\prime}\text{-}var}\right)\left|\bar{\mathbb{X}}\right|_{\frac{p}{2}\text{-}var;\left[s, t\right]}, \quad \forall \left(s, t\right) \in \Delta.
	\end{align*}
	By Propositions \ref{enorm} and \ref{edist}, we have 
	\begin{align*}
		\left\|\Delta R^{\int_{0}^{\cdot} Y d \mathbf{X}}\right\|_{m,\frac{pq}{p+q}\text{-}var} &\lesssim \left(\left\|\delta Y^{M}\right\|_{m,q\text{-}var}+\left\|R^{Y^{J}}\right\|_{m,\frac{qq^{\prime}}{q+q^{\prime}}\text{-}var}\right)\left|\delta \Delta X\right|_{p\text{-}var}\\
		&\quad+\left(\sup_{r \in \left[0, T\right]}\left\|Y_{r}^{\prime}\right\|_{m}+\left\|\delta Y^{\prime}\right\|_{m,q^{\prime}\text{-}var}\right)\left|\Delta \mathbb{X}\right|_{\frac{p}{2}\text{-}var}\\
		&\quad+\left(\left\|\delta \Delta Y^{M}\right\|_{m,q\text{-}var}+\left\|\Delta R^{Y^{J}}\right\|_{m,\frac{qq^{\prime}}{q+q^{\prime}}\text{-}var}\right)\left|\delta \bar{X}\right|_{p\text{-}var}\\
		&\quad+\left(\sup_{r \in \left[0, T\right]}\left\|\Delta Y_{r}^{\prime}\right\|_{m}+\left\|\delta \Delta Y^{\prime}\right\|_{m,q^{\prime}\text{-}var}\right)\left|\bar{\mathbb{X}}\right|_{\frac{p}{2}\text{-}var}\\
		&\lesssim \left(1+\varepsilon\right)\left(M\rho_{p\text{-}var}\left(\mathbf{X}, \bar{\mathbf{X}}\right)+\varepsilon\left\|Y, Y^{\prime};\bar{Y}, \bar{Y}^{\prime}\right\|_{\mathbf{D}_{X, \bar{X}}^{\left(q, q^{\prime}\right)\text{-}var}L_{m}}^{\left(K\right)}\right).
	\end{align*}
	Analogous to \eqref{DYJe}, we have 
	\begin{align*}
		&\left\|\delta \Delta \int_{0}^{\cdot} Y d \mathbf{X}\right\|_{m,p\text{-}var}\\
		&\quad \lesssim \sup_{r \in \left[0, T\right]}\left\|Y_{r}\right\|_{m}\left|\delta \Delta X\right|_{p\text{-}var}+\sup_{r \in \left[0, T\right]}\left\|\Delta Y_{r}\right\|_{m}\left|\delta \bar{X}\right|_{p\text{-}var}+\left\|\Delta R^{\int_{0}^{\cdot} Y d \mathbf{X}}\right\|_{m,\frac{pq}{p+q}\text{-}var}\\
		&\quad \lesssim M\rho_{p\text{-}var}\left(\mathbf{X}, \bar{\mathbf{X}}\right)+\varepsilon\left\|\Delta Y\right\|_{C^{q\text{-}var}L_{m}}+\left\|\Delta R^{\int_{0}^{\cdot} Y d \mathbf{X}}\right\|_{m,\frac{pq}{p+q}\text{-}var}.
	\end{align*}
	Therefore, we have 
	\begin{align*}
		&\left\|\int_{0}^{\cdot} Y d \mathbf{X}, Y; \int_{0}^{\cdot} \bar{Y} d \bar{\mathbf{X}}, \bar{Y}\right\|_{\mathbf{D}_{X, \bar{X}}^{\left(p, q\right)\text{-}var}L_{m}}^{\left(K\right)}\\
		&\quad \lesssim K\left\|\delta \Delta \int_{0}^{\cdot} Y d \mathbf{X}\right\|_{m,p\text{-}var}+\left\|\Delta Y\right\|_{C^{q\text{-}var}L_{m}}+K\left\|\Delta R^{\int_{0}^{\cdot} Y d \mathbf{X}}\right\|_{m,\frac{pq}{p+q}\text{-}var}\\
		&\quad \lesssim K\left(1+\varepsilon\right)M\rho_{p\text{-}var}\left(\mathbf{X}, \bar{\mathbf{X}}\right)+\left(\frac{1}{K}+K\varepsilon\right)\left(1+K\varepsilon\right)\left\|Y, Y^{\prime};\bar{Y}, \bar{Y}^{\prime}\right\|_{\mathbf{D}_{X, \bar{X}}^{\left(q, q^{\prime}\right)\text{-}var}L_{m}}^{\left(K\right)}.
	\end{align*}
\end{proof}
\subsection{Composition of stochastic controlled rough paths} \label{S4.3}
Next, we show that the composition of two stochastic controlled rough paths 
is again a stochastic controlled rough path but of lower integrability.
\begin{proposition} \label{scrplc}
	Let $m_{1}, m_{2} \in \left[1, \infty\right]$, 
	$\left(Y, Y^{\prime}\right) \in \mathbf{D}_{X}^{\left(q, q^{\prime}\right)\text{-}var}L_{m_{1}}\left(\left[0, T\right], \Omega; V\right)$ and 
	$\left(G, G^{\prime}\right) \in \mathbf{D}_{X}^{\left(q, q^{\prime}\right)\text{-}var}L_{m_{2}}\left(\left[0, T\right], \Omega; \mathcal{L}\left(V, \bar{V}\right)\right)$ 
	with $q \leq q^{\prime}$. 
	Then $\left(GY, GY^{\prime}+G^{\prime}Y\right) \in \mathbf{D}_{X}^{\left(q, q^{\prime}\right)\text{-}var}L_{m_{3}}\left(\left[0, T\right], \Omega; \bar{V}\right)$ 
	with $m_{3}:=\frac{m_{1}m_{2}}{m_{1}+m_{2}}$, and for any $K \geq 1$ we have 
	\begin{equation} \label{scrplce}
		\left\|GY, GY^{\prime}+G^{\prime}Y\right\|_{\mathbf{D}_{X}^{\left(q, q^{\prime}\right)\text{-}var}L_{m_{3}}}^{\left(K\right)} \lesssim \left\|G, G^{\prime}\right\|_{\mathbf{D}_{X}^{\left(q, q^{\prime}\right)\text{-}var}L_{m_{2}}}^{\left(1\right)}\left\|Y, Y^{\prime}\right\|_{\mathbf{D}_{X}^{\left(q, q^{\prime}\right)\text{-}var}L_{m_{1}}}^{\left(K\right)}.
	\end{equation}
\end{proposition}
\begin{proof}
	Applying Hölder's inequality, we have 
	\begin{equation*}
		\left\|G_{T}Y_{T}\right\|_{m_{3}} \leq \left\|G_{T}\right\|_{m_{2}}\left\|Y_{T}\right\|_{m_{1}}
	\end{equation*}
	and 
	\begin{align*}
		\left\|\delta \left(GY\right)_{s, t}\right\|_{m_{3}} &\leq \left\|G_{s}\right\|_{m_{2}}\left\|\delta Y_{s, t}\right\|_{m_{1}}+\left\|\delta G_{s, t}\right\|_{m_{2}}\left\|Y_{t}\right\|_{m_{1}}\\
		&\leq \sup_{r \in \left[0, T\right]}\left\|G_{r}^{\prime}\right\|_{m_{2}}\left\|\delta Y\right\|_{m_{1},q\text{-}var;\left[s, t\right]}+\left\|\delta G\right\|_{m_{2},q\text{-}var;\left[s, t\right]}\sup_{r \in \left[0, T\right]}\left\|Y_{r}^{\prime}\right\|_{m_{1}},
	\end{align*}
	for every $\left(s, t\right) \in \Delta$. 
	Then we have $GY \in C^{q\text{-}var}L_{m_{3}}$ and 
	\begin{align*}
		\left\|GY\right\|_{C^{q\text{-}var}L_{m_{3}}} &\lesssim \left\|G_{T}\right\|_{m_{2}}\left\|Y_{T}\right\|_{m_{1}}+\sup_{r \in \left[0, T\right]}\left\|G_{r}^{\prime}\right\|_{m_{2}}\left\|\delta Y\right\|_{m_{1},q\text{-}var}\\
		&\quad+\left\|\delta G\right\|_{m_{2},q\text{-}var}\sup_{r \in \left[0, T\right]}\left\|Y_{r}^{\prime}\right\|_{m_{1}}\\
		&\lesssim \left\|G\right\|_{C^{q\text{-}var}L_{m_{2}}}\left\|Y\right\|_{C^{q\text{-}var}L_{m_{1}}}.
	\end{align*}
	In a similar way, we have $GY^{\prime}+G^{\prime}Y \in C^{q^{\prime}\text{-}var}L_{m_{3}}$ and \
	\begin{align*}
		\left\|GY^{\prime}+G^{\prime}Y\right\|_{C^{q^{\prime}\text{-}var}L_{m_{3}}} &\lesssim \left\|G\right\|_{C^{q\text{-}var}L_{m_{2}}}\left\|Y^{\prime}\right\|_{C^{q^{\prime}\text{-}var}L_{m_{1}}}\\
		&\quad+\left\|G^{\prime}\right\|_{C^{q^{\prime}\text{-}var}L_{m_{2}}}\left\|Y\right\|_{C^{q\text{-}var}L_{m_{1}}}.
	\end{align*}
	For every $\left(s, t\right) \in \Delta$, since 
	\begin{equation*}
		R_{s, t}^{GY}=G_{t}Y_{t}-G_{s}Y_{s}-\left(G_{s}Y_{s}^{\prime}+G_{s}^{\prime}Y_{s}\right)\delta X_{s, t}=\delta G_{s, t}\delta Y_{s, t}+G_{s}R_{s, t}^{Y}+R_{s, t}^{G}Y_{s}, 
	\end{equation*}
	we have 
	\begin{align*}
		\left\|\mathbb{E}_{s}R_{s, t}^{GY}\right\|_{m_{3}} &\leq \left\|\delta G\right\|_{m_{2},q\text{-}var;\left[s, t\right]}\left\|\delta Y\right\|_{m_{1},q\text{-}var;\left[s, t\right]}\\
		&\quad+\left\|\mathbb{E}_{\cdot}R^{G}\right\|_{m_{2},\frac{qq^{\prime}}{q+q^{\prime}}\text{-}var;\left[s, t\right]}\sup_{r \in \left[0, T\right]}\left\|Y_{r}\right\|_{m_{1}}\\
		&\quad+\sup_{r \in \left[0, T\right]}\left\|G_{r}\right\|_{m_{2}}\left\|\mathbb{E}_{\cdot}R^{Y}\right\|_{m_{1},\frac{qq^{\prime}}{q+q^{\prime}}\text{-}var;\left[s, t\right]},
	\end{align*}
	which gives $\mathbb{E}_{\cdot}R^{GY} \in C_{2}^{\frac{qq^{\prime}}{q+q^{\prime}}\text{-}var}L_{m_{3}}$ and 
	\begin{align*}
		\left\|\mathbb{E}_{\cdot}R^{GY}\right\|_{m_{3},\frac{qq^{\prime}}{q+q^{\prime}}\text{-}var} &\lesssim \left\|\delta G\right\|_{m_{2},q\text{-}var}\left\|\delta Y\right\|_{m_{1},q\text{-}var}\\
		&\quad+\left\|\mathbb{E}_{\cdot}R^{G}\right\|_{m_{2},\frac{qq^{\prime}}{q+q^{\prime}}\text{-}var}\sup_{r \in \left[0, T\right]}\left\|Y_{r}\right\|_{m_{1}}\\
		&\quad+\sup_{r \in \left[0, T\right]}\left\|G_{r}\right\|_{m_{2}}\left\|\mathbb{E}_{\cdot}R^{Y}\right\|_{m_{1},\frac{qq^{\prime}}{q+q^{\prime}}\text{-}var}.
	\end{align*}
	Therefore, $\left(GY, GY^{\prime}+G^{\prime}Y\right) \in \mathbf{D}_{X}^{\left(q, q^{\prime}\right)\text{-}var}L_{m_{3}}$ 
	and the estimate \eqref{scrplce} holds. 
\end{proof}
We finish this section with the following  stability of composition, 
whose  proof is analogous to that of estimate \eqref{scrplce}. 
\begin{proposition} \label{lcstab}
	Let $X, \bar{X} \in C^{p\text{-}var}$, 
	$\left(Y, Y^{\prime}\right) \in \mathbf{D}_{X}^{\left(q, q^{\prime}\right)\text{-}var}L_{m_{1}}$, 
	$\left(\bar{Y}, \bar{Y}^{\prime}\right) \in \mathbf{D}_{\bar{X}}^{\left(q, q^{\prime}\right)\text{-}var}L_{m_{1}}$, 
	$\left(G, G^{\prime}\right) \in \mathbf{D}_{X}^{\left(q, q^{\prime}\right)\text{-}var}L_{m_{2}}$ and 
	$\left(\bar{G}, \bar{G}^{\prime}\right) \in \mathbf{D}_{\bar{X}}^{\left(q, q^{\prime}\right)\text{-}var}L_{m_{2}}$ 
	with $q \leq q^{\prime}$ and $m_{1}, m_{2} \in \left[1, \infty\right]$. 
	Define $m_{3}:=\frac{m_{1}m_{2}}{m_{1}+m_{2}}$. 
	Then for any $K \geq 1$, we have 
	\begin{align*}
		&\left\|GY, GY^{\prime}+G^{\prime}Y;\bar{G}\bar{Y}, \bar{G}\bar{Y}^{\prime}+\bar{G}^{\prime}\bar{Y}\right\|_{\mathbf{D}_{X, \bar{X}}^{\left(q, q^{\prime}\right)\text{-}var}L_{m_{3}}}^{\left(K\right)}\\
		&\quad \lesssim \left\|G, G^{\prime}\right\|_{\mathbf{D}_{X}^{\left(q, q^{\prime}\right)\text{-}var}L_{m_{2}}}^{\left(1\right)}\left\|Y, Y^{\prime};\bar{Y}, \bar{Y}^{\prime}\right\|_{\mathbf{D}_{X, \bar{X}}^{\left(q, q^{\prime}\right)\text{-}var}L_{m_{1}}}^{\left(K\right)}\\
		&\quad \quad+\left\|G, G^{\prime};\bar{G}, \bar{G}^{\prime}\right\|_{\mathbf{D}_{X, \bar{X}}^{\left(q, q^{\prime}\right)\text{-}var}L_{m_{2}}}^{\left(1\right)}\left\|Y, Y^{\prime}\right\|_{\mathbf{D}_{X}^{\left(q, q^{\prime}\right)\text{-}var}L_{m_{1}}}^{\left(K\right)}.
	\end{align*}
\end{proposition}
\begin{remark}
	It is crucial to use the norm \eqref{norm} instead of the equivalent norm defined in \eqref{enorme} 
	to investigate the composition of stochastic controlled rough paths. 
	Indeed, the definition of the equivalent norm depends on the decomposition in Proposition \ref{deco}, 
	but it is difficult to obtain the explicit expression of the decomposition of $\left(GY, GY^{\prime}+G^{\prime}Y\right)$ 
	from the decompositions of $\left(Y, Y^{\prime}\right) \in \mathbf{D}_{X}^{\left(q, q^{\prime}\right)\text{-}var}L_{m_{1}}$ 
	and $\left(G, G^{\prime}\right) \in \mathbf{D}_{X}^{\left(q, q^{\prime}\right)\text{-}var}L_{m_{2}}$, 
	unless $\left(G, G^{\prime}\right) \in \mathscr{D}_{X}^{\left(q, q^{\prime}\right)\text{-}var}L_{m_{2}}$.\\
	\indent
	In the original version of the paper, we  used the equivalent norm induced by the decomposition, 
	and assumed that the coefficient $\left(G, G^{\prime}\right)$ of BSDEs with a linear rough drift \eqref{inrbsde3} 
	belongs to $\mathscr{D}_{X}^{\left(p, p\right)\text{-}var}L_{\infty}$. 
	We thank the anonymous referee for driving  us  to extend 
	the coefficient $\left(G, G^{\prime}\right)$ from $\mathscr{D}_{X}^{\left(p, p\right)\text{-}var}L_{\infty}$ 
	to $\mathbf{D}_{X}^{\left(p, p\right)\text{-}var}L_{\infty}$. 
\end{remark}

\section{BSDEs with a linear rough drift} \label{S5}
Consider the following BSDE with a linear rough drift  
\begin{equation} \label{rbsde}
	Y_{t}=\xi+\int_{t}^{T} f_{r}\left(Y_{r}, Z_{r}\right) d r+\int_{t}^{T} \left(G_{r}Y_{r}+H_{r}\right) d \mathbf{X}_{r}-\int_{t}^{T} Z_{r} d W_{r}, \quad t \in \left[0, T\right].
\end{equation}
Here, $\xi$ is an $\mathbb{R}^{k}$-valued random variable, 
$\mathbf{X}=\left(X, \mathbb{X}\right) \in \mathscr{C}^{p\text{-}var}\left([0, T], \mathbb{R}^{e}\right)$ 
and $\left(Y, Z\right)$ is the pair of unknown processes. 
$f: \Omega \times \left[0, T\right] \times \mathbb{R}^{k} \times \mathbb{R}^{k \times d} \rightarrow \mathbb{R}^{k}$ 
is progressively measurable. 
$G: \Omega \times \left[0, T\right] \rightarrow \mathcal{L}\left(\mathbb{R}^{k}, \mathbb{R}^{k \times e}\right)$, 
$G^{\prime}: \Omega \times \left[0, T\right] \rightarrow \mathcal{L}\left(\mathbb{R}^{k}, \mathcal{L}\left(\mathbb{R}^{e}, \mathbb{R}^{k \times e}\right)\right)$ 
(which is implied in \eqref{rbsde}), 
$H: \Omega \times \left[0, T\right] \rightarrow \mathbb{R}^{k \times e}$ and 
$H^{\prime}: \Omega \times \left[0, T\right] \rightarrow \mathcal{L}\left(\mathbb{R}^{e}, \mathbb{R}^{k \times e}\right)$ 
(which is also implied in \eqref{rbsde}) are measurable adapted processes. 
The following definition of solutions is inspired 
by \cite{RSDE3}*{Definition 4.2 and Proposition 4.3}. 
\begin{definition} \label{solud} 
	We call $\left(Y, Z\right)$ a solution to \eqref{rbsde} if the following are satisfied 
	\begin{enumerate}[(i)]
		\item $Y \in \mathbb{L}^{2}\left(\left[0, T\right], \Omega; \mathbb{R}^{k}\right)$ and 
		$Z \in \mathbb{L}^{2}\left(\left[0, T\right], \Omega; \mathbb{R}^{k \times d}\right)$;
		\item $\int_{0}^{T}\left|f_{r}\left(Y_{r}, Z_{r}\right)\right| d r$ is finite a.s.;
		\item $\left(GY+H, G\left(GY+H\right)+G^{\prime}Y+H^{\prime}\right) \in \mathbf{D}_{X}^{\left(q, q^{\prime}\right)\text{-}var}L_{2}\left(\left[0, T\right], \Omega; \mathbb{R}^{k \times e}\right)$ 
		with $q, q^{\prime} \in \left[p, \infty\right)$ such that 
		$\frac{1}{p}+\frac{1}{q}+\frac{1}{q^{\prime}}>1$;
		\item writing $\int_{0}^{\cdot} \left(GY+H\right) d \mathbf{X}$ the $p$-rough stochastic integral of 
		$(GY+H, G(GY+H)+G^{\prime}Y+H^{\prime})$ against $\mathbf{X}$, 
		we have 
		\begin{equation*} 
			Y_{t}=\xi+\int_{t}^{T} f_{r}\left(Y_{r}, Z_{r}\right) d r+\int_{t}^{T} \left(GY+H\right) d \mathbf{X}-\int_{t}^{T} Z_{r} d W_{r}, \quad \forall t \in \left[0, T\right].
		\end{equation*}
	\end{enumerate}
\end{definition}
We introduce the following assumption. 
\begin{assumption} \label{assu}
	~
	\begin{enumerate}[(i)]
		\item $\xi \in L^{2}\left(\Omega, \mathcal{F}_{T}; \mathbb{R}^{k}\right)$ and 
		$f\left(0, 0\right) \in \mathbb{L}^{2}\left(\left[0, T\right], \Omega; \mathbb{R}^{k}\right)$;
		\item $f$ is uniformly Lipschitz continuous in $y$ and $z$, 
		i.e. for every $t \in \left[0, T\right]$, $y_{1}, y_{2} \in \mathbb{R}^{k}$ and $z_{1}, z_{2} \in \mathbb{R}^{k \times d}$, we have 
		\begin{equation*}
			\left|f_{t}\left(y_{1}, z_{1}\right)-f_{t}\left(y_{2}, z_{2}\right)\right| \lesssim \left|y_{1}-y_{2}\right|+\left|z_{1}-z_{2}\right|;
		\end{equation*}
		\item $\mathbf{X}=\left(X, \mathbb{X}\right) \in \mathscr{C}^{p\text{-}var}\left([0, T], \mathbb{R}^{e}\right)$;
		\item $\left(G, G^{\prime}\right) \in \mathbf{D}_{X}^{\left(p, p\right)\text{-}var}L_{\infty}\left(\left[0, T\right], \Omega; \mathcal{L}\left(\mathbb{R}^{k}, \mathbb{R}^{k \times e}\right)\right)$;
		\item $\left(H, H^{\prime}\right) \in \mathbf{D}_{X}^{\left(p, p\right)\text{-}var}L_{2}\left(\left[0, T\right], \Omega; \mathbb{R}^{k \times e}\right)$.
	\end{enumerate}
\end{assumption}
\subsection{Existence and Uniqueness} \label{S5.1}
We first give the following result on solutions to rough BSDE \eqref{rbsde}. 
\begin{proposition} \label{ape} 
	Let Assumption \ref{assu} hold and 
	$\left(Y, Z\right)$ be a solution to rough BSDE \eqref{rbsde}. 
	Then $\left(Y, GY+H\right) \in \mathbf{D}_{X}^{\left(p, p\right)\text{-}var}L_{2}$. 
\end{proposition}
\begin{proof}
	Since 
	\begin{align}
		\left\|f\left(Y, Z\right)\right\|_{\mathbb{L}^{2}} &\lesssim \left(\mathbb{E}\int_{0}^{T} \left(\left|f_{r}\left(0, 0\right)\right|^{2}+\left|Y_{r}\right|^{2}+\left|Z_{r}\right|^{2}\right) d r\right)^{\frac{1}{2}} \notag\\
		&\lesssim \left\|f\left(0, 0\right)\right\|_{\mathbb{L}^{2}}+\left\|Y\right\|_{\mathbb{L}^{2}}+\left\|Z\right\|_{\mathbb{L}^{2}}, \label{finL2}
	\end{align}
	by Lemma \ref{L2scrp}, we have $\left(\int_{0}^{\cdot} f_{r}\left(Y_{r}, Z_{r}\right) d r, 0\right) \in \mathscr{D}_{X}^{\left(p, p\right)\text{-}var}L_{2}$. 
	Since $Z \in \mathbb{L}^{2}$, by Lemma \ref{mtg} we have 
	$\left(\int_{0}^{\cdot} Z_{r} d W_{r}, 0\right) \in \mathbf{D}_{X}^{\left(p, p\right)\text{-}var}L_{2}$. 
	By Theorem \ref{rsi}, $\left(\int_{0}^{\cdot} \left(GY+H\right) d \mathbf{X}, GY+H\right) \in \mathscr{D}_{X}^{\left(p, q\right)\text{-}var}L_{2}$. 
	Hence, $\left(Y, GY+H\right) \in \mathbf{D}_{X}^{\left(p, q\right)\text{-}var}L_{2}$. 
	By Proposition \ref{scrplc}, we have $\left(GY, G\left(GY+H\right)+G^{\prime}Y\right) \in \mathbf{D}_{X}^{\left(p, q\right)\text{-}var}L_{2}$. 
	Then by Theorem \ref{rsi}, we have 
	$\left(\int_{0}^{\cdot} \left(GY+H\right) d \mathbf{X}, GY+H\right) \in \mathscr{D}_{X}^{\left(p, p\right)\text{-}var}L_{2}$. 
	Therefore, $\left(Y, GY+H\right) \in \mathbf{D}_{X}^{\left(p, p\right)\text{-}var}L_{2}$. 
\end{proof}
We now give our main theorem. 
\begin{theorem} \label{solueu}
	Under Assumption \ref{assu}, the rough BSDE \eqref{rbsde} has a unique solution 
	$\left(Y, Z\right)$, and we have 
	\begin{align}
		&\left\|Y, GY+H\right\|_{\mathbf{D}_{X}^{\left(p, p\right)\text{-}var}L_{2}}^{\left(1\right)}+\left\|Z\right\|_{\mathbb{L}^{2}} \notag\\
		&\quad \lesssim_{T, \left|\mathbf{X}\right|_{p\text{-}var}, \left\|G, G^{\prime}\right\|_{\mathbf{D}_{X}^{\left(p, p\right)\text{-}var}L_{\infty}}^{\left(1\right)}} \left\|\xi\right\|_{2}+\left\|f\left(0, 0\right)\right\|_{\mathbb{L}^{2}}+\left\|H, H^{\prime}\right\|_{\mathbf{D}_{X}^{\left(p, p\right)\text{-}var}L_{2}}^{\left(1\right)}. \label{solue}
	\end{align}
\end{theorem}
\begin{proof}
	Let $K \geq 1$ and $\varepsilon \in \left(0, \frac{1}{K}\right]$ be constants 
	which are waiting to be determined later. 
	We first show the existence and uniqueness for 
	$T \vee \left|\mathbf{X}\right|_{p\text{-}var} \leq \varepsilon$. 
	Define 
	\begin{equation*}
		\mathcal{S}:=\mathbf{D}_{X}^{\left(p, p\right)\text{-}var}L_{2} \times \mathbb{L}^{2}, \quad \left\|Y, Y^{\prime}, Z\right\|^{\left(K\right)}:=\left\|Y, Y^{\prime}\right\|_{\mathbf{D}_{X}^{\left(p, p\right)\text{-}var}L_{2}}^{\left(K\right)}+K\left\|Z\right\|_{\mathbb{L}^{2}}.
	\end{equation*}
	Then the normed vector space $\left(\mathcal{S}, \left\|\cdot\right\|^{\left(K\right)}\right)$ 
	is a Banach space.\\
	\indent
	For $\left(Y, Y^{\prime}, Z\right) \in \mathcal{S}$, by Proposition \ref{scrplc} we have 
	$\left(GY, GY^{\prime}+G^{\prime}Y\right) \in \mathbf{D}_{X}^{\left(p, p\right)\text{-}var}L_{2}$. 
	By Theorem \ref{rsi}, we can define the $p$-rough stochastic integral $\int_{0}^{\cdot} \left(GY+H\right) d \mathbf{X}$ 
	of $\left(GY+H, GY^{\prime}+G^{\prime}Y+H^{\prime}\right)$ against $\mathbf{X}$, 
	and we have $\left(\int_{0}^{\cdot} \left(GY+H\right) d \mathbf{X}, GY+H\right) \in \mathscr{D}_{X}^{\left(p, p\right)\text{-}var}L_{2}$. 
	Since $Y \in C^{p\text{-}var}L_{2} \subset \mathbb{L}^{2}$, 
	in view of \eqref{finL2} we have $f\left(Y, Z\right) \in \mathbb{L}^{2}$. 
	Then we have $\xi+\int_{0}^{T} f_{r}\left(Y_{r}, Z_{r}\right) d r+\int_{0}^{T} \left(GY+H\right) d \mathbf{X} \in L^{2}\left(\Omega, \mathcal{F}_{T}; \mathbb{R}^{k}\right)$, 
	and thus $M: \Omega \times \left[0, T\right] \rightarrow \mathbb{R}^{k}$ 
	defined by 
	\begin{equation} \label{Mdenf}
		M_{t}:=\mathbb{E}_{t}\left[\xi+\int_{0}^{T} f_{r}\left(Y_{r}, Z_{r}\right) d r+\int_{0}^{T} \left(GY+H\right) d \mathbf{X}\right], \quad \forall t \in \left[0, T\right]
	\end{equation}
	is a square-integrable martingale. 
	Applying the martingale representation theorem, there exists unique 
	$\Phi^{Z} \in \mathbb{L}^{2}$ such that 
	\begin{equation*}
		M_{t}=M_{0}+\int_{0}^{t} \Phi_{r}^{Z} d W_{r}, \quad \forall t \in \left[0, T\right]. 
	\end{equation*}
	By Lemma \ref{mtg}, we have $\left(\int_{0}^{\cdot} \Phi_{r}^{Z} d W_{r}, 0\right) \in \mathbf{D}_{X}^{\left(p, p\right)\text{-}var}L_{2}$.
	Define 
	\begin{equation*}
		\Phi^{Y}:=\xi+\int_{\cdot}^{T} f_{r}\left(Y_{r}, Z_{r}\right) d r+\int_{\cdot}^{T} \left(GY+H\right) d \mathbf{X}-\int_{\cdot}^{T} \Phi_{r}^{Z} d W_{r}, \quad \Phi^{Y^{\prime}}:=GY+H.
	\end{equation*}
	Then we have 
	\begin{align*}
		\Phi_{t}^{Y}&=\xi+\int_{t}^{T} f_{r}\left(Y_{r}, Z_{r}\right) d r+\int_{t}^{T} \left(GY+H\right) d \mathbf{X}-\left(M_{T}-M_{0}-\int_{0}^{t} \Phi_{r}^{Z} d W_{r}\right)\\
		&=M_{0}-\int_{0}^{t} f_{r}\left(Y_{r}, Z_{r}\right) d r-\int_{0}^{t} \left(GY+H\right) d \mathbf{X}+\int_{0}^{t} \Phi_{r}^{Z} d W_{r}, \quad \forall t \in \left[0, T\right],
	\end{align*}
	and thus $\left(\Phi^{Y}, \Phi^{Y^{\prime}}\right) \in \mathbf{D}_{X}^{\left(p, p\right)\text{-}var}L_{2}$, 
	Hence, we can establish the map 
	\begin{equation} \label{Phidenf}
		\Phi: \mathcal{S} \rightarrow \mathcal{S}, \quad \left(Y, Y^{\prime}, Z\right) \mapsto \left(\Phi^{Y}, \Phi^{Y^{\prime}}, \Phi^{Z}\right). 
	\end{equation}
	\indent
	We now show that $\Phi$ is a contraction map. For any other 
	$\left(\bar{Y}, \bar{Y}^{\prime}, \bar{Z}\right) \in \mathcal{S}$, 
	we define $\bar{M}$ like defining $M$ by \eqref{Mdenf}. 
	Put $\Delta Y:=Y-\bar{Y}$ and similarly for other functions. 
	Since  
	\begin{equation} \label{cone1}
		\left\|\Delta f\left(Y, Z\right)\right\|_{\mathbb{L}^{2}} \lesssim \left(\mathbb{E}\int_{0}^{T} \left|\Delta Y_{r}\right|^{2}+\left|\Delta Z_{r}\right|^{2}d r\right)^{\frac{1}{2}} \lesssim \sup_{r \in \left[0, T\right]}\left\|\Delta Y_{r}\right\|_{2}+\left\|\Delta Z\right\|_{\mathbb{L}^{2}},
	\end{equation}
	by Lemma \ref{L2scrp}, we have 
	\begin{equation} \label{cone2}
		\left\|\int_{0}^{\cdot} \Delta f_{r}\left(Y_{r}, Z_{r}\right) d r, 0\right\|_{\mathbf{D}_{X}^{\left(p, p\right)\text{-}var}L_{2}}^{\left(K\right)} \lesssim KT^{\frac{1}{2}}\left(\sup_{r \in \left[0, T\right]}\left\|\Delta Y_{r}\right\|_{2}+\left\|\Delta Z\right\|_{\mathbb{L}^{2}}\right). 
	\end{equation}
	By Proposition \ref{scrplc}, we have 
	\begin{equation*}
		\left\|G\Delta Y, G\Delta Y^{\prime}+G^{\prime}\Delta Y\right\|_{\mathbf{D}_{X}^{\left(p, p\right)\text{-}var}L_{2}}^{\left(K\right)} \lesssim \left\|G, G^{\prime}\right\|_{\mathbf{D}_{X}^{\left(p, p\right)\text{-}var}L_{\infty}}^{\left(1\right)}\left\|\Delta Y, \Delta Y^{\prime}\right\|_{\mathbf{D}_{X}^{\left(p, p\right)\text{-}var}L_{2}}^{\left(K\right)}.
	\end{equation*}
	Then by Theorem \ref{rsi}, we have 
	\begin{align}
		&\left\|\int_{0}^{\cdot} G\Delta Y d \mathbf{X}, G\Delta Y\right\|_{\mathbf{D}_{X}^{\left(p, p\right)\text{-}var}L_{2}}^{\left(K\right)} \notag\\
		&\quad \lesssim \left\|G\Delta Y, G\Delta Y^{\prime}+G^{\prime}\Delta Y\right\|_{\mathbf{D}_{X}^{\left(p, p\right)\text{-}var}L_{2}}^{\left(K\right)}\left(\frac{1}{K}+K\left|\mathbf{X}\right|_{p\text{-}var}\right)\left(1+K\left|\mathbf{X}\right|_{p\text{-}var}\right) \label{cone3}\\
		&\quad \lesssim \left\|G, G^{\prime}\right\|_{\mathbf{D}_{X}^{\left(p, p\right)\text{-}var}L_{\infty}}^{\left(1\right)}\left\|\Delta Y, \Delta Y^{\prime}\right\|_{\mathbf{D}_{X}^{\left(p, p\right)\text{-}var}L_{2}}^{\left(K\right)}\left(\frac{1}{K}+K\left|\mathbf{X}\right|_{p\text{-}var}\right), \notag
	\end{align}
	and thus 
	\begin{align}
		&\left\|\int_{0}^{T} G\Delta Y d \mathbf{X}\right\|_{2} \notag\\
		&\quad \leq \frac{1}{K}\left\|\int_{0}^{\cdot} G\Delta Y d \mathbf{X}, G\Delta Y\right\|_{\mathbf{D}_{X}^{\left(p, p\right)\text{-}var}L_{2}}^{\left(K\right)} \label{cone4}\\
		&\quad \lesssim \frac{1}{K}\left\|G, G^{\prime}\right\|_{\mathbf{D}_{X}^{\left(p, p\right)\text{-}var}L_{\infty}}^{\left(1\right)}\left\|\Delta Y, \Delta Y^{\prime}\right\|_{\mathbf{D}_{X}^{\left(p, p\right)\text{-}var}L_{2}}^{\left(K\right)}\left(\frac{1}{K}+K\left|\mathbf{X}\right|_{p\text{-}var}\right). \notag
	\end{align}
	Since 
	\begin{equation*}
		\Delta M_{t}=\mathbb{E}_{t}\left[\int_{0}^{T} \Delta f_{r}\left(Y_{r}, Z_{r}\right) d r+\int_{0}^{T} \left(G\Delta Y\right) d \mathbf{X}\right]=\Delta M_{0}+\int_{0}^{t} \Delta \Phi_{r}^{Z} d W_{r},
	\end{equation*}
	for every $t \in \left[0, T\right]$, we have 
	\begin{equation*}
		\left\|\Delta \Phi^{Z}\right\|_{\mathbb{L}^{2}} \leq 2\left\|\Delta M_{T}\right\|_{2} \lesssim T^{\frac{1}{2}}\left\|\Delta f(Y, Z)\right\|_{\mathbb{L}^{2}}+\left\|\int_{0}^{T} G\Delta Y d \mathbf{X}\right\|_{2}. 
	\end{equation*}
	Combining \eqref{cone1}, \eqref{cone4} and the last inequality, we have 
	\begin{equation} \label{cone5}
		\left\|\Delta \Phi^{Z}\right\|_{\mathbb{L}^{2}} \lesssim \frac{1}{K}\left(\frac{1}{K}+K\varepsilon\right)\left(1+\left\|G, G^{\prime}\right\|_{\mathbf{D}_{X}^{\left(p, p\right)\text{-}var}L_{\infty}}^{\left(1\right)}\right)\left\|\Delta Y, \Delta Y^{\prime}, \Delta Z\right\|^{\left(K\right)}.
	\end{equation}
	By Lemma \ref{mtg}, we have 
	\begin{align}
		&\left\|\int_{0}^{\cdot} \Delta \Phi_{r}^{Z} d W_{r}, 0\right\|_{\mathbf{D}_{X}^{\left(p, p\right)\text{-}var}L_{2}}^{\left(K\right)} \notag\\
		&\quad \lesssim K\left\|\Delta \Phi^{Z}\right\|_{\mathbb{L}^{2}} \lesssim \left(\frac{1}{K}+K\varepsilon\right)\left(1+\left\|G, G^{\prime}\right\|_{\mathbf{D}_{X}^{\left(p, p\right)\text{-}var}L_{\infty}}^{\left(1\right)}\right)\left\|\Delta Y, \Delta Y^{\prime}, \Delta Z\right\|^{\left(K\right)}. \label{cone6}
	\end{align}
	Note that 
	\begin{equation*}
		\Delta \Phi^{Y}=\int_{\cdot}^{T} \Delta f_{r}\left(Y_{r}, Z_{r}\right) d r+\int_{\cdot}^{T} G\Delta Y d \mathbf{X}-\int_{\cdot}^{T} \Delta \Phi_{r}^{Z} d W_{r} , \quad \Delta \Phi^{Y^{\prime}}=G\Delta Y.
	\end{equation*}
	Combining \eqref{cone2}, \eqref{cone3} and \eqref{cone6}, we have 
	\begin{align*}
		&\left\|\Delta \Phi^{Y}, \Delta \Phi^{Y^{\prime}}\right\|_{\mathbf{D}_{X}^{\left(p, p\right)\text{-}var}L_{2}}^{\left(K\right)}\\
		&\quad \lesssim \left(\frac{1}{K}+K\varepsilon\right)\left(1+\left\|G, G^{\prime}\right\|_{\mathbf{D}_{X}^{\left(p, p\right)\text{-}var}L_{\infty}}^{\left(1\right)}\right)\left\|\Delta Y, \Delta Y^{\prime}, \Delta Z\right\|^{\left(K\right)}.
	\end{align*}
	Then we have 
	\begin{align*}
		&\left\|\Delta \Phi^{Y}, \Delta \Phi^{Y^{\prime}}, \Delta \Phi^{Z}\right\|^{\left(K\right)}\\
		&\quad \leq C\left(\frac{1}{K}+K\varepsilon\right)\left(1+\left\|G, G^{\prime}\right\|_{\mathbf{D}_{X}^{\left(p, p\right)\text{-}var}L_{\infty}}^{\left(1\right)}\right)\left\|\Delta Y, \Delta Y^{\prime}, \Delta Z\right\|^{\left(K\right)}.
	\end{align*}
	Taking 
	\begin{equation*}
		K:=4\left(1+C\right)\left(1+\left\|G, G^{\prime}\right\|_{\mathbf{D}_{X}^{\left(p, p\right)\text{-}var}L_{\infty}}^{\left(1\right)}\right), \quad \varepsilon:=\frac{1}{K^{2}}, 
	\end{equation*}
	for the above $C$, we have 
	\begin{equation*}
		\left\|\Delta \Phi^{Y}, \Delta \Phi^{Y^{\prime}}, \Delta \Phi^{Z}\right\|^{\left(K\right)} \leq \frac{1}{2}\left\|\Delta Y, \Delta Y^{\prime}, \Delta Z\right\|^{\left(K\right)},
	\end{equation*}
	and thus $\Phi$ is a contraction map.\\
	\indent
	Applying the Banach fixed-point theorem, 
	$\Phi$ has a unique fixed point $\left(Y, Y^{\prime}, Z\right)$ in $\mathcal{S}$, 
	and thus $\left(Y, Z\right)$ is a solution to \eqref{rbsde}. 
	To show the uniqueness, assume $\left(\bar{Y}, \bar{Z}\right)$ is another solution. 
	By Proposition \ref{ape}, 
	$\left(\bar{Y}, G\bar{Y}+H\right) \in \mathbf{D}_{X}^{\left(p, p\right)\text{-}var}L_{2}$ 
	and thus $\left(\bar{Y}, G\bar{Y}+H, \bar{Z}\right)$ is also 
	a fixed point of $\Phi$ in $\mathcal{S}$. 
	Therefore, $\left(\bar{Y}, \bar{Z}\right)=\left(Y, Z\right)$.\\
	\indent
	We now show the estimate \eqref{solue} for $T \vee \left|\mathbf{X}\right|_{p\text{-}var} \leq \varepsilon$. 
	Taking $\left(\bar{Y}, \bar{Y}^{\prime}, \bar{Z}\right)=\left(0, 0, 0\right)$, we have 
	\begin{equation*}
		\bar{M}_{t}=\mathbb{E}_{t}\left[\xi+\int_{0}^{T} f_{r}\left(0, 0\right) d r+\int_{0}^{T} H d \mathbf{X}\right]=\bar{M}_{0}+\int_{0}^{t} \Phi_{r}^{\bar{Z}} d W_{r}, \quad \forall t \in \left[0, T\right],
	\end{equation*}
	\begin{equation*}
		\Phi^{\bar{Y}}=\xi+\int_{\cdot}^{T} f_{r}\left(0, 0\right) d r+\int_{\cdot}^{T} H d \mathbf{X}-\int_{\cdot}^{T} \Phi_{r}^{\bar{Z}} d W_{r},
	\end{equation*}
	and $\Phi^{\bar{Y}^{\prime}}=H$. 
	Analogous to the proof of contraction of the map $\Phi$, we have 
	\begin{equation*}
		\left\|\Phi^{\bar{Y}}, \Phi^{\bar{Y}^{\prime}}, \Phi^{\bar{Z}}\right\|^{\left(K\right)} \lesssim_{\left\|G, G^{\prime}\right\|_{\mathbf{D}_{X}^{\left(p, p\right)\text{-}var}L_{\infty}}^{\left(1\right)}} \left\|\xi\right\|_{2}+\left\|f\left(0, 0\right)\right\|_{\mathbb{L}^{2}}+\left\|H, H^{\prime}\right\|_{\mathbf{D}_{X}^{\left(p, p\right)\text{-}var}L_{2}}^{\left(1\right)}.
	\end{equation*}
	Then we have 
	\begin{align*}
		\left\|Y, GY+H, Z\right\|^{\left(K\right)} &\leq \left\|Y-\Phi^{\bar{Y}}, GY+H-\Phi^{\bar{Y}^{\prime}}, Z-\Phi^{\bar{Z}}\right\|^{\left(K\right)}+\left\|\Phi^{\bar{Y}}, \Phi^{\bar{Y}^{\prime}}, \Phi^{\bar{Z}}\right\|^{\left(K\right)}\\
		&\leq \frac{1}{2}\left\|Y, GY+H, Z\right\|^{\left(K\right)}+\left\|\Phi^{\bar{Y}}, \Phi^{\bar{Y}^{\prime}}, \Phi^{\bar{Z}}\right\|^{\left(K\right)},
	\end{align*}
	and thus 
	\begin{align}
		\left\|Y, GY+H, Z\right\|^{\left(1\right)} &\leq \left\|Y, GY+H, Z\right\|^{\left(K\right)} \lesssim \left\|\Phi^{\bar{Y}}, \Phi^{\bar{Y}^{\prime}}, \Phi^{\bar{Z}}\right\|^{\left(K\right)} \notag\\
		&\lesssim_{\left\|G, G^{\prime}\right\|_{\mathbf{D}_{X}^{\left(p, p\right)\text{-}var}L_{\infty}}^{\left(1\right)}} \left\|\xi\right\|_{2}+\left\|f\left(0, 0\right)\right\|_{\mathbb{L}^{2}}+\left\|H, H^{\prime}\right\|_{\mathbf{D}_{X}^{\left(p, p\right)\text{-}var}L_{2}}^{\left(1\right)}. \label{solue1}
	\end{align}
	\indent
	At last, we show the existence and uniqueness and estimate \eqref{solue} for arbitrary $T$ and $\mathbf{X}$. 
	Define $w: \Delta \rightarrow \left[0, \infty\right)$ by 
	\begin{equation*}
		w\left(s, t\right):= 2^{p-1}\left(\left|\delta X\right|_{p\text{-}var; \left[s, t\right]}^{p}+\left|\mathbb{X}\right|_{\frac{p}{2}\text{-}var; \left[s, t\right]}^{p}\right), \quad \forall \left(s, t\right) \in \Delta.
	\end{equation*}
	By \cite{RPB2}*{Exercise 1.9 and Proposition 5.8}, $w$ is a control and we have 
	\begin{equation*}
		\left|\mathbf{X}\right|_{p\text{-}var; \left[s, t\right]}^{p} \leq w\left(s, t\right) \leq 2^{p-1}\left|\mathbf{X}\right|_{p\text{-}var; \left[s, t\right]}^{p}, \quad \forall \left(s, t\right) \in \Delta.
	\end{equation*}
	Define $s_{0}:=0$ and 
	\begin{equation*}
		s_{i}:=\sup\left\{t \in \left[s_{i-1}, T\right]: w\left(s_{i-1}, t\right) \leq \varepsilon^{p}\right\}, \quad \forall i=1, 2, \cdots.
	\end{equation*}
	Then there exists a positive integer $N$ such that 
	$\left|\mathbf{X}\right|_{p\text{-}var;\left[s_{i}, s_{i+1}\right]} \leq w\left(s_{i}, s_{i+1}\right)^{\frac{1}{p}}=\varepsilon$, 
	for $i=0, \cdots, N-2$, and 
	$\left|\mathbf{X}\right|_{p\text{-}var;\left[s_{N-1}, T\right]} \leq w\left(s_{N-1}, T\right)^{\frac{1}{p}} \leq \varepsilon$. 
	By the superadditivity of $w$, we have 
	\begin{equation*}
		\left(N-1\right)\varepsilon^{p} \leq \sum_{i=0}^{N-2}w\left(s_{i}, s_{i+1}\right) \leq w\left(0, T\right) \leq 2^{p-1}\left|\mathbf{X}\right|_{p\text{-}var}^{p}, 
	\end{equation*}
	and thus $N \leq C_{\left|\mathbf{X}\right|_{p\text{-}var}, \varepsilon}$. 
	Take $\bar{N}:=\lfloor\frac{T}{\varepsilon}\rfloor+1$. 
	Define partitions $\pi:=\left\{0<s_{1}<\cdots<s_{N-1}<T\right\}$, 
	$\bar{\pi}:=\left\{0<\varepsilon<\cdots<\left(\bar{N}-1\right)\varepsilon<T\right\}$ 
	and $\tilde{\pi}=\left\{0=t_{0}<t_{1}<\cdots<t_{\tilde{N}}=T\right\}:=\pi \cup \bar{\pi}$. 
	Then we have $\left|t_{i+1}-t_{i}\right| \vee \left|\mathbf{X}\right|_{p\text{-}var; \left[t_{i}, t_{i+1}\right]} \leq \varepsilon$ 
	for $i=0, \cdots, \tilde{N}-1$, and 
	\begin{equation*}
		\tilde{N} \leq N+\bar{N} \leq C_{T, \left|\mathbf{X}\right|_{p\text{-}var}, \varepsilon} \leq C_{T, \left|\mathbf{X}\right|_{p\text{-}var}, \left\|G, G^{\prime}\right\|_{\mathbf{D}_{X}^{\left(p, p\right)\text{-}var}L_{\infty}}^{\left(1\right)}}.
	\end{equation*}
	Define $Y_{t_{\tilde{N}}}^{\tilde{N}}:=\xi$, and define $\left(Y^{i}, Z^{i}\right)$
	recursively on $\left[t_{i},t_{i+1}\right]$ for $i=0, \cdots, \tilde{N}-1$ 
	by the solution to rough BSDE 
	\begin{align*}
		Y_{t}^{i}&=Y_{t_{i+1}}^{i+1}+\int_{t}^{t_{i+1}} f_{r}\left(Y_{r}^{i}, Z_{r}^{i}\right) d r+\int_{t}^{t_{i+1}} \left(G_{r}Y_{r}^{i}+H_{r}\right) d \mathbf{X}_{r}\\
		&\quad-\int_{t}^{t_{i+1}} Z_{r}^{i} d W_{r}, \quad t \in \left[t_{i}, t_{i+1}\right].
	\end{align*}
	Then define $Y_{t}:=Y_{t}^{\tilde{N}-1}$, $Z_{t}:=Z_{t}^{\tilde{N}-1}$ 
	for every $t \in \left[t_{\tilde{N}-1},t_{\tilde{N}}\right]$, 
	and define $Y_{t}:=Y_{t}^{i}$, $Z_{t}:=Z_{t}^{i}$ 
	for every $t \in \left[t_{i},t_{i+1}\right)$ and $i=0, \cdots, \tilde{N}-2$. 
	By Proposition \ref{ape}, 
	$\left(Y, Z\right)$ is the unique solution to \eqref{rbsde}. 
	Analogous to \eqref{solue1}, we have 
	\begin{align*}
		&\left\|Y^{i}, GY^{i}+H, Z^{i}\right\|^{\left(1\right)}\\
		&\quad \lesssim_{\left\|G, G^{\prime}\right\|_{\mathbf{D}_{X}^{\left(p, p\right)\text{-}var}L_{\infty}}^{\left(1\right)}} \left\|Y_{t_{i+1}}^{i+1}\right\|_{2}+\left\|f\left(0, 0\right)\right\|_{\mathbb{L}^{2}}+\left\|H, H^{\prime}\right\|_{\mathbf{D}_{X}^{\left(p, p\right)\text{-}var}L_{2}}^{\left(1\right)}. 
	\end{align*}
	By induction, we get the estimate \eqref{solue}. 
\end{proof}
\begin{remark}
	Theorem \ref{solueu} is new even in the one-dimensional case, 
	for our terminal value is not necessarily essentially bounded as in Diehl and Friz \cite{RBSDE}. 
	Moreover, our rough integrator $\mathbf{X}$ can be non-geometric. 
\end{remark}
\begin{remark}
	The composition of a stochastic controlled rough path $\left(Y, Y^{\prime}\right) \in \mathbf{D}_{X}^{\left(p, p\right)\text{-}var}L_{m}$ 
	with a nonlinear (even deterministic and time-invariant) vector field $g$, as a stochastic controlled rough path,
	might not stay like the linear rough drift within the same space $\mathbf{D}_{X}^{\left(p, p\right)\text{-}var}L_{m}$. 
	In fact,  both $\left(Y, Y^{\prime}\right) \in \mathbf{D}_{X}^{\left(p, p\right)\text{-}var}L_{m}$ 
	and $g \in Lip^{3}\left(\mathbb{R}^{k}, \mathbb{R}^{k \times e}\right)$ only ensure 
	$\left(g\left(Y\right), Dg\left(Y\right)Y^{\prime}\right) \in \mathbf{D}_{X}^{\left(p, p\right)\text{-}var}L_{\frac{m}{2}}$. As a consequence, our stochastic controlled rough path method fails to apply in a straightforward way 
	to BSDEs with a nonlinear rough drift 
	\begin{equation*}
		Y_{t}=\xi+\int_{t}^{T} f_{r}\left(Y_{r}, Z_{r}\right) d r+\int_{t}^{T} g_{r}\left(Y_{r}\right) d \mathbf{X}_{r}-\int_{t}^{T} Z_{r} d W_{r}, \quad t \in \left[0, T\right].
	\end{equation*}
	\indent
	In contrast, rough SDEs in Friz et al. \cite{RSDE3} 
	\begin{equation*} 
		S_{t}=S_{0}+\int_{0}^{t} b_{r}\left(S_{r}\right) d r+\int_{0}^{t} F_{r}\left(S_{r}\right) d \mathbf{X}_{r}+\int_{0}^{t} \sigma_{r}\left(S_{r}\right) d W_{r}, \quad t \in \left[0, T\right]
	\end{equation*}
	are discussed with the space of $\left(m, \infty\right)$-integrable stochastic controlled rough paths 
	$\mathbf{D}_{X}^{\alpha, \alpha}L_{m, \infty}$ (see \cite{RSDE3}*{Definition 3.2}).  
	Both $\left(Y, Y^{\prime}\right) \in \mathbf{D}_{X}^{\alpha, \alpha}L_{m, \infty}$ 
	and $F \in Lip^{3}\left(\mathbb{R}^{k}, \mathbb{R}^{k \times e}\right)$ 
	ensure $\left(F\left(Y\right), DF\left(Y\right)Y^{\prime}\right) \in \mathbf{D}_{X}^{\alpha, \alpha}L_{m, \infty}$ 
	(see \cite{RSDE3}*{Lemma 3.13}). 
	However, their method still fails to apply in a straightforward way to our rough BSDEs. 
	Indeed, $\left(Y, Y^{\prime}\right) \in \mathbf{D}_{X}^{\alpha, \alpha}L_{m, \infty}$ requires  
	\begin{equation} \label{mii}
		\left\|\mathbb{E}_{s}\left|\delta Y_{s, t}\right|^{m}\right\|_{\infty}^{\frac{1}{m}} \lesssim \left|t-s\right|^{\alpha}, \quad \forall \left(s, t\right) \in \Delta,
	\end{equation}
	which, in our $p$-variation scale, corresponds to the following condition 
	\begin{equation*}
		\sup_{\pi \in \mathcal{P}\left[0, T\right]}\left(\sum_{\left[u, v\right] \in \pi}\left\|\mathbb{E}_{u}\left|\delta Y_{u, v}\right|^{m}\right\|_{\infty}^{\frac{p}{m}}\right)^{\frac{1}{p}} < \infty,
	\end{equation*}
	and fails to be true  for an $L^{m}$-integrable (or even BMO) martingale $Y$. 
	Hence, it seems hopeless that $\left(M, 0\right)$ defined by \eqref{Mdenf} is 
	a stochastic controlled rough path of $\left(m, \infty\right)$-integrability. 
	Friz et al. \cite{RSDE3} assume that $\sigma$ is essentially bounded so that 
	$Y:=\int_{0}^{\cdot} \sigma_{r}\left(S_{r}\right) d W_{r}$ satisfies the condition \eqref{mii}. 
\end{remark}
\begin{remark}
	When $\left(G, G^{\prime}\right) \in \mathscr{D}_{X}^{\left(p, p\right)\text{-}var}L_{\infty}$, 
	we might expect to solve the rough BSDE \eqref{rbsde} in spirit of duality, since the rough drift is linear. 
	More precisely, consider the following RDE with values in 
	$L^{\infty}\left(\Omega; \mathbb{R}^{k \times k}\right)$ 
	\begin{equation} \label{rde}
		S_{t}=I_{k}+\int_{0}^{t} S_{r}G_{r} d \mathbf{X}_{r} +\int_{0}^{t} S_{r}G_{r}^{2} d \left[\mathbf{X}\right]_{r}, \quad t \in \left[0, T\right].
	\end{equation}
	Here, the unknown process $S: \Omega \times \left[0, T\right] \rightarrow \mathbb{R}^{k \times k}$ 
	is viewed as a path from $\left[0, T\right]$ to 
	$L^{\infty}\left(\Omega; \mathbb{R}^{k \times k}\right)$, 
	$I_{k}$ is the identity matrix in $\mathbb{R}^{k \times k}$, 
	$\left[\mathbf{X}\right]$ is the rough bracket of $\mathbf{X}$ defined by 
	\begin{equation*}
		\left[\mathbf{X}\right]: \left[0, T\right] \rightarrow \mathbb{R}^{e} \otimes \mathbb{R}^{e}, \quad \left[\mathbf{X}\right]_{t}:=\delta X_{0, t} \otimes \delta X_{0, t}-2\operatorname{Sym}\left(\mathbb{X}_{0, t}\right), \quad \forall t \in \left[0, T\right],
	\end{equation*}
	and $\int_{0}^{\cdot} SG^{2} d \left[\mathbf{X}\right]$ is a Young integral. 
	Analogous to the proof of \cite{RPJ}*{Theorems 3.2 and 3.8}, 
	we can show that the RDE \eqref{rde} has a unique solution $S$ and 
	$\left(S, SG\right) \in \mathscr{D}_{X}^{\left(p, p\right)\text{-}var}L_{\infty}$.\\
	\indent
	Heuristically, if $\left(Y, Z\right)$ is the solution to \eqref{rbsde}, 
	from the rough Itô formula \cite{RSDE3}*{Theorem 4.13}, we might have 
	\begin{align*}
		S_{t}Y_{t}&=S_{T}\xi+\int_{t}^{T} S_{r}f_{r}\left(Y_{r}, Z_{r}\right) d r+\int_{t}^{T} S_{r}H_{r} d \mathbf{X}_{r}\\
		&\quad+\int_{t}^{T} S_{r}G_{r}H_{r} d \left[\mathbf{X}\right]_{r}-\int_{t}^{T} S_{r}Z_{r} d W_{r}, \quad t \in \left[0, T\right]. 
	\end{align*}
	Then $\left(\tilde{Y}, \tilde{Z}\right) \in C^{p\text{-}var}L_{2} \times \mathbb{L}^{2}$ 
	defined by 
	\begin{equation*}
		\tilde{Y}_{t}:=S_{t}Y_{t}-\int_{t}^{T} SH d \mathbf{X}-\int_{t}^{T} SGH d \left[\mathbf{X}\right], \quad \tilde{Z}_{t}:=S_{t}Z_{t}, \quad \forall t \in \left[0, T\right]
	\end{equation*}
	might solve the BSDE 
	\begin{equation} \label{bsde}
		\tilde{Y}_{t}=S_{T}\xi+\int_{t}^{T} \tilde{f}_{r}\left(\tilde{Y}_{r}, \tilde{Z}_{r}\right) d r-\int_{t}^{T} \tilde{Z}_{r} d W_{r}, \quad t \in \left[0, T\right],
	\end{equation}
	where 
	\begin{equation*}
		\tilde{f}_{t}\left(\tilde{y}, \tilde{z}\right):=S_{t}f_{t}\left(S_{t}^{-1}\left(\tilde{y}+\int_{t}^{T} SH d \mathbf{X}+\int_{t}^{T} SGH d \left[\mathbf{X}\right]\right), S_{t}^{-1}\tilde{z}\right).
	\end{equation*}
	Conversely, if the BSDE \eqref{bsde} has a solution 
	$\left(\tilde{Y}, \tilde{Z}\right)$, then $\left(Y, Z\right)$ defined by 
	\begin{equation*}
		Y_{t}:=S_{t}^{-1}\left(\tilde{Y}_{t}+\int_{t}^{T} SH d \mathbf{X}+\int_{t}^{T} SGH d \left[\mathbf{X}\right]\right), \quad Z_{t}:=S_{t}^{-1}\tilde{Z}_{t}, \quad \forall t \in \left[0, T\right].
	\end{equation*}
	might solve the rough BSDE \eqref{rbsde}. 
	Since the conditions of the rough Itô formula \cite{RSDE3}*{Theorem 4.13} are not ensured within our context, 
	the above derivation stays only formal and is still far away from a rigorous proof. 
\end{remark}
\subsection{Continuity of the solution map} \label{S5.2}
We finish this section with continuity of the solution map. 
\begin{theorem} \label{solumc}
	Let $\left(\xi, f, G, G^{\prime}, H, H^{\prime}, \mathbf{X}\right)$ and $\left(\xi^{n}, f^{n}, G^{n}, G^{n, \prime}, H^{n}, H^{n, \prime}, \mathbf{X}^{n}\right), n=1, 2, \cdots$ satisfy Assumption \ref{assu}, 
	$\left(Y, Z\right)$ and $\left(Y^{n}, Z^{n}\right)$ be the solution to 
	the corresponding rough BSDE \eqref{rbsde}. 
	Assume 
	\begin{align}
		&\left\|\xi^{n}-\xi\right\|_{2}+\left\|f^{n}\left(0, 0\right)-f\left(0, 0\right)\right\|_{\mathbb{L}^{2}}+\left\|G^{n}, G^{n, \prime}; G, G^{\prime}\right\|_{\mathbf{D}_{X^{n}, X}^{\left(p, p\right)\text{-}var}L_{\infty}}^{\left(1\right)} \notag\\
		&\quad+\left\|H^{n}, H^{n, \prime}; H, H^{\prime}\right\|_{\mathbf{D}_{X^{n}, X}^{\left(p, p\right)\text{-}var}L_{2}}^{\left(1\right)}+\rho_{p\text{-}var}\left(\mathbf{X}^{n}, \mathbf{X}\right) \rightarrow 0, \quad as\ n \rightarrow \infty \label{convcond1}
	\end{align}
	and 
	\begin{equation} \label{convcond2}
		f^{n}\left(y, z\right) \rightarrow f\left(y, z\right)\ in\ measure\ d\mathbb{P} \times dt, \quad as\ n \rightarrow \infty, \quad \forall \left(y, z\right) \in \mathbb{R}^{k} \times \mathbb{R}^{k \times d}. 
	\end{equation}
	Then we have 
	\begin{equation} \label{conv}
		\lim_{n \rightarrow \infty}\left(\left\|Y^{n}, G^{n}Y^{n}+H^{n}; Y, GY+H\right\|_{\mathbf{D}_{X^{n}, X}^{\left(p, p\right)\text{-}var}L_{2}}^{\left(1\right)}+\left\|Z^{n}-Z\right\|_{\mathbb{L}^{2}}\right)=0. 
	\end{equation}
\end{theorem}
\begin{proof}
	Write $\xi^{0}:=\xi$, $f^{0}:=f$, $\Delta_{n}Y:=Y^{n}-Y$, and similarly for other functions. 
	Let $K \geq 1$ and $\varepsilon \in \left(0, \frac{1}{K}\right]$ be constants 
	which are waiting to be determined later. Define 
	\begin{align*}
		c_{n}&:=\left\|\Delta_{n}\xi\right\|_{2}+\left\|\Delta_{n}f\left(0, 0\right)\right\|_{\mathbb{L}^{2}}+\left\|G^{n}, G^{n, \prime}; G, G^{\prime}\right\|_{\mathbf{D}_{X^{n}, X}^{\left(p, p\right)\text{-}var}L_{\infty}}^{\left(K\right)}\\
		&\quad+\left\|H^{n}, H^{n, \prime}; H, H^{\prime}\right\|_{\mathbf{D}_{X^{n}, X}^{\left(p, p\right)\text{-}var}L_{2}}^{\left(K\right)}+\rho_{p\text{-}var}\left(\mathbf{X}^{n}, \mathbf{X}\right).
	\end{align*}
	In view of \eqref{convcond1}, there exists a positive constant $M$ such that 
	\begin{equation*}
		\left\|\xi^{n}\right\|_{2} \vee \left\|f^{n}\left(0, 0\right)\right\|_{\mathbb{L}^{2}} \vee \left\|G^{n}, G^{n, \prime}\right\|_{\mathbf{D}_{X^{n}}^{\left(p, p\right)\text{-}var}L_{\infty}}^{\left(1\right)} \vee \left\|H^{n}, H^{n, \prime}\right\|_{\mathbf{D}_{X^{n}}^{\left(p, p\right)\text{-}var}L_{2}}^{\left(1\right)} \vee \left|\mathbf{X}^{n}\right|_{p\text{-}var} \leq M,
	\end{equation*}
	for $n=0, 1, \cdots$. Then there exists a partition $\pi=\left\{0=t_{0}<t_{1}<\cdots<t_{N}=T\right\}$ such that 
	$\left|t_{i+1}-t_{i}\right| \vee \left|\mathbf{X}^{n}\right|_{p\text{-}var; \left[t_{i}, t_{i+1}\right]} \leq \varepsilon$, 
	for $i=0, \cdots, N-1$ and $n=0, 1, \cdots$. Hence, we can assume without loss of generality that 
	$T \vee \sup_{n}\left|\mathbf{X}^{n}\right|_{p\text{-}var} \leq \varepsilon$. 
	By Proposition \ref{solueu} and its proof, we have 
	\begin{equation*}
		\left\|Y^{n}, G^{n}Y^{n}+H^{n}\right\|_{\mathbf{D}_{X^{n}}^{\left(p, p\right)\text{-}var}L_{2}}^{\left(1\right)}+\left\|Z^{n}\right\|_{\mathbb{L}^{2}} \leq C_{M}, \quad  \forall n=0, 1, \cdots.
	\end{equation*}
	For every fixed $n$, since 
	\begin{align*}
		&\Delta_{n}\left[f\left(Y, Z\right)\right]\\
		&\quad=\left[f^{n}\left(Y^{n}, Z^{n}\right)-f^{n}\left(Y, Z\right)\right]+\left[f^{n}\left(Y, Z\right)-f^{n}\left(0, 0\right)-f\left(Y, Z\right)+f\left(0, 0\right)\right]\\
		&\quad \quad+\left[f^{n}\left(0, 0\right)-f\left(0, 0\right)\right]\\
		&\quad=\left[f^{n}\left(Y^{n}, Z^{n}\right)-f^{n}\left(Y, Z\right)\right]+\left[\left(\Delta_{n}f\right)\left(Y, Z\right)-\Delta_{n}f\left(0, 0\right)\right]+\Delta_{n}f\left(0, 0\right),
	\end{align*}
	we have 
	\begin{align}
		&\left\|\Delta_{n}\left[f\left(Y, Z\right)\right]\right\|_{\mathbb{L}^{2}} \notag\\
		&\leq \left\|f^{n}\left(Y^{n}, Z^{n}\right)-f^{n}\left(Y, Z\right)\right\|_{\mathbb{L}^{2}}+\left\|\left(\Delta_{n}f\right)\left(Y, Z\right)-\Delta_{n}f\left(0, 0\right)\right\|_{\mathbb{L}^{2}}+\left\|\Delta_{n}f\left(0, 0\right)\right\|_{\mathbb{L}^{2}} \label{contie1}\\
		&\lesssim \sup_{r \in \left[0, T\right]}\left\|\Delta_{n}Y_{r}\right\|_{2}+\left\|\Delta_{n}Z\right\|_{\mathbb{L}^{2}}+\left\|\left(\Delta_{n}f\right)\left(Y, Z\right)-\Delta_{n}f\left(0, 0\right)\right\|_{\mathbb{L}^{2}}+c_{n}. \notag
	\end{align}
	Then by Lemma \ref{L2scrp}, we have 
	\begin{align}
		&\left\|\int_{0}^{\cdot} f_{r}^{n}\left(Y_{r}^{n}, Z_{r}^{n}\right) d r, 0; \int_{0}^{\cdot} f_{r}\left(Y_{r}, Z_{r}\right) d r, 0\right\|_{\mathbf{D}_{X^{n}, X}^{\left(p, p\right)\text{-}var}L_{2}}^{\left(K\right)} \notag\\
		&\quad \lesssim K\varepsilon^{\frac{1}{2}}\left(\sup_{r \in \left[0, T\right]}\left\|\Delta_{n} Y_{r}\right\|_{2}+\left\|\Delta_{n} Z\right\|_{\mathbb{L}^{2}}\right) \label{contie2}\\
		&\quad \quad+C_{K, \varepsilon}\left(\left\|\left(\Delta_{n}f\right)\left(Y, Z\right)-\Delta_{n}f\left(0, 0\right)\right\|_{\mathbb{L}^{2}}+c_{n}\right). \notag
	\end{align}
	By Proposition \ref{lcstab}, we have 
	\begin{align*}
		&\left\|G^{n}Y^{n}, G^{n}\left(G^{n}Y^{n}+H^{n}\right)+G^{n, \prime}Y^{n}; GY, G\left(GY+H\right)+G^{\prime}Y\right\|_{\mathbf{D}_{X^{n}, X}^{\left(p, p\right)\text{-}var}L_{2}}^{\left(K\right)}\\
		&\quad \lesssim_{M} \left\|Y^{n}, G^{n}Y^{n}+H^{n};Y, GY+H\right\|_{\mathbf{D}_{X^{n}, X}^{\left(p, p\right)\text{-}var}L_{2}}^{\left(K\right)}+c_{n}. 
	\end{align*}
	Then by Theorem \ref{rsistab}, we have 
	\begin{align}
		&\left\|\int_{0}^{\cdot} G^{n}Y^{n}+H^{n} d \mathbf{X}^{n}, Y^{n}; \int_{0}^{\cdot} GY+H d \mathbf{X}, Y\right\|_{\mathbf{D}_{X^{n}, X}^{\left(p, p\right)\text{-}var}L_{2}}^{\left(K\right)} \notag\\
		&\quad \lesssim_{M} \left(\frac{1}{K}+K\varepsilon\right)\left\|G^{n}Y^{n}+H^{n}, G^{n}\left(G^{n}Y^{n}+H^{n}\right)+G^{n, \prime}Y^{n}+H^{n, \prime};\right. \notag\\ 
		&\quad \quad \left.GY+H, G\left(GY+H\right)+G^{\prime}Y+H^{\prime}\right\|_{\mathbf{D}_{X^{n}, X}^{\left(p, p\right)\text{-}var}L_{2}}^{\left(K\right)}+C_{K, \varepsilon}c_{n} \label{contie3}\\
		&\quad \lesssim_{M} \left(\frac{1}{K}+K\varepsilon\right)\left\|Y^{n}, G^{n}Y^{n}+H^{n};Y, GY+H\right\|_{\mathbf{D}_{X^{n}, X}^{\left(p, p\right)\text{-}var}L_{2}}^{\left(K\right)}+C_{K, \varepsilon}c_{n}, \notag
	\end{align}
	and thus 
	\begin{align}
		&\left\|\Delta_{n}\int_{0}^{T} GY+H d \mathbf{X}\right\|_{2} \notag\\
		&\quad \leq \frac{1}{K}\left\|\int_{0}^{\cdot} G^{n}Y^{n}+H^{n} d \mathbf{X}^{n}, Y^{n}; \int_{0}^{\cdot} GY+H d \mathbf{X}, Y\right\|_{\mathbf{D}_{X^{n}, X}^{\left(p, p\right)\text{-}var}L_{2}}^{\left(K\right)} \label{contie4}\\
		&\quad \lesssim_{M} \frac{1}{K}\left(\frac{1}{K}+K\varepsilon\right)\left\|Y^{n}, G^{n}Y^{n}+H^{n};Y, GY+H\right\|_{\mathbf{D}_{X^{n}, X}^{\left(p, p\right)\text{-}var}L_{2}}^{\left(K\right)}+C_{K, \varepsilon}c_{n}, \notag
	\end{align}
	Define $M^{n}: \Omega \times \left[0, T\right]$ for $n=0, 1, \cdots$ by 
	\begin{equation*} 
		M_{t}^{n}:=\mathbb{E}_{t}\left[\xi^{n}+\int_{0}^{T} f_{r}^{n}\left(Y_{r}^{n}, Z_{r}^{n}\right) d r+\int_{0}^{T} \left(G^{n}Y^{n}+H^{n}\right) d \mathbf{X}^{n}\right], \quad \forall t \in \left[0, T\right].
	\end{equation*}
	Then 
	\begin{equation*}
		M_{t}^{n}=M_{0}^{n}+\int_{0}^{t} Z_{r}^{n} d W_{r}, \quad \forall t \in \left[0, T\right], \quad \forall n=0, 1, \cdots,
	\end{equation*}
	and thus 
	\begin{align*}
		\left\|\Delta_{n}Z\right\|_{\mathbb{L}^{2}} \leq 2\left\|\Delta_{n} M_{T}\right\|_{2} \lesssim \varepsilon^{\frac{1}{2}}\left\|\Delta_{n}\left[f(Y, Z)\right]\right\|_{\mathbb{L}^{2}}+\left\|\Delta_{n}\int_{0}^{T} GY+H d \mathbf{X}\right\|_{2}.
	\end{align*}
	Combining \eqref{contie1}, \eqref{contie4} and the last inequality, we have 
	\begin{align}
		&\left\|\Delta_{n}Z\right\|_{\mathbb{L}^{2}} \notag\\
		&\quad \lesssim_{M} \frac{1}{K}\left(\frac{1}{K}+K\varepsilon\right)\left(\left\|Y^{n}, G^{n}Y^{n}+H^{n}; Y, GY+H\right\|_{\mathbf{D}_{X^{n}, X}^{\left(p, p\right)\text{-}var}L_{2}}^{\left(K\right)}+K\left\|\Delta_{n}Z\right\|_{\mathbb{L}^{2}}\right) \label{contie5}\\
		&\quad \quad+C_{K, \varepsilon}\left(\left\|\left(\Delta_{n}f\right)\left(Y, Z\right)-\Delta_{n}f\left(0, 0\right)\right\|_{\mathbb{L}^{2}}+c_{n}\right). \notag
	\end{align}
	By Lemma \ref{mtg}, we have 
	\begin{align}
		&\left\|\int_{0}^{\cdot} Z_{r}^{n} d W_{r},0; \int_{0}^{\cdot} Z_{r} d W_{r}, 0\right\|_{\mathbf{D}_{X^{n}, X}^{\left(p, p\right)\text{-}var}L_{2}}^{\left(K\right)} \notag\\
		&\quad \lesssim_{M} \left(\frac{1}{K}+K\varepsilon\right)\left(\left\|Y^{n}, G^{n}Y^{n}+H^{n}; Y, GY+H\right\|_{\mathbf{D}_{X^{n}, X}^{\left(p, p\right)\text{-}var}L_{2}}^{\left(K\right)}+K\left\|\Delta_{n}Z\right\|_{\mathbb{L}^{2}}\right) \label{contie6}\\
		&\quad \quad+C_{K, \varepsilon}\left(\left\|\left(\Delta_{n}f\right)\left(Y, Z\right)-\Delta_{n}f\left(0, 0\right)\right\|_{\mathbb{L}^{2}}+c_{n}\right). \notag
	\end{align}
	Note that 
	\begin{equation*}
		\Delta_{n}Y=\int_{\cdot}^{T} \Delta_{n}\left[f_{r}\left(Y_{r}, Z_{r}\right)\right] d r+\Delta_{n}\int_{\cdot}^{T} G\Delta Y d \mathbf{X}-\int_{\cdot}^{T} \Delta_{n} Z_{r} d W_{r}, \quad \forall n=0, 1, \cdots.
	\end{equation*}
	Combining \eqref{contie2}, \eqref{contie3} and \eqref{contie6}, we have 
	\begin{align*}
		&\left\|Y^{n}, G^{n}Y^{n}+H^{n}; Y, GY+H\right\|_{\mathbf{D}_{X^{n}, X}^{\left(p, p\right)\text{-}var}L_{2}}^{\left(K\right)}\\ 
		&\quad \lesssim_{M} \left(\frac{1}{K}+K\varepsilon\right)\left(\left\|Y^{n}, G^{n}Y^{n}+H^{n}; Y, GY+H\right\|_{\mathbf{D}_{X^{n}, X}^{\left(p, p\right)\text{-}var}L_{2}}^{\left(K\right)}+K\left\|\Delta_{n}Z\right\|_{\mathbb{L}^{2}}\right)\\
		&\quad \quad+C_{K, \varepsilon}\left(\left\|\left(\Delta_{n}f\right)\left(Y, Z\right)-\Delta_{n}f\left(0, 0\right)\right\|_{\mathbb{L}^{2}}+c_{n}\right).
	\end{align*}
	Then 
	\begin{align*}
		&\left\|Y^{n}, G^{n}Y^{n}+H^{n}; Y, GY+H\right\|_{\mathbf{D}_{X^{n}, X}^{\left(p, p\right)\text{-}var}L_{2}}^{\left(K\right)}+K\left\|\Delta_{n}Z\right\|_{\mathbb{L}^{2}}\\ 
		&\quad \leq C_{M}\left(\frac{1}{K}+K\varepsilon\right)\left(\left\|Y^{n}, G^{n}Y^{n}+H^{n}; Y, GY+H\right\|_{\mathbf{D}_{X^{n}, X}^{\left(p, p\right)\text{-}var}L_{2}}^{\left(K\right)}+K\left\|\Delta_{n}Z\right\|_{\mathbb{L}^{2}}\right)\\
		&\quad \quad+C_{M, K, \varepsilon}\left(\left\|\left(\Delta_{n}f\right)\left(Y, Z\right)-\Delta_{n}f\left(0, 0\right)\right\|_{\mathbb{L}^{2}}+c_{n}\right).
	\end{align*}
	Taking $K:=4\left(1+C_{M}\right)$ and $\varepsilon:=\frac{1}{K^{2}}$ for the above $C_{M}$, 
	we have $C_{M}\left(\frac{1}{K}+K\varepsilon^{\frac{1}{2}}\right) \leq \frac{1}{2}$ and thus 
	\begin{align*}
		&\left\|Y^{n}, G^{n}Y^{n}+H^{n}; Y, GY+H\right\|_{\mathbf{D}_{X^{n}, X}^{\left(p, p\right)\text{-}var}L_{2}}^{\left(K\right)}+K\left\|\Delta_{n}Z\right\|_{\mathbb{L}^{2}}\\
		&\quad \lesssim_{M} \left\|\left(\Delta_{n}f\right)\left(Y, Z\right)-\Delta_{n}f\left(0, 0\right)\right\|_{\mathbb{L}^{2}}+c_{n}. 
	\end{align*}
	In view of \eqref{convcond2}, we have 
	\begin{equation*}
		\left(\Delta_{n}f\right)\left(Y, Z\right)-\Delta_{n}f\left(0, 0\right) \rightarrow 0\ in\ measure\ d\mathbb{P} \times dt, \quad as\ n \rightarrow \infty.
	\end{equation*}
	Since 
	\begin{equation*}
		\left\|\left(\Delta_{n}f\right)\left(Y, Z\right)-\Delta_{n}f\left(0, 0\right)\right\|_{\mathbb{L}^{2}} \lesssim \sup_{r \in \left[0, T\right]}\left\|Y_{r}\right\|_{2}+\left\|Z\right\|_{\mathbb{L}^{2}} \lesssim C_{M}, \quad n=0, 1, \cdots, 
	\end{equation*}
	the dominated convergence theorem gives 
	\begin{equation*}
		\lim_{n \rightarrow \infty}\left\|\left(\Delta_{n}f\right)\left(Y, Z\right)-\Delta_{n}f\left(0, 0\right)\right\|_{\mathbb{L}^{2}}=0. 
	\end{equation*}
	Therefore, we get \eqref{conv}.
\end{proof}

\section{Application to systems of rough PDEs} \label{S6}
We now specialize to a Markovian model. 
For any $\left(s, x\right) \in \left[0, T\right] \times \mathbb{R}^{d}$ 
and $\mathbf{X}=\left(X, \mathbb{X}\right) \in \mathscr{C}_{g}^{0, p\text{-}var}\left([0, T], \mathbb{R}^{e}\right)$, 
consider the following SDE and rough BSDE 
\begin{align}
	S_{t}^{s, x, \mathbf{X}}&=x+\int_{s}^{t} b_{r}\left(S_{r}^{s, x, \mathbf{X}}, \mathbf{X}\right) d r+\int_{s}^{t} \sigma_{r}\left(S_{r}^{s, x, \mathbf{X}}, \mathbf{X}\right) d W_{r}, \quad t \in \left[s, T\right]; \label{fbsde1}\\
	Y_{t}^{s, x, \mathbf{X}}&=l\left(S_{T}^{s, x, \mathbf{X}}, \mathbf{X}\right)+\int_{t}^{T} f_{r}\left(S_{r}^{s, x, \mathbf{X}}, Y_{r}^{s, x, \mathbf{X}}, Z_{r}^{s, x, \mathbf{X}}, \mathbf{X}\right) d r \notag\\
	&\quad+\int_{t}^{T} \left(G_{r}\left(\mathbf{X}\right)Y_{r}^{s, x, \mathbf{X}}+h_{r}\left(S_{r}^{s, x, \mathbf{X}}, \mathbf{X}\right)+H_{r}\left(\mathbf{X}\right)\right) d \mathbf{X}_{r} \label{fbsde2}\\
	&\quad-\int_{t}^{T} Z_{r}^{s, x, \mathbf{X}} d W_{r}, \quad t \in \left[s, T\right]. \notag
\end{align}
Here, $S$ is $d$-dimensional and $Y$ is $k$-dimensional. 
Coefficients 
$b: \left[0, T\right] \times \mathbb{R}^{d} \times \mathscr{C}_{g}^{0, p\text{-}var} \rightarrow \mathbb{R}^{d}$, 
$\sigma: \left[0, T\right] \times \mathbb{R}^{d} \times \mathscr{C}_{g}^{0, p\text{-}var} \rightarrow \mathbb{R}^{d \times d}$, 
$l: \mathbb{R}^{d} \times \mathscr{C}_{g}^{0, p\text{-}var} \rightarrow \mathbb{R}^{k}$, 
$f: \left[0, T\right] \times \mathbb{R}^{d} \times \mathbb{R}^{k} \times \mathbb{R}^{{k} \times d} \times \mathscr{C}_{g}^{0, p\text{-}var} \rightarrow \mathbb{R}^{k}$, 
$G: \left[0, T\right] \times \mathscr{C}_{g}^{0, p\text{-}var} \rightarrow \mathcal{L}\left(\mathbb{R}^{k}, \mathbb{R}^{k \times e}\right)$, 
$G^{\prime}: \left[0, T\right] \times \mathscr{C}_{g}^{0, p\text{-}var} \rightarrow \mathcal{L}\left(\mathbb{R}^{k}, \mathcal{L}\left(\mathbb{R}^{e}, \mathbb{R}^{k \times e}\right)\right)$ 
(which is implied in \eqref{fbsde2}), 
$h: \left[0, T\right] \times \mathbb{R}^{d} \times \mathscr{C}_{g}^{0, p\text{-}var} \rightarrow \mathbb{R}^{k \times e}$, 
$H: \left[0, T\right] \times \mathscr{C}_{g}^{0, p\text{-}var} \rightarrow \mathbb{R}^{k \times e}$ 
and $H^{\prime}: \left[0, T\right] \times \mathscr{C}_{g}^{0, p\text{-}var} \rightarrow \mathcal{L}\left(\mathbb{R}^{e}, \mathbb{R}^{k \times e}\right)$ 
(which is also implied in \eqref{fbsde2}) are deterministic and depend on $\mathbf{X}$. 
Equations \eqref{fbsde1} and \eqref{fbsde2} are connected to the following system of rough PDEs 
\begin{equation} \label{rpde}
	\left\{
	\begin{aligned}
		&du\left(t, x, \mathbf{X}\right)+\frac{1}{2}D^{2}u\left(t, x, \mathbf{X}\right)\sigma_{t}^{2}\left(x, \mathbf{X}\right)dt+Du\left(t, x, \mathbf{X}\right)b_{t}\left(x, \mathbf{X}\right)dt\\
		&\quad+f_{t}\left(x, u\left(t, x, \mathbf{X}\right), Du\left(t, x, \mathbf{X}\right)\sigma_{t}\left(x, \mathbf{X}\right), \mathbf{X}\right)dt\\
		&\quad+\left[G_{t}\left(\mathbf{X}\right)u\left(t, x, \mathbf{X}\right)+h_{t}\left(x, \mathbf{X}\right)+H_{t}\left(\mathbf{X}\right)\right]d\mathbf{X}_{t}=0,\\
		&\quad \left(t, x, \mathbf{X}\right) \in \left[0, T\right) \times \mathbb{R}^{d} \times \mathscr{C}_{g}^{0, p\text{-}var};\\
		&u\left(T, x, \mathbf{X}\right)=l\left(x, \mathbf{X}\right), \quad \left(x, \mathbf{X}\right) \in \mathbb{R}^{d} \times \mathscr{C}_{g}^{0, p\text{-}var},
	\end{aligned}
	\right.
\end{equation}
where 
\begin{equation*}
	D^{2}u\left(t, x, \mathbf{X}\right)\sigma_{t}^{2}\left(x, \mathbf{X}\right):=\left(\begin{array}{c}
		\operatorname{Tr}\left[D^{2}u_{1}\left(t, x, \mathbf{X}\right)\left(\sigma\sigma^{\top}\right)_{t}\left(x, \mathbf{X}\right)\right] \\
		\vdots \\
		\operatorname{Tr}\left[D^{2}u_{k}\left(t, x, \mathbf{X}\right)\left(\sigma\sigma^{\top}\right)_{t}\left(x, \mathbf{X}\right)\right]
		\end{array}\right).
\end{equation*}
Denote by $BUC\left(V, \bar{V}\right)$ the space of bounded uniformly continuous 
functions from $V$ to $\bar{V}$ with the uniform topology. 
The space of (deterministic) controlled rough paths (controlled by $X$) of finite 
$\left(p, p\right)$-variation with values in $V$ is denoted by 
$\mathscr{D}_{X}^{\left(p, p\right)\text{-}var}\left(\left[0, T\right], V\right)$. 
It can be regarded as a subspace of 
$\mathscr{D}_{X}^{\left(p, p\right)\text{-}var}L_{\infty}\left(\left[0, T\right], \Omega; V\right)$. 
For $K \geq 1$, we define the norm $\left|\cdot\right|_{\mathscr{D}_{X}^{\left(p, p\right)\text{-}var}}^{\left(K\right)}$ 
and ``distance'' $\left|\cdot; \cdot\right|_{\mathscr{D}_{X, \bar{X}}^{\left(p, p\right)\text{-}var}}^{\left(K\right)}$ 
like defining $\left\|\cdot\right\|_{\mathbf{D}_{X}^{\left(p, p\right)\text{-}var}L_{\infty}}^{\left(K\right)}$ 
and $\left\|\cdot;\cdot\right\|_{\mathbf{D}_{X, \bar{X}}^{\left(p, p\right)\text{-}var}L_{\infty}}^{\left(K\right)}$ 
by \eqref{norm} and \eqref{dist}, but removing the dependency on $\omega$. 
Denote by $C_{b}^{1, 2}\left(\left[0, T\right] \times \mathbb{R}^{d}, V\right)$ 
the space of functions $F: \left[0, T\right] \times \mathbb{R}^{d} \rightarrow V$ 
which are once and twice continuously differentiable with respect to $t$ and $x$, 
respectively, such that 
\begin{equation*}
	\left|F\right|_{C_{b}^{1, 2}}:=\left|F\right|_{\infty}+\left|\partial_{t}F\right|_{\infty}+\left|DF\right|_{\infty}+\left|D^{2}F\right|_{\infty} < \infty.
\end{equation*}
\begin{definition} \label{rpdesolu}
	We call $u: \left[0, T\right] \times \mathbb{R}^{d} \times \mathscr{C}_{g}^{0, p\text{-}var} \rightarrow \mathbb{R}^{k}$ 
	a solution to \eqref{rpde} if the map 
	$\mathbf{X} \mapsto u\left(\cdot, \cdot, \mathbf{X}\right)$ 
	is continuous from $\mathscr{C}_{g}^{0, p\text{-}var}$ to 
	$BUC\left(\left[0, T\right] \times \mathbb{R}^{d}, \mathbb{R}^{k}\right)$, 
	and for any smooth path 
	$X \in C^{\infty}\left(\left[0, T\right], \mathbb{R}^{e}\right)$ 
	and its canonical rough lift $\mathbf{X}$, 
	$u\left(\cdot, \cdot, \mathbf{X}\right)$ is a viscosity solution to the following system of PDEs 
	\begin{equation} \label{pde}
		\left\{
		\begin{aligned}
			&\partial_{t}v\left(t, x\right)+\frac{1}{2}D^{2}v\left(t, x\right)\sigma_{t}^{2}\left(x, \mathbf{X}\right)+Dv\left(t, x\right)b_{t}\left(x, \mathbf{X}\right)\\
			&\quad+f_{t}\left(x, v\left(t, x\right), Dv\left(t, x\right)\sigma_{t}\left(x, \mathbf{X}\right), \mathbf{X}\right)\\
			&\quad+\left[G_{t}\left(\mathbf{X}\right)v\left(t, x\right)+h_{t}\left(x, \mathbf{X}\right)+H_{t}\left(\mathbf{X}\right)\right]\dot{X}_{t}=0, \quad \left(t, x\right) \in \left[0, T\right) \times \mathbb{R}^{d};\\
			&v\left(T, x\right)=l\left(x, \mathbf{X}\right), \quad x \in \mathbb{R}^{d}.
		\end{aligned}
		\right.
	\end{equation}
\end{definition}
We introduce the following assumption. 
\begin{assumption} \label{rpdeassu}
	For any $\mathbf{X} \in \mathscr{C}_{g}^{0, p\text{-}var}$, 
	\begin{enumerate}[(i)]
		\item $b_{\cdot}\left(\cdot, \mathbf{X}\right)$ and $\sigma_{\cdot}\left(\cdot, \mathbf{X}\right)$ 
		are bounded continuous and uniformly Lipschitz continuous in $x$;
		\item $l\left(\cdot, \mathbf{X}\right)$ is bounded and uniformly Lipschitz continuous;
		\item $f_{\cdot}\left(\cdot, \cdot, \cdot, \mathbf{X}\right)$ is bounded continuous 
		and uniformly Lipschitz continuous in $x, y$ and $z$;
		\item $h_{\cdot}\left(\cdot, \mathbf{X}\right) \in C_{b}^{1, 2}\left(\left[0, T\right] \times \mathbb{R}^{d}, \mathbb{R}^{k \times e}\right)$;
  		\item $\left(G\left(\mathbf{X}\right), G^{\prime}\left(\mathbf{X}\right)\right), \in \mathscr{D}_{X}^{\left(p, p\right)\text{-}var}\left(\left[0, T\right], \mathcal{L}\left(\mathbb{R}^{k}, \mathbb{R}^{k \times e}\right)\right)$;
    	\item $\left(H\left(\mathbf{X}\right), H^{\prime}\left(\mathbf{X}\right)\right) \in \mathscr{D}_{X}^{\left(p, p\right)\text{-}var}\left(\left[0, T\right], \mathbb{R}^{k \times e}\right)$;
     	\item for any sequence of smooth paths 
		$X^{n} \in C^{\infty}\left(\left[0, T\right], \mathbb{R}^{e}\right), n=1, 2, \cdots$ 
		such that 
		\begin{equation*}
			\lim_{n \rightarrow \infty} \rho_{p\text{-}var}\left(\mathbf{X}^{n}, \mathbf{X}\right)=0, 
		\end{equation*}
		we have 
		\begin{align*}
			&\left|b_{\cdot}\left(\cdot, \mathbf{X}^{n}\right)-b_{\cdot}\left(\cdot, \mathbf{X}\right)\right|_{\infty}+\left|\sigma_{\cdot}\left(\cdot, \mathbf{X}^{n}\right)-\sigma_{\cdot}\left(\cdot, \mathbf{X}\right)\right|_{\infty}+\left|l\left(\cdot, \mathbf{X}^{n}\right)-l\left(\cdot, \mathbf{X}\right)\right|_{\infty} \notag\\
			&\quad+\left|f_{\cdot}\left(\cdot, \cdot, \cdot, \mathbf{X}^{n}\right)-f_{\cdot}\left(\cdot, \cdot, \cdot, \mathbf{X}\right)\right|_{\infty}+\left|h_{\cdot}\left(\cdot, \mathbf{X}^{n}\right)-h_{\cdot}\left(\cdot, \mathbf{X}\right)\right|_{C_{b}^{1, 2}} \notag\\
			&\quad+\left|G\left(\mathbf{X}^{n}\right), G^{\prime}\left(\mathbf{X}^{n}\right); G\left(\mathbf{X}\right), G^{\prime}\left(\mathbf{X}\right)\right|_{\mathscr{D}_{X^{n}, X}^{\left(p, p\right)\text{-}var}}^{\left(1\right)} \notag\\
			&\quad+\left|H\left(\mathbf{X}^{n}\right), H^{\prime}\left(\mathbf{X}^{n}\right); H\left(\mathbf{X}\right), H^{\prime}\left(\mathbf{X}\right)\right|_{\mathscr{D}_{X^{n}, X}^{\left(p, p\right)\text{-}var}}^{\left(1\right)} \rightarrow 0, \quad as\ n \rightarrow \infty.
		\end{align*}
	\end{enumerate}
\end{assumption}
\subsection{Feynman--Kac formula} \label{S6.1}
We now give the following result on solutions to system of rough PDEs. 
\begin{theorem} \label{rpdesolueu}
	Under Assumption \ref{rpdeassu}, the system of rough PDEs \eqref{rpde} has a solution $u$, 
	and we have the stochastic representation 
	\begin{equation*}
		u\left(t, x, \mathbf{X}\right)=Y_{t}^{t, x, \mathbf{X}}, \quad \forall \left(t, x, \mathbf{X}\right) \in \left[0, T\right] \times \mathbb{R}^{d} \times \mathscr{C}_{g}^{0, p\text{-}var}, 
	\end{equation*}
	where $\left(Y^{s, x, \mathbf{X}}, Z^{s, x, \mathbf{X}}\right)$ 
	is the solution to rough BSDE \eqref{fbsde2}. 
	Moreover, if additionally coefficients $b$ and $\sigma$ are time-invariant, 
	then $u$ is the unique solution to \eqref{rpde}. 
\end{theorem}
\begin{proof}
	For any $\left(s, x, \mathbf{X}\right) \in \left[0, T\right] \times \mathbb{R}^{d} \times \mathscr{C}_{g}^{0, p\text{-}var}$,  
	by the classical SDE theory (e.g. \cite{BMaSC}*{Theorem 5.2.9}), 
	the SDE \eqref{fbsde1} has a unique (strong) solution $S^{s, x, \mathbf{X}}$. 
	Applying Itô's formula, we have 
	\begin{align*}
		&h_{t}\left(S_{t}^{s, x, \mathbf{X}}, \mathbf{X}\right)\\
		&\quad=h_{s}\left(x, \mathbf{X}\right)+\int_{s}^{t} \left(\partial_{t}h_{r}\left(S_{r}^{s, x, \mathbf{X}}, \mathbf{X}\right)+Dh_{r}\left(S_{r}^{s, x, \mathbf{X}}, \mathbf{X}\right)b_{r}\left(S_{r}^{s, x, \mathbf{X}}, \mathbf{X}\right)\right)dr\\
		&\quad \quad+\frac{1}{2}\int_{s}^{t}D^{2}h_{r}\left(S_{r}^{s, x, \mathbf{X}}, \mathbf{X}\right)\sigma_{r}^{2}\left(S_{r}^{s, x, \mathbf{X}}, \mathbf{X}\right)d r\\
		&\quad \quad+\int_{s}^{t} Dh_{r}\left(S_{r}^{s, x, \mathbf{X}}, \mathbf{X}\right)\sigma_{r}\left(S_{r}^{s, x, \mathbf{X}}, \mathbf{X}\right) d W_{r}, \quad \forall t \in \left[s, T\right].
	\end{align*}
	Then by Lemmas \ref{L2scrp} and \ref{mtg}, 
	$\left(h\left(S^{s, x, \mathbf{X}}, \mathbf{X}\right), 0\right) \in \mathbf{D}_{X}^{\left(p, p\right)\text{-}var}L_{2}$. 
	By Theorem \ref{solueu}, the rough BSDE \eqref{fbsde2} has a unique solution 
	$\left(Y^{s, x, \mathbf{X}}, Z^{s, x, \mathbf{X}}\right)$. Define 
	\begin{equation*}
		u\left(t, x, \mathbf{X}\right):=Y_{t}^{t, x, \mathbf{X}}, \quad \forall \left(t, x, \mathbf{X}\right) \in \left[0, T\right] \times \mathbb{R}^{d} \times \mathscr{C}_{g}^{0, p\text{-}var}.
	\end{equation*}
	\indent
	If $\mathbf{X}$ is the canonical rough lift of a smooth path 
	$X \in C^{\infty}\left(\left[0, T\right], \mathbb{R}^{e}\right)$, 
	by \cite{FK1}*{Theorem 2.2} $u\left(\cdot, \cdot, \mathbf{X}\right)$ 
	is a viscosity solution to \eqref{pde} in $BUC$. 
	For general $\mathbf{X}$, take a sequence of smooth paths 
	$X^{n} \in C^{\infty}\left(\left[0, T\right], \mathbb{R}^{e}\right)$, 
	$n=1, 2, \cdots$ such that 
	\begin{equation*}
		\lim_{n \rightarrow \infty} \rho_{p\text{-}var}\left(\mathbf{X}^{n}, \mathbf{X}\right)=0. 
	\end{equation*}
	For $n=1, 2, \cdots$ put 
	$\Delta_{n}S^{s, x}:=S^{s, x, \mathbf{X}^{n}}-S^{s, x, \mathbf{X}}$, 
	$\Delta_{n}b:=b_{\cdot}\left(\cdot, \mathbf{X}^{n}\right)-b_{\cdot}\left(\cdot, \mathbf{X}\right)$ 
	and similarly for other functions. 
	For any fixed $\left(s, x\right) \in \left[0, T\right] \times \mathbb{R}^{d}$ 
	and $n$, we have 
	\begin{align*}
		&\mathbb{E}\sup_{r \in \left[s, T\right]}\left|\Delta_{n}S_{r}^{s, x}\right|^{2}\\
		&\quad \lesssim \mathbb{E}\left[\left(\int_{s}^{T}\left|\left(\Delta_{n}b\right)_{r}\left(S_{r}^{s, x, \mathbf{X}^{n}}\right)\right|dr\right)^{2}+\int_{s}^{T}\left|\left(\Delta_{n}\sigma\right)_{r}\left(S_{r}^{s, x, \mathbf{X}^{n}}\right)\right|^{2}dr\right]\\
		&\quad \lesssim_{T} \left|\Delta_{n}b\right|_{\infty}^{2}+\left|\Delta_{n}\sigma \right|_{\infty}^{2}. 
	\end{align*}
	Then we have 
	\begin{align*}
		\left\|\Delta_{n}\left[l\left(S_{T}^{s, x, \mathbf{X}}, \mathbf{X}\right)\right]\right\|_{2} &\leq \left\|\left(\Delta_{n}l\right)\left(S_{T}^{s, x, \mathbf{X}^{n}}\right)\right\|_{2}+\left\|l\left(S_{T}^{s, x, \mathbf{X}^{n}}, \mathbf{X}\right)-l\left(S_{T}^{s, x, \mathbf{X}}, \mathbf{X}\right)\right\|_{2}\\
		&\lesssim \left|\Delta_{n}l\right|_{\infty}+\left(\mathbb{E}\sup_{r \in \left[s, T\right]}\left|\Delta_{n}S_{r}^{s, x}\right|^{2}\right)^{\frac{1}{2}}\\
		&\lesssim_{T} \left|\Delta_{n}b\right|_{\infty}+\left|\Delta_{n}\sigma \right|_{\infty}+\left|\Delta_{n}l\right|_{\infty}.
	\end{align*}
	Simiarily, we have 
	\begin{align*}
		&\left\|f\left(S^{s, x, \mathbf{X}^{n}}, Y^{s, x, \mathbf{X}^{n}}, Z^{s, x, \mathbf{X}^{n}}, \mathbf{X}^{n}\right)-f\left(S^{s, x, \mathbf{X}}, Y^{s, x, \mathbf{X}^{n}}, Z^{s, x, \mathbf{X}^{n}}, \mathbf{X}\right)\right\|_{\mathbb{L}^{2}}\\
		&\quad \lesssim_{T} \left|\Delta_{n}b\right|_{\infty}+\left|\Delta_{n}\sigma \right|_{\infty}+\left|\Delta_{n}f\right|_{\infty}
	\end{align*}
	and 
	\begin{equation*}
		\left\|\Delta_{n}\left[\partial_{t}h\left(S^{s, x, \mathbf{X}}, \mathbf{X}\right)\right]\right\|_{\mathbb{L}^{2}} \lesssim_{T} \left|\Delta_{n}b\right|_{\infty}+\left|\Delta_{n}\sigma \right|_{\infty}+\left|\Delta_{n}h\right|_{C_{b}^{1, 2}}.
	\end{equation*}
	Since 
	\begin{align*}
		\Delta_{n}\left[Dh\left(S^{s, x, \mathbf{X}}, \mathbf{X}\right)b\left(S^{s, x, \mathbf{X}}, \mathbf{X}\right)\right]&=\Delta_{n}\left[Dh\left(S^{s, x, \mathbf{X}}, \mathbf{X}\right)\right]b\left(S^{s, x, \mathbf{X}^{n}}, \mathbf{X}^{n}\right)\\
		&\quad+Dh\left(S^{s, x, \mathbf{X}}, \mathbf{X}\right)\Delta_{n}\left[b\left(S^{s, x, \mathbf{X}}, \mathbf{X}\right)\right],
	\end{align*}
	we have 
	\begin{align*}
		&\left\|\Delta_{n}\left[Dh\left(S^{s, x, \mathbf{X}}, \mathbf{X}\right)b\left(S^{s, x, \mathbf{X}}, \mathbf{X}\right)\right]\right\|_{\mathbb{L}^{2}}\\
		&\quad \leq \left\|\Delta_{n}\left[Dh\left(S^{s, x, \mathbf{X}}, \mathbf{X}\right)\right]\right\|_{\mathbb{L}^{2}}\left|b_{\cdot}\left(\cdot, \mathbf{X}^{n}\right)\right|_{\infty}+\left|Dh_{\cdot}\left(\cdot, \mathbf{X}\right)\right|_{\infty}\left\|\Delta_{n}\left[b\left(S^{s, x, \mathbf{X}}, \mathbf{X}\right)\right]\right\|_{\mathbb{L}^{2}}\\
		&\quad \lesssim_{T} \left(\sup_{n}\left|\sigma_{\cdot}\left(\cdot, \mathbf{X}^{n}\right)\right|_{\infty}+\left|h_{\cdot}\left(\cdot, \mathbf{X}\right)\right|_{C_{b}^{1, 2}}\right)\left(\left|\Delta_{n}b\right|_{\infty}+\left|\Delta_{n}\sigma \right|_{\infty}+\left|\Delta_{n}h\right|_{C_{b}^{1, 2}}\right).
	\end{align*}
	Simiarily, we have 
	\begin{align*}
		&\left\|\int_{s}^{T} \Delta_{n}\left[Dh_{r}\left(S_{r}^{s, x, \mathbf{X}}, \mathbf{X}\right)\sigma_{r}\left(S_{r}^{s, x, \mathbf{X}}, \mathbf{X}\right)\right] d W_{r}\right\|_{2}\\
		&\quad=\left\|\Delta_{n}\left[Dh\left(S^{s, x, \mathbf{X}}, \mathbf{X}\right)\sigma\left(S^{s, x, \mathbf{X}}, \mathbf{X}\right)\right]\right\|_{\mathbb{L}^{2}}\\
		&\quad \lesssim_{T} \left(\sup_{n}\left|\sigma_{\cdot}\left(\cdot, \mathbf{X}^{n}\right)\right|_{\infty}+\left|h_{\cdot}\left(\cdot, \mathbf{X}\right)\right|_{C_{b}^{1, 2}}\right)\left(\left|\Delta_{n}b\right|_{\infty}+\left|\Delta_{n}\sigma \right|_{\infty}+\left|\Delta_{n}h\right|_{C_{b}^{1, 2}}\right)
	\end{align*}
	and 
	\begin{align*}
		&\left\|\Delta_{n}\left[D^{2}h\left(S^{s, x, \mathbf{X}}, \mathbf{X}\right)\sigma^{2}\left(S^{s, x, \mathbf{X}}, \mathbf{X}\right)\right]\right\|_{\mathbb{L}^{2}}\\
		&\quad \lesssim_{T} \left(\sup_{n}\left|\sigma_{\cdot}\left(\cdot, \mathbf{X}^{n}\right)\right|_{\infty}^{2}+\left|h_{\cdot}\left(\cdot, \mathbf{X}\right)\right|_{C_{b}^{1, 2}}^{2}\right)\left(\left|\Delta_{n}b\right|_{\infty}+\left|\Delta_{n}\sigma \right|_{\infty}+\left|\Delta_{n}h\right|_{C_{b}^{1, 2}}\right).
	\end{align*}
	Note that 
	\begin{align*}
		\Delta_{n}\left[h\left(S^{s, x, \mathbf{X}}, \mathbf{X}\right)\right]&=\Delta_{n}h_{s}\left(x\right)+\int_{s}^{\cdot} \Delta_{n}\left[\partial_{t}h_{r}\left(S_{r}^{s, x, \mathbf{X}}, \mathbf{X}\right)\right] dr\\
		&\quad+\int_{s}^{\cdot} \Delta_{n}\left[Dh_{r}\left(S_{r}^{s, x, \mathbf{X}}, \mathbf{X}\right)b_{r}\left(S_{r}^{s, x, \mathbf{X}}, \mathbf{X}\right)\right]dr\\
		&\quad+\frac{1}{2}\int_{s}^{\cdot}\Delta_{n}\left[D^{2}h_{r}\left(S_{r}^{s, x, \mathbf{X}}, \mathbf{X}\right)\sigma_{r}^{2}\left(S_{r}^{s, x, \mathbf{X}}, \mathbf{X}\right)\right]d r\\
		&\quad+\int_{s}^{\cdot} \Delta_{n}\left[Dh_{r}\left(S_{r}^{s, x, \mathbf{X}}, \mathbf{X}\right)\sigma_{r}\left(S_{r}^{s, x, \mathbf{X}}, \mathbf{X}\right)\right] d W_{r}.
	\end{align*}
	Combining the above inequalities, we have 
	\begin{align*}
		&\left\|h\left(S^{s, x, \mathbf{X}^{n}}, \mathbf{X}^{n}\right), 0; h\left(S^{s, x, \mathbf{X}}, \mathbf{X}\right), 0\right\|_{\mathbf{D}_{X^{n}, X}^{\left(p, p\right)\text{-}var}L_{2}}^{\left(1\right)}\\
		&\quad \lesssim_{T} \left(1+\sup_{n}\left|b_{\cdot}\left(\cdot, \mathbf{X}^{n}\right)\right|_{\infty}+\sup_{n}\left|\sigma_{\cdot}\left(\cdot, \mathbf{X}^{n}\right)\right|_{\infty}^{2}+\left|h_{\cdot}\left(\cdot, \mathbf{X}\right)\right|_{C_{b}^{1, 2}}^{2}\right)\\
		&\quad \quad \times \left(\left|\Delta_{n}b\right|_{\infty}+\left|\Delta_{n}\sigma \right|_{\infty}+\left|\Delta_{n}h\right|_{C_{b}^{1, 2}}\right).
	\end{align*}
	Analogous to the proof of Theorem \ref{solumc}, we have 
	\begin{equation*}
		\lim_{n \rightarrow \infty}\sup_{\left(s, x\right) \in \left[0, T\right] \times \mathbb{R}^{d}}\sup_{r \in \left[s, T\right]}\left\|\Delta_{n}Y_{r}^{s, x}\right\|_{2}=0,
	\end{equation*}
	and thus $u\left(\cdot, \cdot, \mathbf{X}^{n}\right)$ converges uniformly to 
	$u\left(\cdot, \cdot, \mathbf{X}\right)$, as $n \rightarrow \infty$. 
	Therefore, $u$ is a solution to \eqref{rpde}.\\
	\indent
	We now show the uniqueness for time-invariant coefficients $b$ and $\sigma$. 
	Assume that $\bar{u}$ is another solution. 
	For any smooth path $X \in C^{\infty}\left(\left[0, T\right], \mathbb{R}^{e}\right)$ 
	and its canonical rough lift $\mathbf{X}$, by \cite{FK2}*{Theorem 5.1} 
	the system of PDEs \eqref{pde} has a unique viscosity solution in $BUC$. 
	Hence, $\bar{u}\left(\cdot, \cdot, \mathbf{X}\right)=u\left(\cdot, \cdot, \mathbf{X}\right)$. 
	Since $\mathbf{X} \mapsto u\left(\cdot, \cdot, \mathbf{X}\right)$ 
	and $\mathbf{X} \mapsto \bar{u}\left(\cdot, \cdot, \mathbf{X}\right)$ are continuous 
	from $\mathscr{C}_{g}^{0, p\text{-}var}$ to $BUC$, we have $\bar{u}=u$. 
\end{proof}
\begin{remark}
	Our stochastic representation result is new even in the one-dimensional case, 
	for our rough drift coefficient can be time-varying rather than time-invariant 
	as in Diehl and Friz \cite{RBSDE}. 
\end{remark}
\begin{remark}
	Our solution to system of rough PDEs \eqref{rpde} is defined as a continuous extension of 
	the viscosity solution map for systems of PDEs. 
	Similar definitions can be found in \cite{RPDE2}*{Theorem 1} and \cite{RBSDE}*{Theorem 12}. 
	Besides, there are several ways to define viscosity solutions to (one-dimensional) rough PDEs. 
	In Gubinelli et. al \cite{RPDEVS2}*{Definitions 3.6 and 3.8}, 
	they define viscosity and jet-viscosity solutions to fully nonlinear rough PDEs 
	\begin{equation} \label{rpde2}
		\left\{
		\begin{aligned}
			&du\left(t, x\right)=F_{t}\left(x, u, Du, D^{2}u\right)dt+g_{t}\left(x, u, Du\right)d\mathbf{X}_{t}, \quad \left(t, x\right) \in \left(0, T\right] \times \mathbb{R}^{d};\\
			&u\left(0, x\right)=u_{0}\left(x\right), \quad x \in \mathbb{R}^{d},
		\end{aligned}
		\right.
	\end{equation}
	like defining those to PDEs through text functions and semi-jets, respectively. 
	In the particular case that $d=e=1$, $\mathbb{X}=\frac{1}{2}\left(\delta X\right)^{2}$ 
	and $\mathbf{X}$ is truly rough, Buckdahn et al. \cite{RPDEVS1}*{Definition 5.1} 
	define viscosity solutions to rough PDEs \eqref{rpde2} via the method of characteristics, 
	and show that their notion of viscosity solutions is equivalent to the alternative 
	using semi-jets (see \cite{RPDEVS1}*{Proposition 5.3}). 
	However, we may not show that our notion of solutions is consistent with theirs 
	due to the lack of general results for the continuity of those viscosity solutions 
	with respect to driving rough paths.
\end{remark}
\begin{remark}
	We may also solve equation \eqref{rpde} by finding solutions $u$ 
	as continuous maps from $\mathscr{C}_{g}^{0, p\text{-}var}$ 
	to the continuous function space $C\left(\left[0, T\right] \times \mathbb{R}^{d}, \mathbb{R}^{k}\right)$ 
	with pointwise convergence topology or locally uniform topology. 
	In this case, we may appropriately weaken the boundedness and convergence 
	conditions in Assumption \ref{rpdeassu}. 
	We do not pursue the details of the extension here. 
\end{remark}
\subsection{Connection to SPDEs}
We finish this section with the connection between system of rough PDEs \eqref{rpde} with corresponding SPDE. 
Let $\left(\tilde{\Omega}, \tilde{\mathcal{F}}, \tilde{\mathbb{P}}\right)$ be another complete probability space 
carrying an $e$-dimensional Brownian motion $B$, define $\mathbf{B}$ as the Stratonovich lift of $B$, i.e. 
\begin{equation*}
	\mathbb{B}_{s, t}:=\int_{s}^{t} \delta B_{s, r} \otimes \circ d B_{r}, \quad \forall \left(s, t\right) \in \Delta, 
\end{equation*}
and $\mathbf{B}:=\left(B, \mathbb{B}\right)$, where $\circ$ denotes the Stratonovich integration. 
Then $\mathbf{B}\left(\tilde{\omega}\right) \in \mathscr{C}_{g}^{0, p\text{-}var}$ for any $p \in \left(2, 3\right)$ 
and a.s. $\tilde{\omega} \in \tilde{\Omega}$. 
By \cite{RPB1}*{Proposition 5.12}, via the transformation 
\begin{equation*}
	\tilde{u}\left(t, x, \mathbf{X}\right):=u\left(T-t, x, \overleftarrow{\mathbf{X}}\right), \quad \forall \left(t, x, \mathbf{X}\right) \in \left[0, T\right] \times \mathbb{R}^{d} \times \mathscr{C}_{g}^{0, p\text{-}var},
\end{equation*}
where 
	\begin{equation*}
		\overleftarrow{X}_{\cdot}:=X_{T-\cdot}, \quad \overleftarrow{\mathbb{X}}_{\cdot, \cdot}:=\mathbb{X}_{T-\cdot, T-\cdot}, \quad \overleftarrow{\mathbf{X}}:=\left(\overleftarrow{X}, \overleftarrow{\mathbb{X}}\right),
	\end{equation*}
we can transform the system of rough PDEs \eqref{rpde} into forward form 
\begin{equation} \label{frpde}
	\left\{
	\begin{aligned}
		&d\tilde{u}\left(t, x, \mathbf{X}\right)=\frac{1}{2}D^{2}\tilde{u}\left(t, x, \mathbf{X}\right)\tilde{\sigma}_{t}^{2}\left(x, \mathbf{X}\right)dt+D\tilde{u}\left(t, x, \mathbf{X}\right)\tilde{b}_{t}\left(x, \mathbf{X}\right)dt\\
		&\quad+\tilde{f}_{t}\left(x, \tilde{u}\left(t, x, \mathbf{X}\right), D\tilde{u}\left(t, x, \mathbf{X}\right)\tilde{\sigma}_{t}\left(x, \mathbf{X}\right), \mathbf{X}\right)dt\\
		&\quad+\left[\tilde{G}_{t}\left(\mathbf{X}\right)\tilde{u}\left(t, x, \mathbf{X}\right)+\tilde{h}_{t}\left(x, \mathbf{X}\right)+\tilde{H}_{t}\left(\mathbf{X}\right)\right]d\mathbf{X}_{t},\\
		&\quad \left(t, x, \mathbf{X}\right) \in \left[0, T\right) \times \mathbb{R}^{d} \times \mathscr{C}_{g}^{0, p\text{-}var};\\
		&\tilde{u}\left(0, x, \mathbf{X}\right)=\tilde{l}\left(x, \mathbf{X}\right), \quad \left(x, \mathbf{X}\right) \in \mathbb{R}^{d} \times \mathscr{C}_{g}^{0, p\text{-}var},
	\end{aligned}
	\right.
\end{equation}
where 
\begin{equation*}
	\tilde{\varphi}_{t}\left(x, \mathbf{X}\right):=\varphi_{T-t}\left(x, \overleftarrow{\mathbf{X}}\right), \quad \forall \left(t, x, \mathbf{X}\right) \in \left[0, T\right] \times \mathbb{R}^{d} \times \mathscr{C}_{g}^{0, p\text{-}var},
\end{equation*}
for $\varphi=b, \sigma, l, f, h, G, G^{\prime}, H$ and $H^{\prime}$ (for simplicity we omit other variables). 
If 
\begin{equation} \label{pathd}
	\tilde{\varphi}_{t}\left(x, \mathbf{X}\right)=\tilde{\varphi}_{t}\left(x, X_{\cdot \wedge t}\right), \quad \forall \left(t, x, \mathbf{X}\right) \in \left[0, T\right] \times \mathbb{R}^{d} \times \mathscr{C}_{g}^{0, p\text{-}var}.
\end{equation}
Then the system of rough PDEs \eqref{frpde} is naturally connected to SPDE 
\begin{equation} \label{spde}
	\left\{
	\begin{aligned}
		&d\tilde{v}\left(t, x\right)=\frac{1}{2}D^{2}\tilde{v}\left(t, x\right)\tilde{\sigma}_{t}^{2}\left(x, B_{\cdot \wedge t}\right)dt+D\tilde{v}\left(t, x\right)\tilde{b}_{t}\left(x, B_{\cdot \wedge t}\right)dt\\
		&\quad+\tilde{f}_{t}\left(x, \tilde{v}\left(t, x\right), D\tilde{v}\left(t, x\right)\tilde{\sigma}_{t}\left(x, B_{\cdot \wedge t}\right), B_{\cdot \wedge t}\right)dt\\
		&\quad+\left[\tilde{G}_{t}\left(B_{\cdot \wedge t}\right)\tilde{v}\left(t, x\right)+\tilde{h}_{t}\left(x, B_{\cdot \wedge t}\right)+\tilde{H}_{t}\left(B_{\cdot \wedge t}\right)\right] \circ dB_{t}, \quad \left(t, x\right) \in \left(0, T\right] \times \mathbb{R}^{d};\\
		&\tilde{v}\left(0, x\right)=\tilde{l}\left(x, B_{0}\right), \quad x \in \mathbb{R}^{d}.
	\end{aligned}
	\right.
\end{equation}
In view of the continuity of map $\mathbf{X} \mapsto \tilde{u}\left(\cdot, \cdot, \mathbf{X}\right)$, 
by \cite{RPB1}*{Proposition 3.6} we have the following result. 
\begin{proposition} \label{spdesolueu}
	Let Assumption \ref{rpdeassu} and the condition \eqref{pathd} hold, 
	$\tilde{u}$ be a solution to system of rough PDEs \eqref{frpde}, and 
	$B^{n}, n=0, 1, \cdots$,  be the dyadic piecewise linear approximation to $B$, i.e. 
	\begin{equation*}
		B_{t}^{n}:=B_{i2^{-n}T}+\left(\frac{2^{n}t}{T}-i\right)\delta B_{i2^{-n}T, \left(i+1\right)2^{-n}T},
	\end{equation*}
	for every $t \in \left[i2^{-n}T, \left(i+1\right)2^{-n}T\right]$ and $i=0, \cdots, 2^{n}-1$. 
	Define 
	\begin{equation*} 
		\tilde{v}\left(\cdot, \cdot, \tilde{\omega}\right):=\tilde{u}\left(\cdot, \cdot, \mathbf{B}\left(\tilde{\omega}\right)\right), \quad \tilde{v}^{n}\left(\cdot, \cdot, \tilde{\omega}\right):=\tilde{u}\left(\cdot, \cdot, \mathbf{B}^{n}\left(\tilde{\omega}\right)\right), \quad a.s.\ \tilde{\omega} \in \tilde{\Omega},
	\end{equation*}
	where 
	\begin{equation*}
		\mathbb{B}_{s, t}^{n}:=\int_{s}^{t} \delta B_{s, r}^{n} \otimes \dot{B}_{r}^{n} d r, \quad \forall \left(s, t\right) \in \Delta, 
	\end{equation*}
	and $\mathbf{B}^{n}:=\left(B^{n}, \mathbb{B}^{n}\right)$. 
	Then $\tilde{v}^{n}$ is a viscosity solution to the following (random) system of PDEs 
	\begin{equation*}
		\left\{
		\begin{aligned}
			&d\tilde{v}^{n}\left(t, x\right)=\frac{1}{2}D^{2}\tilde{v}^{n}\left(t, x\right)\tilde{\sigma}_{t}^{2}\left(x, B_{\cdot \wedge t}^{n}\right)dt+D\tilde{v}^{n}\left(t, x\right)\tilde{b}_{t}\left(x, B_{\cdot \wedge t}^{n}\right)dt\\
			&\quad+\tilde{f}_{t}\left(x, \tilde{v}^{n}\left(t, x\right), D\tilde{v}^{n}\left(t, x\right)\tilde{\sigma}_{t}\left(x, B_{\cdot \wedge t}^{n}\right), B_{\cdot \wedge t}^{n}\right)dt\\
			&\quad+\left[\tilde{G}_{t}\left(B_{\cdot \wedge t}^{n}\right)\tilde{v}^{n}\left(t, x\right)+\tilde{h}_{t}\left(x, B_{\cdot \wedge t}^{n}\right)+\tilde{H}_{t}\left(B_{\cdot \wedge t}^{n}\right)\right]\dot{B}_{t}^{n}dt, \quad \left(t, x\right) \in \left(0, T\right] \times \mathbb{R}^{d};\\
			&\tilde{v}^{n}\left(0, x\right)=\tilde{l}\left(x, B_{0}\right), \quad x \in \mathbb{R}^{d},
		\end{aligned}
		\right.
	\end{equation*}
	and we have 
	\begin{equation*}
		\lim_{n \rightarrow \infty}\left|\tilde{v}^{n}\left(\cdot, \cdot, \tilde{\omega}\right)-\tilde{v}\left(\cdot, \cdot, \tilde{\omega}\right)\right|_{\infty}=0, \quad a.s.\ \tilde{\omega} \in \tilde{\Omega}.
	\end{equation*}
\end{proposition}
\begin{remark}
	When all the coefficients in SPDE \eqref{spde} do not depend on $B$, 
	by the Wong--Zakai approximation, $\tilde{v}$ is indeed a solution to \eqref{spde}. 
\end{remark}
\begin{appendix}
\section{Some well-posedness results on BSDEs} \label{SA}
Consider the following BSDE 
\begin{equation}\label{qbsde}
	Y_{t}=\xi+\int_{t}^{T} f_{r}\left(Y_{r}, Z_{r}\right) d r-\int_{t}^{T} Z_{r}\, d W_{r}, \quad t \in \left[0, T\right].
\end{equation}
Here, $\xi$ is an $\mathbb{R}^{k}$-valued random variable, 
$\left(Y, Z\right)$ is the pair of unknown processes and the random vector field 
$f:\Omega \times \left[0, T\right] \times \mathbb{R}^{k} \times \mathbb{R}^{k \times d} \rightarrow \mathbb{R}^{k}$ 
is progressively measurable. 
\subsection{Well-posedness under the small terminal value condition} \label{SA.1}
The following result is based on Tevzadze \cite{QBSDE}*{Proposition 1}. 
\begin{theorem} \label{bsdesolueu}
	Let Assumption \ref{assu2} hold. 
	Then there exist sufficiently small positive real numbers $\varepsilon$ and $R$,
	only depending on $T$ and $L$, such that for 
	$\left\|\xi\right\|_{\infty} \leq \varepsilon$ and 
	$\left\|\int_{0}^{T} \left(\lambda_{r}+\mu_{r}^{2}\right) d r\right\|_{\infty} \leq \varepsilon$, 
	the BSDE \eqref{qbsde} has a unique solution 
	$\left(Y, Z\right) \in \mathbb{S}^{\infty} \times BMO$ satisfying 
	\begin{equation} \label{bsdesolue1}
		\left\|Y\right\|_{\mathbb{S}^{\infty}}+\left\|Z\right\|_{BMO} \leq R. 
	\end{equation}
\end{theorem}
\begin{proof}
	Suppose that $\left\|\xi\right\|_{\infty}\vee\left\|\int_{0}^{T} \left(\lambda_{r}+\mu_{r}^{2}\right) d r\right\|_{\infty} \leq 1$. 
	Define 
	\begin{equation*} 
		\mathcal{B}:=\left\{\left(Y, Z\right) \in \mathbb{S}^{\infty} \times BMO :Y_{T}=\xi, \left\|Y, Z\right\|:=\left\|Y\right\|_{\mathbb{S}^{\infty}}+\left\|Z\right\|_{BMO} \leq R\right\}, 
	\end{equation*}
	where $R \in \left(0, 1\right]$ is waiting to be determined later. 
	For any $\left(Y, Z\right) \in \mathcal{B}$, 
	in view of the properties of BMO martingales (e.g. \cite{BMOm}*{Corollary 2.1}), we have 
	\begin{align*}
		\mathbb{E}\left(\int_{0}^{T} \left|f_{r}\left(Y_{r}, Z_{r}\right)\right|d r\right)^{2} &\leq \mathbb{E}\left(\int_{0}^{T} \left(\lambda_{r}+L\left|Y_{r}\right|^{2}+L\left|Z_{r}\right|^{2}\right) d r\right)^{2}\\
		&\lesssim_{L} \left\|\int_{0}^{T} \lambda_{r} d r\right\|_{\infty}^{2}+T^{2}\left\|Y\right\|_{\mathbb{S}^{\infty}}^{4}+T\left\|Z\right\|_{BMO}^{4}. 
	\end{align*}
	Then linear BSDE 
	\begin{equation*}
		\Psi_{t}^{Y}=\xi+\int_{t}^{T} f_{r}\left(Y_{r}, Z_{r}\right) d r-\int_{t}^{T} \Psi_{r}^{Z} d W_{r}, \quad t \in \left[0, T\right]
	\end{equation*}
	has a unique solution $\left(\Psi^{Y}, \Psi^{Z}\right) \in \mathbb{L}^{2} \times \mathbb{L}^{2}$. 
	Applying Itô's formula, we have 
	\begin{equation*}
		\left|\Psi_{t}^{Y}\right|^{2}=\left|\xi\right|^{2}+2\int_{t}^{T} \left(\Psi_{t}^{Y}\right)^{\top}f_{r}\left(Y_{r}, Z_{r}\right) d r-2\int_{t}^{T} \left(\Psi_{t}^{Y}\right)^{\top}\Psi_{r}^{Z} d W_{r}-\int_{t}^{T} \left|\Psi_{r}^{Z}\right|^{2} d r,
	\end{equation*}
	for every $t \in \left[0, T\right]$. Then 
	\begin{align*}
		&\left|\Psi_{t}^{Y}\right|^{2}+\mathbb{E}_{t}\int_{t}^{T} \left|\Psi_{r}^{Z}\right|^{2} d r\\
		&\quad \leq \left\|\xi\right\|_{\infty}^{2}+2\mathbb{E}_{t}\int_{t}^{T} \left|\Psi_{t}^{Y}\right|\left|f_{r}\left(Y_{r}, Z_{r}\right)\right| d r\\
		&\quad \leq \frac{1}{4}\left\|\Psi^{Y}\right\|_{\mathbb{S}^{\infty}}^{2}+\left\|\xi\right\|_{\infty}^{2}+C\left(\mathbb{E}_{t}\int_{t}^{T} \left(\lambda_{r}+L\left|Y_{r}\right|^{2}+L\left|Z_{r}\right|^{2}\right) d r\right)^{2}\\
		&\quad \leq \frac{1}{4}\left\|\Psi^{Y}\right\|_{\mathbb{S}^{\infty}}^{2}+C_{L}\left(\left\|\xi\right\|_{\infty}^{2}+\left\|\int_{0}^{T} \lambda_{r} d r\right\|_{\infty}^{2}+T^{2}\left\|Y\right\|_{\mathbb{S}^{\infty}}^{4}+\left\|Z\right\|_{BMO}^{4}\right)\\
		&\quad \leq \frac{1}{4}\left\|\Psi^{Y}\right\|_{\mathbb{S}^{\infty}}^{2}+C_{L}\left(\left\|\xi\right\|_{\infty}^{2}+\left\|\int_{0}^{T} \lambda_{r} d r\right\|_{\infty}^{2}+\left(T^{2}+1\right)R^{4}\right).
	\end{align*}
	Hence, we have 
	\begin{equation*}
		\left\|\Psi^{Y}, \Psi^{Z}\right\| \leq c_{1}\left(\left\|\xi\right\|_{\infty}+\left\|\int_{0}^{T} \lambda_{r} d r\right\|_{\infty}+\left(T+1\right)R^{2}\right).
	\end{equation*}
	Here $c_{1} \geq 1$ only depends on $L$, and so does the following constant $c_{2}$. 
	Therefore, we can establish the map 
	\begin{equation*} 
		\Psi: \mathcal{B} \rightarrow \mathbb{S}^{\infty} \times BMO, \quad \left(Y, Z\right) \mapsto \left(\Psi^{Y}, \Psi^{Z}\right). 
	\end{equation*}
	\indent
	For any other $\left(\bar{Y}, \bar{Z}\right) \in \mathcal{B}$, 
	put $\Delta Y:=Y-\bar{Y}$, $\Delta Z:=Z-\bar{Z}$ and similarly for other functions. 
	In a similar way, we have 
	\begin{align*}
		&\left|\Delta \Psi_{t}^{Y}\right|^{2}+\mathbb{E}_{t}\int_{t}^{T} \left|\Delta \Psi_{r}^{Z}\right|^{2} d r\\
		&\leq \frac{1}{4}\left\|\Delta \Psi^{Y}\right\|_{\mathbb{S}^{\infty}}^{2}+C\left(\mathbb{E}_{t}\int_{t}^{T} \left|\Delta f_{r}\left(Y_{r}, Z_{r}\right)\right| d r\right)^{2}\\
		&\leq \frac{1}{4}\left\|\Delta \Psi^{Y}\right\|_{\mathbb{S}^{\infty}}^{2}+C\left(\mathbb{E}_{t}\int_{t}^{T} \left(\lambda_{r}+L\left(\left|Y_{r}\right|^{2}+\left|\bar{Y}_{r}\right|^{2}+\left|Z_{r}\right|^{2}+\left|\bar{Z}_{r}\right|^{2}\right)\right)\left|\Delta Y_{r}\right| d r\right)^{2}\\
		&~+C\left(\mathbb{E}_{t}\int_{t}^{T} \left(\mu_{r}+L\left(\left|Y_{r}\right|+\left|\bar{Y}_{r}\right|+\left|Z_{r}\right|+\left|\bar{Z}_{r}\right|\right)\right)\left|\Delta Z_{r}\right| d r\right)^{2}\\
		&\leq \frac{1}{4}\left\|\Delta \Psi^{Y}\right\|_{\mathbb{S}^{\infty}}^{2}\\
		&~+C_{L}\left(\left\|\int_{0}^{T} \lambda_{r} d r\right\|_{\infty}+T\left\|Y\right\|_{\mathbb{S}^{\infty}}^{2}+T\left\|\bar{Y}\right\|_{\mathbb{S}^{\infty}}^{2}+\left\|Z\right\|_{BMO}^{2}+\left\|\bar{Z}\right\|_{BMO}^{2}\right)^{2}\left\|\Delta Y\right\|_{\mathbb{S}^{\infty}}^{2}\\
		&~+C_{L}\left(\left\|\int_{0}^{T} \mu_{r}^{2} d r\right\|_{\infty}+T\left\|Y\right\|_{\mathbb{S}^{\infty}}^{2}+T\left\|\bar{Y}\right\|_{\mathbb{S}^{\infty}}^{2}+\left\|Z\right\|_{BMO}^{2}+\left\|\bar{Z}\right\|_{BMO}^{2}\right)\left\|\Delta Z\right\|_{BMO}^{2}\\
		&\leq \frac{1}{4}\left\|\Delta \Psi^{Y}\right\|_{\mathbb{S}^{\infty}}^{2}+C_{L}\left(\left\|\int_{0}^{T} \left(\lambda_{r}+\mu_{r}^{2}\right) d r\right\|_{\infty}+\left(T^{2}+1\right)R^{2}\right)\left\|\Delta Y, \Delta Z\right\|^{2}.
	\end{align*}
	Hence, we have 
	\begin{equation*}
		\left\|\Delta \Psi^{Y}, \Delta \Psi^{Z}\right\| \leq c_{2}\left(\left\|\int_{0}^{T} \left(\lambda_{r}+\mu_{r}^{2}\right) d r\right\|_{\infty}^{\frac{1}{2}}+\left(T+1\right)R\right)\left\|\Delta Y, \Delta Z\right\|.
	\end{equation*}
	\indent
	Take 
	\begin{equation*}
		\varepsilon:=\frac{1}{16\left(c_{1} \vee c_{2}\right)^{2}\left(T+1\right)}, \quad R:=\frac{1}{4\left(c_{1} \vee c_{2}\right)\left(T+1\right)}.
	\end{equation*}
	Then for $\left\|\xi\right\|_{\infty}\vee \left\|\int_{0}^{T} \left(\lambda_{r}+\mu_{r}^{2}\right) d r\right\|_{\infty} \leq \varepsilon$, 
	we have 
	\begin{equation*}
		c_{1}\left(\left\|\xi\right\|_{\infty}+\left\|\int_{0}^{T} \lambda_{r} d r\right\|_{\infty}+\left(T+1\right)R^{2}\right) \leq R
	\end{equation*}
	and 
	\begin{equation*}
		c_{2}\left(\left\|\int_{0}^{T} \left(\lambda_{r}+\mu_{r}^{2}\right) d r\right\|_{\infty}^{\frac{1}{2}}+\left(T+1\right)R\right) \leq \frac{1}{2}.
	\end{equation*}
	Therefore, $\Psi$ is a contraction map in $\mathcal{B}$. 
	Applying the fixed-point theorem, $\Psi$ has a unique fixed point 
	$\left(Y, Z\right)$ in $\mathcal{B}$, which is the unique solution to \eqref{qbsde} 
	satisfying the estimate \eqref{bsdesolue1}. 
\end{proof}
Then we have the following stability result. 
\begin{theorem} \label{bsdesolumc}
	Assume that $\left(\xi, f\right)$ and $\left(\xi^{n}, f^{n}\right)$ for $n=1, 2, \cdots,$ 
	satisfy Assumption \ref{assu2} with the same $L, \lambda$ and $\mu$, 
	and the inequality 
	$$\left\|\xi\right\|_{\infty} \vee \sup_{n}\left\|\xi^{n}\right\|_{\infty} \vee \left\|\int_{0}^{T} \left(\lambda_{r}+\mu_{r}^{2}\right) d r\right\|_{\infty} \leq \varepsilon$$
	holds for the constant $\varepsilon$ in Theorem \ref{bsdesolueu}. 
	Let $\left(Y, Z\right)$ and $\left(Y^{n}, Z^{n}\right), n=1, 2, \cdots,$ 
	be the solution to the corresponding BSDE \eqref{qbsde} satisfying 
	the corresponding estimate \eqref{bsdesolue1}. Assume 
	\begin{equation*} 
		\lim_{n \rightarrow \infty}\left(\left\|\xi^{n}-\xi\right\|_{\infty}+\mathop{\esssup}\limits_{\left(\omega, t\right)}\mathbb{E}_{t}\int_{t}^{T} \left|f_{r}^{n}\left(Y_{r}, Z_{r}\right)-f_{r}\left(Y_{r}, Z_{r}\right)\right| d r\right)=0.
	\end{equation*}
	Then, we have 
	\begin{equation*} 
		\lim_{n \rightarrow \infty}\left(\left\|Y^{n}-Y\right\|_{\mathbb{S}^{\infty}}+\left\|Z^{n}-Z\right\|_{BMO}\right)=0.
	\end{equation*}
\end{theorem}
\begin{proof}
	Put $\Delta_{n}\xi:=\xi^{n}-\xi$, $\Delta_{n}f:=f^{n}-f$ 
	and similarly for other functions. For every fixed $n$, since 
	\begin{equation*}
		\Delta_{n}\left[f\left(Y, Z\right)\right]=\left[f^{n}\left(Y^{n}, Z^{n}\right)-f^{n}\left(Y, Z\right)\right]+\left(\Delta_{n}f\right)\left(Y, Z\right), 
	\end{equation*}
	analogous to the proof of Theorem \ref{bsdesolueu}, we have 
	\begin{align*}
		&\left|\Delta_{n}Y_{t}\right|^{2}+\mathbb{E}_{t}\int_{t}^{T} \left|\Delta_{n}Z_{r}\right|^{2} d r\\
		&\quad \leq \frac{1}{4}\left\|\Delta_{n}Y\right\|_{\mathbb{S}^{\infty}}^{2}+\left\|\Delta_{n}\xi\right\|_{\infty}+C\left(\mathbb{E}_{t}\int_{t}^{T} \left|\Delta_{n}\left[f\left(Y, Z\right)\right]_{r}\right| d r\right)^{2}\\
		&\quad \leq \frac{1}{4}\left\|\Delta_{n}Y\right\|_{\mathbb{S}^{\infty}}^{2}+C\left(\mathbb{E}_{t}\int_{t}^{T} \left|f_{r}^{n}\left(Y_{r}^{n}, Z_{r}^{n}\right)-f_{r}^{n}\left(Y_{r}, Z_{r}\right)\right| d r\right)^{2}\\
		&\quad \quad+\left\|\Delta_{n}\xi\right\|_{\infty}+C\left(\mathbb{E}_{t}\int_{t}^{T} \left|\left(\Delta_{n}f\right)_{r}\left(Y_{r}, Z_{r}\right)\right| d r\right)^{2}\\
		&\quad \leq \frac{1}{4}\left\|\Delta_{n}Y\right\|_{\mathbb{S}^{\infty}}^{2}+C_{L}\left(\left\|\int_{0}^{T} \left(\lambda_{r}+\mu_{r}^{2}\right) d r\right\|_{\infty}+\left(T^{2}+1\right)R^{2}\right)\left\|\Delta_{n}Y, \Delta_{n}Z\right\|^{2}\\
		&\quad \quad+C\left(\left\|\Delta_{n}\xi\right\|_{\infty}+\mathop{\esssup}\limits_{\left(\omega, t\right)}\mathbb{E}_{t}\int_{t}^{T} \left|\left(\Delta_{n}f\right)_{r}\left(Y_{r}, Z_{r}\right)\right| d r\right)^{2}.
	\end{align*}
	Then 
	\begin{align*}
		\left\|\Delta_{n}Y, \Delta_{n}Z\right\| &\leq \sqrt{2}c_{2}\left(\left\|\int_{0}^{T} \left(\lambda_{r}+\mu_{r}^{2}\right) d r\right\|_{\infty}^{\frac{1}{2}}+\left(T+1\right)R\right)\left\|\Delta_{n}Y, \Delta_{n}Z\right\|\\
		&\quad+C\left(\left\|\Delta_{n}\xi\right\|_{\infty}+\mathop{\esssup}\limits_{\left(\omega, t\right)}\mathbb{E}_{t}\int_{t}^{T} \left|\left(\Delta_{n}f\right)_{r}\left(Y_{r}, Z_{r}\right)\right| d r\right)\\
		&\leq \frac{\sqrt{2}}{2}\left\|\Delta_{n}Y, \Delta_{n}Z\right\|\\
		&\quad+C\left(\left\|\Delta_{n}\xi\right\|_{\infty}+\mathop{\esssup}\limits_{\left(\omega, t\right)}\mathbb{E}_{t}\int_{t}^{T} \left|\left(\Delta_{n}f\right)_{r}\left(Y_{r}, Z_{r}\right)\right| d r\right),
	\end{align*}
	and thus 
	\begin{align*}
		&\left\|\Delta_{n}Y, \Delta_{n}Z\right\|\\
		&\quad \lesssim \left\|\Delta_{n}\xi\right\|_{\infty}+\mathop{\esssup}\limits_{\left(\omega, t\right)}\mathbb{E}_{t}\int_{t}^{T} \left|\left(\Delta_{n}f\right)_{r}\left(Y_{r}, Z_{r}\right)\right| d r \rightarrow 0, \quad as\ n \rightarrow \infty.
	\end{align*}
\end{proof}
\subsection{Well-posedness of diagonally quadratic BSDEs} \label{SA.2}
The following result is based on Fan et al. \cite{DQBSDE2}*{Theorem 2.1, p. 110}, which requires (see \cite{DQBSDE2}*{Assumption (H2), p. 109})  that
the generator is locally Lipschitz continuous with respect to $y$ and the Lipschitz coefficient 
has a linear growth in $z$.  Here,  we assume 
\begin{equation*}
	\left|\partial_{y}f_{t}\left(y, z\right)\right| \lesssim 1+\left|y\right|^{2}+\kappa\left|z\right|^{2}, \quad \forall \left(t, y, z\right) \in \left[0, T\right] \times \mathbb{R}^{k} \times \mathbb{R}^{k \times d},
\end{equation*}
for a sufficiently small positive constant $\kappa$.
\begin{theorem} \label{bsdesolueu2}
	Let Assumption \ref{assu3} hold for $\kappa$ sufficiently small 
	(where the upper bound only depends on $\left\|\xi\right\|_{\infty}$). 
	Then there exist $\varepsilon > 0$ and a bounded subset $\mathcal{B}_{\varepsilon}$ of 
	$\mathbb{S}^{\infty}\left(\left[T-\varepsilon, T\right], \Omega; \mathbb{R}^{k}\right) \times BMO\left(\left[T-\varepsilon, T\right], \Omega; \mathbb{R}^{k \times d}\right)$, 
	only depending on $\left\|\xi\right\|_{\infty}$, such that the BSDE \eqref{qbsde} has a unique local solution $\left(Y, Z\right)$ 
	on the time interval $\left[T-\varepsilon, T\right]$ with $\left(Y, Z\right) \in \mathcal{B}_{\varepsilon}$. 
\end{theorem}
\begin{proof}
	For every $z \in \mathbb{R}^{k \times d}, h \in \mathbb{R}^{1 \times d}$ and $i=1, \cdots, k$, 
	denote by $z\left(h; i\right)$ the matrix in $\mathbb{R}^{k \times d}$ whose $i$th row is $h$ and whose $j$th row is $z^{j}$ for $j \neq i$. 
	By \cite{DQBSDE2}*{Lemma A.1}, for any $\left(U, V\right) \in \mathbb{S}^{\infty} \times BMO$, BSDE 
	\begin{equation*}
		Y_{t}^{i}=\xi^{i}+\int_{t}^{T} f_{r}^{i}\left(U_{r}, V_{r}\left(Z_{r}^{i}; i\right)\right) d r-\int_{t}^{T} Z_{r}^{i} d W_{r}, \quad t \in \left[0, T\right]; \quad i=1, \cdots, k
	\end{equation*}
	has a unique solution $\left(Y, Z\right) \in \mathbb{S}^{\infty} \times BMO$. 
	Define the quadratic solution map 
	\begin{equation*}
		\Gamma: \mathbb{S}^{\infty} \times BMO \rightarrow \mathbb{S}^{\infty} \times BMO, \quad \Gamma\left(U, V\right):=\left(Y, Z\right), \quad \forall \left(U, V\right) \in \mathbb{S}^{\infty} \times BMO.
	\end{equation*}
	Analogous to the proof of \cite{DQBSDE2}*{Theorem 2.1}, their exist $\varepsilon_{0}, R_{1}, R_{2} > 0$ only depending on $\left\|\xi\right\|_{\infty}$, 
	such that for every $\varepsilon \in \left(0, \varepsilon_{0}\right]$ and $\left(U, V\right) \in \mathcal{B}_{\varepsilon}$, 
	we have $\Gamma\left(U, V\right) \in \mathcal{B}_{\varepsilon}$, where 
	\begin{align*}
		\mathcal{B}_{\varepsilon}:=&\left\{\left(U, V\right) \in \mathbb{S}^{\infty}\left(\left[T-\varepsilon, T\right], \Omega; \mathbb{R}^{k}\right) \times BMO\left(\left[T-\varepsilon, T\right], \Omega; \mathbb{R}^{k \times d}\right):\right.\\
		&\quad \left.\left\|U\right\|_{\mathbb{S}^{\infty}} \leq R_{1}, \left\|V\right\|_{BMO}^{2} \leq R_{2}\right\}.
	\end{align*}
	For any fixed $\varepsilon \in \left(0, \varepsilon_{0}\right]$ and $\left(U, V\right), \left(\bar{U}, \bar{V}\right) \in \mathcal{B}_{\varepsilon}$, write 
	$\left(Y, Z\right):=\Gamma\left(U, V\right)$ and $\left(\bar{Y}, \bar{Z}\right):=\Gamma\left(\bar{U}, \bar{V}\right)$. 
	Put $\Delta Y:=Y-\bar{Y}$, $\Delta Z:=Z-\bar{Z}$ and similarly for other functions. 
	For every $t \in \left[T-\varepsilon, T\right]$ and $i=1, \cdots, k$, we have 
	\begin{equation*}
		\left|f_{t}^{i}\left(U_{t}, V_{t}\left(Z_{t}^{i}; i\right)\right)-f_{t}^{i}\left(U_{t}, V_{t}\left(\bar{Z}_{t}^{i}; i\right)\right)\right| \lesssim \left(1+\left|U_{t}\right|+\left|V_{t}\right|+\left|Z_{t}\right|\right)\left|\Delta Z_{t}^{i}\right|.
	\end{equation*}
	Analogous to the proof of \cite{DQBSDE2}*{Theorem 2.1}, 
	for every $i=1, \cdots, k$, there exist an equivalent probability measure $\mathbb{P}^{i}$ 
	and a process $\bar{W}\left(i\right)$, such that $\bar{W}\left(i\right)$ is a standard Brownian motion under $\mathbb{P}^{i}$, 
	\begin{align*}
		\Delta Y_{t}^{i}&=\int_{t}^{T} \left(f_{r}^{i}\left(U_{r}, V_{r}\left(\bar{Z}_{r}^{i}; i\right)\right)-f_{r}^{i}\left(\bar{U}_{r}, \bar{V}_{r}\left(\bar{Z}_{r}^{i}; i\right)\right)\right) dr\\
		&\quad-\int_{t}^{T} \Delta Z_{r}^{i} d\bar{W}_{r}\left(i\right), \quad \forall t \in \left[T-\varepsilon, T\right], 
	\end{align*}
	and 
	\begin{equation*}
		\left\|F\right\|_{BMO} \lesssim \left\|F\right\|_{BMO\left(\mathbb{P}^{i}\right)} \lesssim \left\|F\right\|_{BMO}, \quad \forall F \in BMO\left(\left[T-\varepsilon, T\right], \Omega; \mathbb{R}^{k \times d}\right).
	\end{equation*}
	By \cite{BMOm}*{Corollary 2.1}, applying Hölder's inequality, for every $t \in \left[T-\varepsilon, T\right]$ and $i=1, \cdots, k$ we have 
	\begin{align*}
		&\left|\Delta Y_{t}^{i}\right|^{2}+\mathbb{E}_{t}^{i}\int_{t}^{T}\left|\Delta Z_{r}^{i}\right|^{2} dr\\
		&\quad \leq \mathbb{E}_{t}^{i}\left(\int_{t}^{T} \left|f_{r}^{i}\left(U_{r}, V_{r}\left(\bar{Z}_{r}^{i}; i\right)\right)-f_{r}^{i}\left(\bar{U}_{r}, \bar{V}_{r}\left(\bar{Z}_{r}^{i}; i\right)\right)\right| dr\right)^{2}\\
		&\quad \lesssim \mathbb{E}_{t}^{i}\left(\int_{t}^{T} \left(1+\left|U_{r}\right|^{2}+\left|\bar{U}_{r}\right|^{2}+\kappa\left|V_{r}\right|^{2}+\kappa\left|\bar{V}_{r}\right|^{2}\right)\left|\Delta U_{r}\right| dr\right)^{2}\\
		&\quad \quad+\mathbb{E}_{t}^{i}\left(\int_{t}^{T} \left(1+\left|U_{r}\right|+\left|\bar{U}_{r}\right|+\left|V_{r}\right|^{\theta}+\left|\bar{V}_{r}\right|^{\theta}\right)\left|\Delta V_{r}\right| dr\right)^{2}\\
		&\quad \lesssim \left(\varepsilon^{2}\left(1+\left\|U\right\|_{\mathbb{S}^{\infty}}^{4}+\left\|\bar{U}\right\|_{\mathbb{S}^{\infty}}^{4}\right)+\kappa^{2}\mathbb{E}_{t}^{i}\left(\int_{t}^{T} \left(\left|V_{r}\right|^{2}+\left|\bar{V}_{r}\right|^{2}\right) dr\right)^{2}\right)\left\|\Delta U\right\|_{\mathbb{S}^{\infty}}^{2}\\
		&\quad \quad+\left(1+\left\|U\right\|_{\mathbb{S}^{\infty}}^{2}+\left\|\bar{U}\right\|_{\mathbb{S}^{\infty}}^{2}\right)\mathbb{E}_{t}^{i}\left(\int_{t}^{T} \left|\Delta V_{r}\right| dr\right)^{2}\\
		&\quad \quad+\left(\mathbb{E}_{t}^{i}\left(\int_{t}^{T} \left(\left|V_{r}\right|^{2\theta}+\left|\bar{V}_{r}\right|^{2\theta}\right) dr\right)^{2}\right)^{\frac{1}{2}}\left(\mathbb{E}_{t}^{i}\left(\int_{t}^{T} \left|\Delta V_{r}\right|^{2} dr\right)^{2}\right)^{\frac{1}{2}}\\
		&\quad \lesssim \left(\varepsilon^{2}\left(1+\left\|U\right\|_{\mathbb{S}^{\infty}}^{4}+\left\|\bar{U}\right\|_{\mathbb{S}^{\infty}}^{4}\right)+\kappa^{2}\left(\left\|V\right\|_{BMO}^{4}+\left\|\bar{V}\right\|_{BMO}^{4}\right)\right)\left\|\Delta U\right\|_{\mathbb{S}^{\infty}}^{2}\\
		&\quad \quad+\left(\varepsilon\left(1+\left\|U\right\|_{\mathbb{S}^{\infty}}^{2}+\left\|\bar{U}\right\|_{\mathbb{S}^{\infty}}^{2}\right)+\varepsilon^{1-\theta}\left(\left\|V\right\|_{BMO}^{2\theta}+\left\|\bar{V}\right\|_{BMO}^{2\theta}\right)\right)\left\|\Delta V\right\|_{BMO}^{2}\\
		&\quad \lesssim \left(\varepsilon^{1-\theta}+\kappa^{2}\right)\left(1+R_{1}^{4}+R_{2}^{2}\right)\left(\left\|\Delta U\right\|_{\mathbb{S}^{\infty}}^{2}+\left\|\Delta V\right\|_{BMO}^{2}\right).
	\end{align*}
	Therefore, $\Gamma$ is a contraction map in $\mathcal{B}_{\varepsilon}$ when $\varepsilon$ and $\kappa$ are sufficiently small. 
	Applying the fixed-point theorem, we see that $\Gamma$ has a unique fixed point $\left(Y, Z\right)$ in $\mathcal{B}_{\varepsilon}$, 
	which is the unique local solution to \eqref{qbsde} on the time interval $\left[T-\varepsilon, T\right]$ in $\mathcal{B}_{\varepsilon}$. 
\end{proof}
In a similar way, we have the following stability result. 
\begin{theorem} \label{bsdesolumc2}
	Let $\left(\xi, f\right)$ and $\left(\xi^{n}, f^{n}\right), n=1, 2, \cdots$ 
	satisfy Assumption \ref{assu3} with the same $\kappa$ and $\theta$. 
	Assume that 
	$\left\|\xi\right\|_{\infty} \vee \sup_{n}\left\|\xi^{n}\right\| \leq M$ 
	for some positive real number $M$, and $\kappa$ is sufficiently small 
	(where the upper bound only depends on $M$). 
	Let $\left(Y, Z\right), \left(Y^{n}, Z^{n}\right) \in \mathcal{B}_{\varepsilon}, n=1, 2, \cdots$ 
	be the local solution to the corresponding BSDE \eqref{qbsde} on the time interval $\left[T-\varepsilon, T\right]$. 
	Assume 
	\begin{equation*} 
		\lim_{n \rightarrow \infty}\left(\left\|\xi^{n}-\xi\right\|_{\infty}+\mathop{\esssup}\limits_{\left(\omega, t\right)}\mathbb{E}_{t}\int_{t}^{T} \left|f_{r}^{n}\left(Y_{r}, Z_{r}\right)-f_{r}\left(Y_{r}, Z_{r}\right)\right| d r\right)=0.
	\end{equation*}
	Then, we have 
	\begin{equation*} 
		\lim_{n \rightarrow \infty}\left(\left\|Y^{n}-Y\right\|_{\mathbb{S}^{\infty}}+\left\|Z^{n}-Z\right\|_{BMO}\right)=0.
	\end{equation*}
\end{theorem}
\end{appendix}

\bibliographystyle{amsplain}
\bibsection

\end{document}